\documentclass[10pt]{article}

    \makeatletter
    \let\@fnsymbol\@arabic
    \makeatother

\usepackage{fancybox}
\usepackage{shadow}
\usepackage[all]{xy}
\usepackage{url}
\usepackage{rotating}
\usepackage{amsmath}
\usepackage{amssymb}
\usepackage{amsthm}
\usepackage{amscd}
\usepackage{amsfonts}
\usepackage{graphicx}%
\usepackage{fancyhdr}
\usepackage{geometry}
\usepackage{color}
\usepackage{xcolor}
\usepackage{mathrsfs}
\usepackage{amsmath}
\usepackage{enumerate}
\usepackage{pgf,tikz}

\usepackage{hyperref}
\hypersetup{
    colorlinks=true, 
    linktoc=all,     
    linkcolor=black,  
    citecolor=black
}

\usetikzlibrary{arrows}

\newcommand{\N}{\mathbb{N}}

\newcommand{\R}{\mathbb{R}} 

\newcommand{\rotmedgeq}{\mathrel{\text{\rotatebox[origin=c]{-45}{$\Rightarrow$}}​}}

\allowdisplaybreaks

\title{An Ergodic Theorem for Fleming-Viot Models in Random Environments \footnote{Research partially supported by CNPq.}}
\author{Arash Jamshidpey \thanks{IMPA, Estrada Dona Castorina 110, 22460-320, Rio de Janeiro, RJ, Brazil, email: arashj@impa.br}}
\date{2016}

\newtheorem{theorem}{Theorem}[section]
\newtheorem{lemma}[theorem]{Lemma}
\newtheorem{proposition}[theorem]{Proposition}

\newtheorem{remark}[theorem]{Remark}
\newtheorem{definition}{Definition}

\begin{document}

\maketitle

\begin{abstract}
The Fleming-Viot (FV) process is a measure-valued diffusion that models the evolution of type frequencies in a countable population which evolves under resampling (genetic drift), mutation, and selection. In the classic FV model the fitness (strength) of types is given by a measurable function. In this paper, we introduce and study the Fleming-Viot process in random environment (FVRE), when by random environment we mean the fitness of types is a stochastic process with c\`adl\`ag paths. We identify FVRE as the unique solution to a so called quenched martingale problem and derive some of its properties via martingale and duality methods. We develop the duality methods for general time-inhomogeneous and quenched martingale problems. In fact, some important aspects of the duality relations only appears for time-inhomogeneous (and quenched) martingale problems. For example, we see that duals evolve backward in time with respect to the main Markov process whose evolution is forward in time. Using a family of function-valued dual processes for FVRE, we prove that, as the number of individuals $N$ tends to $\infty$, the measure-valued Moran process $\mu_N^{e_N}$ (with fitness process $e_N$) converges weakly in Skorokhod topology of c\`adl\`ag functions to the FVRE process $\mu^e$ (with fitness process $e$), if $e_N \rightarrow e$ a.s. in Skorokhod topology of c\`adl\`ag functions. We also study the long-time behaviour of FVRE process $(\mu_t^e)_{t\geq 0}$ joint with its fitness process $e=(e_t)_{t\geq 0}$ and prove that the joint FV-environment process $(\mu_t^e,e_t)_{t\geq 0}$ is ergodic under the assumption of weak ergodicity of $e$.
\end{abstract}

\tableofcontents

\section{Introduction}

Probabilistic models play a crucial role in population genetics. In particular, for a long time, different popular models in interacting particle systems have been used to model several population dynamics. In fact two important mechanisms of evolution in population dynamics, namely mutation and natural selection, are better to be understood as random time-varying parameters. The dynamics of a population is effected by environmental changes. In fact, the genetic variations exist in the genomes of species and these variations, in turn, are in interaction with environments. Natural selection, as the most important mechanism of evolution, favors the fitter type in an organism. The fitness of different types determines the role of "natural selection" in a population and depends on important environmental parameters. It is a function of environmental changes and other evolutionary mechanisms, i.e. mutation and genetic drift. Subsequently, an important question is the effect of environmental changes on the structure of the population. ``Adaptive processes have taken centre in molecular evolutionary biology. Time dependent fitness functions has opposing effects on adaptation. Rapid fluctuations enhance the stochasticity of the evolutionary process and impede long-term adaptation.\cite{Mustonen-Lassig10}" In other words, living in rapidly varying environments, a population is not able to adapt to the environment. Because of simplicity, the existing probabilistic models in population genetics mainly concern problems in which the natural selection is not time-dependent. This decreases the validity of models and does not allow the study of the interactions between the environment and the population. In other words, they cannot explain the real effect of the environment on adaptation of a population system. In fact, it is both more realistic and also challenging to have a random environment varying in time.

In this paper, we study long time behaviours of some countable probabilistic population dynamics in random environment. For this purpose we make use of the martingale problem and the duality method and we develop a generalization of existing methods in the literature to the case of time-inhomogeneous Markov processes. In particular, in this paper, the duality method is studied for time-inhomogeneous Markov processes and Markov processes in random environments. In the case of their existence, dual processes are powerful tools to prove uniqueness of martingale problems and to understand the long-time behaviour of Markov processes. We apply these methods in order to define the Fleming-Viot process in random environment. In fact this process arises as a weak limit of the so called Moran processes in random environment which are natural generalizations of their counterpart in deterministic environment. Identifying the Fleming-Viot process in random environment as the solution to a quenched martingale problem, we study its long-time behaviour via studying the long-time behaviour of its dual process.

The classic particle Moran process is a basic probabilistic population dynamic which models the evolution of frequency of types (alleles) in a population with $N\in \N$ individuals. Letting the fitness of types be a stochastic process, we generalize this model, and introduce a finite population system in random environment, namely the particle Moran process in random environment (PMRE) with type space $I$, and resampling, mutation, and selection rates $\frac{\gamma}{2},\beta,\frac{\alpha}{N}>0$, respectively. Here, we assume that the type space is a general metric space. However for the results of the paper, we always assume that $I$ is compact. Let $E$ be a family of continuous functions from $I$ to $[0,1]$, endowed with the sup-norm topology. Later, for the results of this paper, we also assume that $E$ is compact. A (bounded) fitness process is an $E$-valued measurable stochastic process $(e_t)_{t\geq 0}$,  with c\`adl\`ag paths, defined on a probability space  $(\Omega,\mathbb{P},\mathcal{F})$. The fitness process has the role of natural selection in the environment, and determines the fitness of types in the population dynamic. Consider a population of $N$ individuals, labelled by $1,...,N$.  The PMRE process is a continuous-time, $I^N$-valued Markov process in which each individual carries a configuration (type) in $I$ and population evolves as a pure jump process when jumps occur at independent Poisson times of resampling (genetic drift), mutation and selection. More precisely, for $i=1,...,N$ and $t\geq 0$, denote by $Y_i^{N,e}(t)$ the type of individual $i$ at time $t$, and let $Y^{N,e}(t):=(Y_1^{N,e}(t),...,Y_N^{N,e}(t))$. The PMRE process $(Y^{N,e}(t))_{t\geq 0}$ evolves as follows. Between every ordered pair of individuals $i,j$ ($i\neq j$), resampling  events occur at rate $\gamma/2 >0$ and upon a resampling event the $j$th individual dies and is replaced by an offspring of the $i$th individual. Also the type of every individual, independently, changes from $x \in I$ to $y \in I$ with mutations at rate $\beta q(x, dy)$ where $\beta>0$, and $q$ is a stochastic kernel on $I$. Every ordered pair of individuals $i,j$ ($i\neq j$) is involved in a possible selective event at rate $\alpha /N$, for $\alpha \geq 0$.  Upon a possible selective event at time $t$,  with probability $e_t( Y_i^{N,e} (t))$, the $j$th individual dies and is replaced by an offspring of the $i$th one, and with probability $1-e_t( Y_i^{N,e} (t))$ no change happens. Note that, always, there exist constants $\beta',\beta''\geq 0$ and probability kernels $q'(dy)$ and $q''(x,dy)$, where
\begin{equation}
\begin{array}{l}
\beta q(x,dy)=\beta'q'(dy)+\beta''q''(x,dy).
\end{array}
\end{equation}
We call $\beta' q'(dy)$ the parent-independent component of the mutation.

Considering the frequency of alleles at each time, it is convenient to project $Y^{N,e}$ onto a purely atomic (with at most $N$ atoms) measure-valued process on $\mathcal{P}^N(I)$, that is the space of all probability measures $m$ on $I$ with at most $N$ atoms such that $Nm(.)$ is a counting measure. More precisely, for any $t\geq 0$, let
\begin{equation}
\mu_N^e(t)=\frac{1}{N}\sum\limits_{i=1}^N \delta_{Y_i^{N,e}(t)}
\end{equation}
where, for $a\in I$, $\delta_a$ is the delta measure on $a$. For some results in this paper, we assume that $e$ is independent of the initial distribution of $\mu_N^e$, the mutation kernel, and Poisson times of jumps (for the dual process). Let $E$ be a compact subset of $\mathcal{C}(I,[0,1])$ equipped with the sup-norm topology. We assume that the fitness process is a measurable stochastic process with sample paths in $D_E[0,\infty)$, the space of c\`adl\`ag functions endowed with the Skorokhod topology. Letting $N\rightarrow \infty$, the Fleming-Viot process in random environment arises as the weak limit of $\mu_N^e$ in $D_{\mathcal{P}(I)}[0,\infty)$, where $\mathcal{P}(I)$ is the space of Borel probability measures on $I$ endowed with the topology of weak convergence. We characterize this process as a solution to a martingale problem in random environment (called quenched martingale problem). The main purpose of this paper is to study the long-time behaviour of Fleming-Viot processes in random environment. In order to do that, we develop the duality method to the case of time-dependent and quenched martingale problems.
 Our goal is to set up the martingale and duality method for measure-valued Moran and Fleming-Viot processes in random environments. We study the convergence and ergodic theorems for these processes. We organize the paper as follows. In the rest of the first section, after introducing some general notations, we set up the time-inhomogeneous martingale problems and bring some criteria for existence and uniqueness of solutions. In subsection 1.3 we introduce the notion of operator-valued stochastic processes and generalize the time-inhomogeneous martingale problem to quenched martingale problems in order to characterize Markov processes in random environments as their solutions. In this section, we also define the joint annealed-environment process, where we consider the evolution of the annealed process together with its associated environment. Section 2 is devoted to martingale characterization of Moran and Fleming-Viot processes in random environments (r.e.). The statement of the main theorems will be given in this section as well. Section 3 develops the duality method in the case of general time-inhomogeneous and quenched martingale problems. Section 4 presents a function-valued dual for the Fleming-Viot process in random environment and studies its long-time behaviour. In section 5, we prove the convergence of infinitesimal generators of Moran processes in random environments to that of the Fleming-Viot process in random environment. The proof of the well-posedness of the quenched Fleming-Viot martingale problem, along with the convergence of the Moran process in r.e. to Fleming-Viot process in r.e., will come in section 6. Section 7 is devoted to the proof of continuity of sample paths of the Fleming-Viot process in r.e.. Finally, in section 8, we prove the ergodic theorem for the Fleming-Viot process in random environment.

\subsection{General notations}
For metric spaces $(S,d_S)$ and $(S',d_{S'})$, we denote by $\mathcal C(S,S')=\mathcal{C}^0(S,S')$ and $\mathcal C^k(S,S')$ (for $k \geq 1$) the space of all continuous,  and $k$ times continuously differentiable (Borel measurable) functions form $S$ to $S'$, respectively. In particular, when $S'$ is the set of real numbers with the standard topology, we replace $\mathcal C(S,S')$ and $\mathcal C^k(S,S')$ by $\mathcal C(S)=\mathcal{C}^0(S)$ and $\mathcal{C}^k(S)$, respectively. Let $\mathcal{B}(S)$, $\bar{\mathcal C}(S)=\bar{\mathcal{C}}^0(S)$, and $\bar{\mathcal C}^k(S)$ (for $k \geq 1$) be the space of all bounded, bounded continuous,  and bounded $k$ times continuously differentiable Borel measurable real-valued functions on $S$, respectively, with norm $\|f\|_\infty=\|f\|_\infty^S:=\sup_{x \in S} |f(x)|$. The topology induced by this norm is called the sup-norm topology. 
We equip the space of all $S$-valued c\`adl\`ag functions, namely the space of all right continuous with left limits $S$-valued functions defined on $\R_+$, with Skorokhod topology, and denote it by $D_S[0,\infty)$. We denote by $\mathscr{B}(S)$ both the Borel $\sigma$-field and the space of all Borel measurable real-valued functions on $S$. Denote by $\mathcal{P}(S)$ the space of all (Borel) probability measures on $S$, equipped with the weak topology, and let "$\Rightarrow$" denote convergence in distribution. Also for $S_n\subset S$, for natural numbers $n$, say a sequence of $S_n$-valued random variables, namely $(Z_n)_{n\in\N}$, converges weakly to an $S$-valued random variable $Z$, if $\iota_n(Z_n)\Rightarrow Z$ as $n\rightarrow \infty$, where $\iota_n:S_n\rightarrow S$ is the natural embedding map. In general, for two topological spaces $S_1$ and $S_2$, by $S_1\times S_2$ we mean the Cartesian product of two spaces equipped with the product topology, and by $\mathcal{P}(S_1\times S_2)$ we mean the space of all Borel probability measures on $S_1\times S_2$. Otherwise, we shall indicate it if we furnish the product space with another topology. Also, denote by $<m,f>$ or $<f,m>$ 
the integration $\int_S f dm$ for $m\in \mathcal{P}(S)$ and $f\in \mathcal{B}(S)$ (or more generally, when $m\in \mathcal{P}(S)$ is given, for all $m$-integrable functions).

Throughout this paper, $S$ is a general Polish space, i.e. a separable completely metrizable topological space, with at least two elements (to avoid triviality), and we assume $(\Omega, \mathbb{P}, \mathcal{F})$ is a probability space, and all random variables and stochastic processes will be defined on this space. Also, we restrict random variables and stochastic processes to take values only on Polish spaces. We denote by $\mathbb{P}\zeta^{-1}$ the law of an $S$-valued random variable $\zeta$ (similarly, a measurable stochastic process $\zeta=(\zeta_t)_{t\geq 0}$). Let $m=\mathbb{P}\zeta^{-1}$. For an $m$-integrable real-valued function $f$ on $S$, the expected value of $f(\zeta)$ is denoted by $\mathbb{E}[f(\zeta)]$, or to emphasise the law of $\zeta$, by $\mathbb{E}_m[f(\zeta)]$. Also, by $\mathbb{E}^x[f(\zeta)]$ ($\mathbb{E}^{p_0}[f(\zeta)]$, respectively), we put emphasis on the initial state $x\in S$ (initial distribution $p_0\in \mathcal{P}(S)$, respectively) of the process $\zeta$.
\subsection{Time-inhomogeneous martingale problem: existence and uniqueness}
We can think of an operator $G$ on a Banach space $L$ as a subset of $L\times L$. This definition allows $G$ to be a multi-valued operator. A linear operator is one that is a linear subspace of $L\times L$. Observe that a linear operator $G$ is single-valued, if the condition $(0,g)\in G$ implies $g=0$. For a single-valued linear operator $G$, the domain of $G$, denoted by $\mathcal{D}(G)$, is the set of elements of $L$ on which $G$ is defined. In other words, $\mathcal{D}(G)=\{f\in L: (f,g)\in G\}$. Also, the range of an operator $G$ is denoted by $\mathcal{R}(G)=\{g\in L: (f,g)\in G\}$. Let $\mathcal{D}$ be a linear subspace of $L$. A time-dependent single-valued linear operator $G$ is a mapping from $\R_+ \times \mathcal{D}$ to $L$ such that $G(t,.):\mathcal{D} \rightarrow L$ is a single-valued linear operator. For simplicity, we set $G_t(.)=G(t,.)$ for any $t\geq 0$. In the sequel, we only deal with time-dependent linear operators for which the domain of $G_t$ is $\mathcal{D}$ for any $t\geq 0$. Therefore we define the domain of $G$ to be $\mathcal{D}(G)=\mathcal{D}$. As discussed above, we can identify the time-dependent generator of an $S$-valued inhomogeneous Markov process with the domain $\mathcal{D}\subset \mathcal{B}(S)$ as a subset of $\mathcal{B}(S) \times \mathcal{B}(S)$ (not necessarily linear). For our purposes in this paper we assume that all the operators are linear and single-valued, therefore their domains are linear subspaces of $\mathcal{B}(S)$. In the sequel we consider the generators of Markov processes as both single-valued linear operators and also linear subspaces of $\mathcal{B}(S) \times \mathcal{B}(S)$ for which containing $(0,g)$ implies $g=0$. Similarly, for a time-dependent infinitesimal generator of a time-inhomogeneous Markov process, $G:\R_+ \times \mathcal{D}\rightarrow \mathcal{B}(S)$, for any $t\geq 0$, we assume that $G_t$ is a single-valued linear operator and also a linear subspace of $\mathcal{B}(S) \times \mathcal{B}(S)$ for which containing $(0,g)$ implies $g=0$.

Here we speak of martingale problems for general Markov processes (the time-inhomogeneous case). A martingale problem is identified by a triple $(G,\mathcal{D},\textbf{P}_0)$, where $\textbf{P}_0 \in \mathcal{P}(S)$, $\mathcal{D} \subset \mathcal{B}(S)$, and $G: \R_+ \times \mathcal{D} \rightarrow \mathcal{B}(S)$ is a time-dependent linear operator with domain $\mathcal{D}$.
\begin{definition}
An $S$-valued measurable stochastic process $\zeta=(\zeta_t)_{t \in \R}$, defined on $(\Omega,\mathbb{P},\mathcal{F})$ is said to be a solution of the martingale problem $(G,\mathcal{D},\textbf{P}_0)$ if for $\zeta$ with sample paths in $D_S[0,\infty)$ and initial distribution $\textbf{P}_0$, for any $f\in \mathcal{D}$
\begin{equation}
f(\zeta(t))-\int\limits_0^t G_s f(\zeta(s)) ds
\end{equation} 
is a $\textbf{P}$-martingale with respect to the canonical filtration, where $\textbf{P}\in \mathcal{P}(D_S[0,\infty))$ is the law of $\zeta$. We also say that $\textbf{P}$ is a solution of $(G,\mathcal{D},\textbf{P}_0)$. The process $\zeta$ or its law $\textbf{P}$ is said to be a general solution if the sample paths of $\zeta$ are not necessarily in $D_S[0,\infty)$, i.e. the support of $\textbf{P}$ is not contained in $D_S[0,\infty)$. We say the martingale problem is well-posed if there is a unique solution (with paths in $D_S[0,\infty)$, general solutions not considered), that is there exists a unique $\textbf{P}\in \mathcal{P}(D_S[0,\infty))$ that solves the martingale problem. It is said to be $\mathcal{C}([0,\infty),S)$-well posed, if it has a unique solution $\textbf{P}$ in $\mathcal P(\mathcal{C}([0,\infty),S))$.
\end{definition}
The following concepts are useful in order to prove the uniqueness of martingale problems. 
\begin{definition}
We say a set of functions $U \subset \bar{C}(S)$ (more generally, $U\subset \mathcal{B}(S)$) separates points if for every $x,y \in S$ with $x\neq y$ there exists a function $f \in U$ for which $f(x)\neq f(y)$. We say $U$ vanishes nowhere if for any $x\in S$ there exists a function $f\in U$ such that $f(x)\neq 0$.
\end{definition}
\begin{definition}\label{Definition-convergence and measure-determining general}
A collection of functions $U \subset \bar{C}(S)$ (more generally, $U\subset \mathcal{B}(S)$) is said to be measure-determining on $\mathcal{M}\subset \mathcal{P}(S)$ if for $\textbf{P},\textbf{P}' \in \mathcal{M}$, assuming
\begin{equation}
\int\limits_S f d\textbf{P}=\int\limits_S f d\textbf{P}'
\end{equation}
for all $f \in U$ implies $\textbf{P}=\textbf{P}'$. We say $U$ is measure-determining, if it is measure-determining on $\mathcal{P}(S)$. Also, we say $U$ is convergence-determining on $\mathcal{M}\subset \mathcal{P}(S)$ if for the sequence of probability measures $(\textbf{P}_n)_{n\in\N}$ and the probability measure $\textbf{P}$ in $\mathcal{M}$
\begin{equation}\label{convergence-determining}
\lim\limits_{n \rightarrow \infty} \int\limits_S f d\textbf{P}_n = \int\limits_S f d \textbf{P} \ \ for \ \ all \ \ f\in U
\end{equation}
implies $\textbf{P}_n \Rightarrow \textbf{P}$. We say $U$ is convergence-determining, if it is convergence-determining on $\mathcal{P}(S)$.
\end{definition}
If $U\subset \bar{\mathcal{C}}(S)$ is convergence-determining then it is measure-determining , but the converse is not true in general. Two concepts are equivalent for compact $S$ (see Lemma 3.4.3 \cite{EK86-book}).

In order to be able to transform some useful properties from the time-independent martingale problems to time-dependent ones, it is convenient to define the space-time process for the $S$-valued stochastic process $\zeta=(\zeta_t)_{t\geq 0}$ by $\zeta_t^*=(\zeta_t,t)$, which is an $S\times \R_+$-valued stochastic process. Consider the particular case when $\zeta$ is an $S$-valued time-inhomogeneous Markov process with time-dependent generator $G:\R_+ \times \mathcal{D}(G) \rightarrow \mathcal{B}(S)$, i.e. $(G_t: \mathcal{D}(G) \rightarrow \mathcal{B}(S))_{t\geq 0}$. Let $(T_{s,r})_{0\leq s \leq t}$ be the time-inhomogeneous semigroup of $\zeta$, corresponding to $G$ defined by  
\begin{equation}
T_{s,r}f(x)=\int\limits_S f(y)p(s,x;r,dy) \ for \ f\in \mathcal{B}(S),
\end{equation}
where $p(s,x;r,dy)$ is the transition probability for $\zeta$. From the definition of space-time process, $\zeta^*$ is also Markov with the transition probability
\begin{equation}
p^*(t,(x,s),d(y,r))=\delta_{t+s}(r)p(s,x;r,dy)
\end{equation}
where $\delta_u(.)$ is the delta measure on $\R_+$ for $u\geq 0$. Therefore, the semigroup of $\zeta^*$ is given by
\begin{equation}
\begin{array}{l}
T_t^*f^*(x,s)=\int\limits_{S\times[0,\infty)} f^*(y,r)p^*(t,(x,s),d(y,r))=\\
\int\limits_S f^*(y,t+s)p(s,x;t+s,dy)
\end{array}
\end{equation}
for $f^*\in \mathcal{B}(S\times [0,\infty))$. Let $G^*$ be the infinitesimal generator of $\zeta^*$ with domain $\mathcal{D}(G^*)$. Let $\mathcal{D}^*:=\{fh\in \mathcal{B}(S\times [0,\infty)):f\in \mathcal{D}(A),h\in \mathcal{C}_c^1([0,\infty),\R)\}$ where $\mathcal{C}_c^1([0,\infty),\R)$ is  the set of all continuously differentiable real valued functions on $[0,\infty)$ with compact support. In particular for any $f^*=fh \in \mathcal{D}^*$ with $f\in \mathcal{D}(A)$ and $h\in \mathcal{C}_c^1([0,\infty),\R)$, we have
\begin{equation}
T_t^*fh(x,s)=h(s+t)T_{s,s+t}f(x).
\end{equation}
Therefore, $\mathcal{D}^*\subset \mathcal{D}(G^*)$, and the infinitesimal generator of $\zeta^*$ restricted to the domain $\mathcal{D}^*$ is given by
\begin{equation}
\begin{array}{l}
G^*fh(x,s)=
h(s)G_sf(x)+h'(s)f(x).
\end{array}
\end{equation}
Let $G'=G^*|_{\mathcal{D}^*}$. In other words, $G' \subset\mathcal{B}(S\times \R_+) \times \mathcal{B}(S\times\R_+)$ is defined by
\begin{equation}
G'=\{(fh,gh+fh'):(f,g)\in G, h\in \mathcal{C}_c^1([0,\infty),\R)\}.
\end{equation}
By Theorem $4.7.1$ of (Ethier and Kurtz-1986 \cite{EK86-book}) , $\zeta$ is a solution to the time-dependent martingale problem $(G,\mathcal{D}(G),\mathbf{P}_0)$ for $\mathbf{P}_0\in \mathcal{P}(S)$ if and only if $\zeta^*$ is a solution to the martingale problem $(G^*,\mathcal{D}^*,\mathbf{P}_0^*)$, where $\mathbf{P}_0^*\in \mathcal{P}(S\times [0,\infty))$ is the image of $\mathbf{P}_0$ under the projection $x\mapsto (x,0)$, that is $\mathbf{P}_0^*(A,r)=\delta_0(r)\mathbf{P}_0(A)$ for $A\in \mathscr{B}(S)$. If, in addition, we assume $\mathcal{D}(G) \subset \bar{\mathcal{C}}(S)$ and that it also separates points and vanishes nowhere, then we can extend $G'$ to a subset of $G^*$ whose domain is an algebra which separates points. As $G^*$ is linear and as $\mathcal{D}^*$ is closed under pointwise multiplication of functions, the algebra of functions generated by $\mathcal{D}^*$, denoted by $\mathcal{D}^{**}$, is a linear subspace of $\mathcal{D}(G^*)$. Also $\mathcal{D}^{**}$ separates points and vanishes nowhere. Hence $\mathcal{D}(G^*)$ is dense in $\bar{\mathcal{C}}(S\times \R_+)$ in the topology of uniform convergence on compact sets that concludes $G^*:\mathcal{D}(G^*) \rightarrow \bar{\mathcal{C}}(S\times \R_+)$.
 Let $G''=G^*|_{\mathcal{D}^{**}}$. Consider the martingale problem $(G^*,\mathcal{D}^{**},\mathbf{P}_0^*)$ (or, with an equivalent notation, $(G'',\mathcal{D}^{**},\mathbf{P}_0^*)$). By linearity, any solution to $(G^*,\mathcal{D}^{**},\mathbf{P}_0^*)$ is a solution to $(G^*,\mathcal{D}^*,\mathbf{P}_0^*)$ and vice versa. Hence $\zeta$ is a solution to the time-dependent martingale problem $(G,\mathcal{D}(G),\mathbf{P}_0)$ if and only if $\zeta^*$ is a solution to the martingale problem $(G^*,\mathcal{D}^{**},\mathbf{P}_0^*)$.

The following lemma is useful to prove uniqueness in the case that we have a Markov solution of a time-dependent martingale problem.
\begin{lemma}\label{constructive uniqueness of martingale problem}
Let $S$ be a Polish space and $G=(G_t)_{t\geq0}$ be a time-dependent linear operator on $\mathcal{B}(S)$ with the domain $\mathcal{D}(G)\subset \bar{\mathcal{C}}(S)$ that contains an algebra of functions that separates points. Suppose there exists an $S$-valued Markov process with generator $G$ and initial distribution $\mathbf{P}_0 \in \mathcal{P}(S)$. Then the time-inhomogeneous martingale problem $(G,\mathcal{D}(G),\mathbf{P}_0)$ is well-posed.
\end{lemma}
\begin{remark}
The lemma remains true if $S$ is a separable metric space.
\end{remark}
\begin{proof}
Without loss of generality, we assume $\mathcal{D}(G)$ is an algebra, separating points. Otherwise, we prove the theorem for the subalgebra of $\mathcal{D}(G)$ with this property, which has at least one Markov solution, and hence the uniqueness of the latter implies the uniqueness of the original martingale problem. Let a Markov process $\zeta=(\zeta_t)_{t\geq 0}$ be a solution to the $(G,\mathcal{D}(G),\mathbf{P}_0)$ martingale problem. Then by discussion before the above lemma, the Markov space-time process $\zeta^*=(\zeta_t^*)_{t\geq 0}$ defined by $\zeta^*_t=(\zeta_t,t)$ is a solution to the martingale problem $(G^*,\mathcal{D}^{**},\mathbf{P}_0^*)$ (equivalently, $(G^*,\mathcal{D}^{**},\mathbf{P}_0^*)$), where $G^*$, $\mathcal{D}^{**}$ and $\mathbf{P}_0^*$ are defined as before. Therefore, it is sufficient to prove that $(\zeta_t^*)_{t\geq 0}$ is the unique solution to this martingale problem. 
Since $G''$ is the infinitesimal generator of $\zeta^*$ restricted to the domain $\mathcal{D}^{**}$, it is dissipative and there is $\lambda >0$ such that $\mathcal{R}(\lambda-G'')=\overline{\mathcal{D}(G'')}=\overline{\mathcal{D}^{**}}$. Thus, by theorem $4.4.1$ in \cite{EK86-book}, we need only to show that $\mathcal{D}^{**}$ is measure-determining. But $\mathcal{D}^{**}\subset \bar{\mathcal{C}}(S\times\R_+)$ and the algebra $\mathcal{D}^{**}$ separates points. The latter follows from the fact that both $\mathcal{D}(G)$ and $\mathcal{C}_c^1([0,\infty),\R_+)$ separate points. Hence by Theorem 3.4.5 \cite{EK86-book} $\mathcal{D}^{**}$ is measure-determining. This finishes the proof.
\end{proof}
In fact, we can see that the uniqueness of one-dimensional distributions of solutions of a martingale problem $(G,\mathcal{D},\textbf{P}_0)$ guarantees the uniqueness of the finite-dimensional distributions which, in turn, implies uniqueness of the martingale problem. Another important fact  about martingale formulation is that any unique solution of a martingale problem $(G,\mathcal{D},\textbf{P}_0)$ is strongly Markovian. The following is a restatement of Theorem $4.4.2$ and Corollary $4.4.3$ (Ethier and Kurtz 1986 \cite{EK86-book}) in the case of time-inhomogeneous martingale problems.
\begin{proposition}\label{Proposition-uniqueness-one dimensional distributions}
Let $S$ be a separable metric space, and let $G:\R_+\times \mathcal{B}(S)\rightarrow \mathcal{B}(S)$ be a time-dependent linear operator. If the one-dimensional distributions of any two possible solutions $\zeta^{(1)}$ and $\zeta^{(2)}$ of the martingale problem $(G,\mathcal{D},\textbf{P}_0)$ (with sample paths in $\mathcal{C}([0,\infty),S)$, respectively) coincide, i.e. if
\begin{equation}
\mathbb{P}(\zeta^{(1)}(t)\in A)=\mathbb{P}(\zeta^{(2)}(t)\in A)
\end{equation}
for any $t\geq 0$ and $A \in \mathscr{B}(S)$, then the martingale problem $(G,\mathcal{D},\textbf{P}_0)$ has at most one solution (with sample paths in $\mathcal{C}([0,\infty),S)$, respectively). In the case of existence, the solution is a Markov process. In addition to the above assumptions, if $G$ is a linear time-dependent operator satisfying $G:\R_+\times \bar{\mathcal{C}}(S)\rightarrow \mathcal{B}(S)$, then in the case of existence, the unique solution of $(G,\mathcal{D},\textbf{P}_0)$, namely $\zeta$, is strongly Markov, i.e. for any a.s. finite stopping time $\tau$ (with respect to the canonical filtration of $\zeta$, namely $\{\mathcal{F}_t^{\zeta}\}_{t\in \R_+})$
\begin{equation}
\mathbb{E}_{\mathbb{P}\zeta^{-1}}[f(\zeta(\tau+t)|\mathcal{F}_{\tau}^{\zeta}]=\mathbb{E}_{\mathbb{P}\zeta^{-1}}[f(\zeta(\tau+t)|\zeta(\tau)]
\end{equation}
\end{proposition}
\begin{proof}
The same argument as the one used in the proof of Theorem $4.4.2$ and Corollary $4.4.3$ (Ethier and Kurtz 1986 \cite{EK86-book}) (in homogeneous martingale problem case) proves the proposition.
\end{proof}

To have the uniqueness for a time-inhomogeneous martingale problem $(G, \mathcal{D},\mathbf{P}_0)$, it is necessary and sufficient that the one-dimensional distributions of any two possible solutions of the martingale problem coincide. The uniqueness of the one-dimensional distributions concludes that every solution is a Markov process, and, hence, it implies the uniqueness of time-dependent semigroups of two solutions. For general theory of martingale problems one can see \cite{SV79-book,EK86-book,D93}.

In the next subsection we develop the martingale problems with stochastic operator-valued processes.
\subsection{Quenched martingale problem in random environment and stochastic operator process}
As we restrict our attention to population dynamics in random time-varying environments (i.e. fitness processes), we are dealing with the Markov processes which are not only inhomogeneous in time but also their generators are random. Therefore, the idea of martingale problem should be extended in order to identify time-inhomogeneous Markov processes in random environments. We first describe what we mean by a stochastic process in a random environment.
\begin{definition}\label{quenched process}
Let $S'$ be a Polish space and let $\{\zeta^x\}_{x\in S'}$ be a family of $S$-valued measurable stochastic processes with laws $\{\textbf{P}^x\}_{x\in S'}$ all defined on the Borel probability space $(\Omega, \mathbb{P}, \mathcal{F})$ and with sample paths in $D_S[0,\infty)$ a.s.. As the supports of all $\{\textbf{P}^x\}_{x\in S'}$ are in $D_S[0,\infty)$, we regard these measures as the elements of $\mathcal{P}(D_S[0,\infty))$. Suppose the map $x \mapsto \textbf{P}^{x}(B)$ is measurable for any Borel measurable subset $B$ contained in $D_S[0,\infty)$. Let $X: \Omega \rightarrow S'$ be Borel measurable, i.e. $X$ is an $S'$-valued random variable. The mapping $\zeta^X:\omega\mapsto \zeta^{X(\omega)}(\omega)$ is called an $S$-valued stochastic (annealed) process in random environment $X$ with law $\textbf{P}$ which is the average over all $\textbf{P}^x$, i.e.
\begin{equation}
\textbf{P}(.)=\int\limits_\Omega \textbf{P}^x(.) \mathbb{P}^{env}(dx)
\end{equation}
where $\mathbb{P}^{env}=X*\mathbb{P}=\mathbb{P} X^{-1}$ is the law of $X$. Recall that $X*\mathbb{P}$ is the push-forward measure of $\mathbb{P}$ under the random variable $X$. For any $x \in S'$, $\zeta^x$ is called the quenched process with the given environment $x$.
As $\textbf{P} \in \mathcal{P}(D_S[0,\infty))$ 
, the annealed process has sample paths in $D_S[0,\infty)$.
\end{definition}

The fact that the quenched processes are Markov does not guarantee that the annealed process is Markov.
 Also, $\zeta^X$ need not to be Markov when, for $\mathbb{P}^{env}$-a.e. $x\in S'$, $\zeta^x$ is Markov but $X$ is not so.

We can consider a stochastic process in r.e. from the perspective of its generator which is a random time-varying generator. This hints us to think of the process by keeping the information of random variations for its generator. The following develops this concept.

\begin{definition}
Consider a Banach space $L$ with a linear subset $\mathcal{D}\subset L$ (i.e. $\mathcal{D}$ is closed under vector addition and scalar multiplication) and denote by $\mathscr{L}(\mathcal{D},L)$ the set of all bounded linear operators with domain $\mathcal{D}$ on $L$, equipped with the operator norm. By a linear operator process on $L$ with the domain $\mathcal{D}\subset L$ we mean an $\mathscr{L}(\mathcal{D},L)$-valued stochastic process, i.e. a mapping $G:\Omega \times \R_+ \rightarrow \mathscr{L}(\mathcal{D},L)$ such that $G(.,t)$ is a measurable function for any $t\in \R_+$ and $G(\omega,t)$ is a bounded linear operator on $L$ with domain $\mathcal{D}$ for any $t\in \R_+$ and $\omega \in \Omega$. A measurable linear operator process $G$ is a linear operator process for which $G$ is a measurable function. We, interchangeably, denote $G$ either as above, or as a function from $\Omega \times \R_+ \times \mathcal{D}$ to $L$, i.e. for any $f\in \mathcal{D}$ we set $G(\omega,t,f)=G(\omega,t)f$. Also, by $(G_t)_{t\geq 0}$ we denote a general linear operator process.
\end{definition}
As we deal with only  linear operators, we call the process defined above an "operator process".
\begin{definition}\label{Definition-operator process}
For an $S'$-valued random variable $X:\Omega\rightarrow S'$, we say an operator process $G:\Omega \times \R_+ \rightarrow \mathscr{L}(\mathcal{D},L)$ is consistent with $X$ (or with its law $\mathbb{P}X^{-1}\in \mathcal{P}(S')$) if for any $t\in \R_+$ the function $G(.,t)$ is constant on every measurable pre-image $X^{-1}(x)$ for $x \in supp(\mathbb{P}X^{-1})$. In this case we set $G(X^{-1}(x),s):=G(\omega,s)$ for an arbitrary $\omega\in  X^{-1}(x)$.
\end{definition}
Let $S'$ be a Polish space and $\mathcal{D}$ be a linear subspace of $\mathcal{B}(S)$. Let a probability measure $\mathbb{P}^{env} \in \mathbb{P}(S')$ be the distribution of an $S'$-valued random variable $X:\Omega\rightarrow S'$, i.e. $\mathbb{P}^{env}=\mathbb{P}X^{-1}$. Consider the operator process $G:\Omega \times \R_+ \times \mathcal{D} \rightarrow \mathcal{B}(S)$ which is consistent with $X$. We identify a time-inhomogeneous martingale problem in random environment (r.e.), or a quenched martingale problem in r.e. $X$, with a quadruple $(G, \mathcal{D}, \textbf{P}_0, \mathbb{P}^{env})$ where $\textbf{P}_0: S' \rightarrow \mathcal{P}(S)$ is measurable. From now on, when we speak of a quenched martingale problem $(G, \mathcal{D}, \textbf{P}_0, \mathbb{P}^{env})$, we automatically assume that $G$ is consistent with $\mathbb{P}^{env}$.
\begin{definition} \textit{(Quenched martingale problem in r.e.)}\label{Quenched martingale problem in r.e.}
Let $X$ be an $S'$-valued random variable with law $\mathbb{P}^{env}$. An $S$-valued stochastic process in r.e. $X$, namely $\zeta^X$, with the family of quenched laws $\{\textbf{P}^x\}_{x \in S'}$ and initial distributions $\{\textbf{P}_0(x)\}_{x\in S'}$ is said to be a solution of the quenched martingale problem $(G,\mathcal{D},\textbf{P}_0, \mathbb{P}^{env})$ if
for any $f \in \mathcal{D}$
\begin{equation}
f(\zeta_t^x)-\int\limits_0^t G_x(s) f(\zeta_s^x) ds
\end{equation} 
is a $\textbf{P}^x$-martingale with respect to the canonical filtration, for $\mathbb{P}^{env}$-almost all $x\in S'$, where
\begin{equation}
G_x(s)=G(X^{-1}(x),s)
\end{equation}
In this case, we also say $\{\textbf{P}^x\}_{x\in S'}$ is a solution to the quenched martingale problem. We say $\{\textbf{P}^x\}_{x\in S'}$ is a general solution to the martingale problem, if there exists $U\in \mathscr{B}(S')$ with $\mathbb{P}^{env}(U)>0$ such that for any $x\in U$ the support of $\textbf{P}^x$ is not in $D_S[0,\infty)$. We say the martingale problem is well-posed if there is a unique $\zeta^X$ solution with these properties (general solutions are not considered in the definition of uniqueness), i.e. there exists a unique ($\mathbb{P}^{env}$-a.s. uniquely determined) family $\{\textbf{P}^x\}_{x\in S'}$ satisfying the above conditions.
\end{definition}

\begin{remark}
In fact, each quenched martingale problem in r.e. $(G,\mathcal{D},\textbf{P}_0, \mathbb{P}^{env})$, $\mathbb{P}^{env}$-a.s. uniquely determines an $S'$-indexed family of time-inhomogeneous martingale problems $\{(G_x,\mathcal{D},\textbf{P}_0(x))\}_{x\in S'}$ along with the environment probability measure $\mathbb{P}^{env}$, and vice versa. The problems of existence and uniqueness of the first are equivalent to the problems of existence and uniqueness of the second family $\mathbb{P}^{env}$-a.s..
\end{remark}
Another process drawing our attention is a joint stochastic process and its environment. 
\begin{definition}
Suppose $\zeta^X$ be as defined in Definition \ref{quenched process}, i.e. a stochastic process in random environment $X$, with an extra assumption $S'=D_{\mathbf{E}}[0,\infty)$ for a Polish space $\mathbf{E}$. 
The random variable $X$ can be regarded as an $\mathbf{E}$-valued stochastic process with sample paths in $S'$. For $t\in \R_+$ and $\omega\in \Omega$, define the joint $S\times \mathbf{E}$-valued stochastic process by $\zeta_t(\omega)=(\zeta_t^{X(\omega)}(\omega),X_t(\omega))$ and call it the joint annealed-environment process. In fact for each $\omega\in \Omega$, $\zeta(\omega)$ gives a trajectory of environment and a trajectory of the process in that environment.
\end{definition}

Having the law of the joint process $\zeta$, how can we retrieve the law of $\zeta^X$? Let $\textbf{P}^*$ be the law of $\zeta$ and $\mathbb{P}^{env}$ be the law of $X$. Since $S$ and $S'$ are Polish, by disintegration theorem, there exists a unique family of probability measures $\{\textbf{P}^{*x}\}_{x\in S'}\subset \mathcal{P}(D_S[0,\infty)\times S')$ (unique w.r.t. $\mathbb{P}^{env}$, that is for any such family $\{\textbf{P}^{**x}\}_{x\in S'}$, $\mathbb{P}^{env}(\{x\in S': \textbf{P}^{*x}=\textbf{P}^{**x}\})=1$) satisfying
\begin{enumerate}[(i)]
\item The map $x \mapsto \textbf{P}^{*x}(B)$ is Borel measurable for any Borel measurable set $B \subset D_S[0,\infty)\times S'$.
\item $\textbf{P}^{*x}((D_S[0,\infty)\times S') \setminus Y_x)=0$ for $\mathbb{P}^{env}$-almost all $x \in S'$ where $Y_x:=\{(s,x)\}_{s\in D_S[0,\infty)}$. 
\item For any Borel measurable subset $B$ of $D_S[0,\infty)\times S'$ we have 
\begin{equation}
\textbf{P}^*(B)=\int\limits_{S'} \textbf{P}^{*x}(B) \mathbb{P}^{env}(dx).
\end{equation}
\end{enumerate}
Then, for $\mathbb{P}^{env}$-a.e. $x\in S'$, $\textbf{P}^x$ will be the push-forward measure of $\textbf{P}^{*x}$ under the measurable projection from $D_S[0,\infty)\times S'$ onto $D_S[0,\infty)$. We also can observe that the annealed measure $\textbf{P}$ is the push-forward measure of $\textbf{P}^*$ under the projection from $D_S[0,\infty)\times S'$ onto $D_S[0,\infty)$, and it can gives another way to construct quenched measures, as they are in fact conditional measures of $\textbf{P}$ and can be derived by disintegration theorem for $\textbf{P}$ and $\mathbb{P}^{env}$. The following diagram summarizes the relations of these measures.
\begin{equation}
\begin{array}{ccc}
\textbf{P}^* & \xrightarrow{push-forward} & \textbf{P}\\
\downarrow &   & \downarrow \\
\{\textbf{P}^{*x}\}_{x\in S'} & \xrightarrow{push-forward} & \{\textbf{P}^{x}\}_{x\in S'}
\end{array}
\end{equation}


\section{Moran and Fleming-Viot processes in random environments: Martingale characterization}
In this section, first, we identify the Moran process in r.e. as a quenched martingale problem and prove its wellposedness, and then we define the generator of the Fleming-Viot (FV) process in r.e. and the quenched martingale problem for it as well. Also, in this section, we state some main results of this paper for Moran and FV processes in r.e.. This includes the wellposedness of the quenched martingale problem for the FV process in r.e., some properties of this process such as continuity of the sample paths almost surely, and the weak convergence of the quenched (and annealed) measure-valued Moran processes to the quenched (and annealed) FV process, when the environments of the first (fitness processes) converge to that of the second. Also, under the assumption of existence of a parent-independent component of the mutation process and certain assumptions for the environment process, we state an ergodic theorem for the annealed-environment process. The proofs of the main theorems will come in sections \ref{Section-proof of convergence of MRE to FVRE and wellposedness}, \ref{Section-Continuity}, and \ref{Section-proof of ergodic theorem}.

Throughout this paper we assume that $I$ is a compact metric space, called type space. Any element of $I$ is called a type or allele. We also assume that $E \subset \mathcal{C}(I,[0,1])$ is compact. A fitness function (or selection intensity function) is a Borel measurable function from $I$ to $[0,1]$. In this paper, we assume that fitness functions are in $C(I,[0,1])$.  

\begin{definition}
A \textit{fitness process} is an $E$-valued measurable stochastic process defined on $(\Omega, \mathbb{P}, \mathcal{F})$ with sample paths in $D_E[0,\infty)$. When the fitness process is Markov, we call it Markov fitness.
\end{definition}
\begin{remark}
Restricting the fitness processes to have sample paths in $D_E[0,\infty)$ is an essential assumption to guarantee the generators of Moran and Fleming-Viot processes in random environments exist for a suitable set of functions.
\end{remark}
Let $e$ be a fitness process. As $E$ is a compact space and therefore separable, $e$ can be regarded as a $D_E[0,\infty)$-valued random variable defined on $(\Omega, \mathbb{P}, \mathcal{F})$, that is $e:\Omega \rightarrow D_E[0,\infty)$ be a measurable map. We denote by
\begin{equation}
\mathbb{P}^{env}=\mathbb{P}^{env,e}:=e*\mathbb{P}=\mathbb{P}e^{-1}\in \mathcal{P}(D_E[0,\infty))
\end{equation}
the distribution of $e$. For simplicity of notation, we let $e_t=e(t)$ for a fitness process $e$. We frequently denote by $e=(e_t)_{t\geq 0}$ a fitness process and by $\tilde{e}\in D_E[0,\infty)$ a trajectory of $e$. Also, we denote by $\hat{e}\in E$ a fitness function. We emphasise that, in the sequel, a fitness process $e$ is regarded as both an $E$-valued measurable stochastic process with sample paths in $D_E[0,\infty)$ and a $D_E[0,\infty)$-valued random variable with the law $\mathbb{P}^{env}$. We assume that the possible times of selection occur with rate $\alpha/N$ independently for every individual, and at a possible time $t$ of selection for individual $i\leq N$, a selective event occurs with probability $e_t(Y_i^{N,e}(t))$. Recall that $Y_i^{N,e}(t)$ is the type of individual $i$ at time $t$, in the particle Moran process in random environment (PMRE) with $N$ individuals. For the results in this paper we assume that the fitness process is either a general $E$-valued stochastic process or a Markov process. 
 We continue this section with identifying the Moran process in r.e. (MRE) as a solution of a quenched martingale problem in r.e..
\subsection{Moran process in random environments}
By a Moran process we think of the measure-valued Moran process with resampling, mutation and selection with a compact type space $I$ and an $E$-valued fitness process as constructed in introduction in detail. Recall that $E$ is compact in this paper. For $N\in \N$, let $\mathcal{P}^N(S)$ be the set of all purely atomic probability measures in $\mathcal{P}(S)$ with at most $N$ atoms such that $Nm(.)$ is a counting measure.
In other words, $\mathcal{P}^N(S)$ is the image of $I^N$ under the map
\begin{equation}
(a_1,...,a_N) \mapsto \frac{1}{N}\sum\limits_{i=1}^N \delta_{a_i}
\end{equation}
from $I^N$ to $\mathcal{P}(S)$ where $\delta_a$ is the delta measure with support $\{a\}\subset S$. An element of $\mathcal{P}^N(S)$ is called an empirical measure on $S$ (with at most $N$ atoms). In this section we assume that the number of individuals $N\geq 1$ is fixed. With $N$ individuals and type space $I$, let $(\mu_N^e(t))_{t\in \R_+}$ be a measure-valued ($\mathcal{P}^N(I)$-valued) Moran process with fitness process $e$ whose law is given by $\mathbb{P}^{env}=\mathbb{P}^{env,e} \in \mathcal{P}(D_E[0,\infty))$ and let $\alpha/N, \beta, \gamma/2>0$ be the selection, mutation, and resampling rates, respectively. We assume that the fitness process $e$  evolves between jumps and is independent of the Poisson times of jumps (for resampling, mutation and selection), the initial distribution, and also of the mutation kernel, i.e. it is independent of the outcome of a mutation event that occurs on type $a$ for every $a\in I$. Let $q(x,dy)$ be a stochastic kernel for the mutation process on the state space $I$, that is the type of the offspring of an individual with type (allele) $x$ after a mutation event follows the transition function $q(x,dy)$. As $q(x,dy)$ can either depend on $x$ or not, it is always possible to write the mutation kernel as
\begin{equation}\label{mutation kernel}
\beta q(x,dy)=\beta' q'(dy)+\beta '' q''(x,dy)
\end{equation}
for $\beta', \beta'' \geq 0$ s.t. $\beta'+\beta''=\beta>0$. The first term in the right hand side of equation (\ref{mutation kernel}) is called parent-independent component of the mutation event. When there is no ambiguity in the notation and the fitness process is known, we drop the superscript $e$ and denote MRE with the fitness function $e$ by $\mu_N$. Also we denote by $\mu_N^{\tilde{e}}$ the quenched Moran process with the deterministic fitness process $\tilde{e} \in D_E[0,\infty)$.

To study $\mu_N$ as a quenched martingale problem in r.e. $e$, we need to determine a convenient set of functions as the domain of its generator. We use the following domain for the generator of $\mu_N$ which has been used by several authors as a domain for the generator of the classic measure-valued Moran process. For an empirical measure $m\in \mathcal{P}^N(I)$, let $m^{(N)} \in \mathcal{P}(I^N)$ be the $N$ times sampling measure without replacement from $m$, i.e. letting $da=d(a_1,...,a_N)$
\begin{equation}\label{Convolution power-without replacement}
m^{(N)}(da)=m(da_1)\times \frac{Nm-\delta_{a_1}}{N-1}(da_2) \times ...\times \frac{Nm-\sum_{i=1}^{N-1}\delta_{a_i}}{1}(da_N)
\end{equation}

Let $\tilde{\mathfrak{F}}_N$ be the algebra generated by all functions $\tilde\Phi_N^f:\mathcal{P}^N(I) \rightarrow \R$ for $f\in \mathcal{C}(I^N)$ with
\begin{equation}
\tilde\Phi_N^f(m)=<m^{(N)},f>=\int\limits_{I^N} fdm^{(N)}.
\end{equation}
Note that $\bar{\mathcal{C}}(I)=\mathcal{C}(I)$, $\bar{\mathcal{C}}(I^N)=\mathcal{C}(I^N)$, $\bar{\mathcal{C}}(\mathcal{P}(I))=\mathcal{C}(\mathcal{P}(I))$ and $\bar{\mathcal{C}}(\mathcal{P}^N(I))=\mathcal{C}(\mathcal{P}^N(I))$, since $I$ and therefore $\mathcal{P}(I)$ and $\mathcal{P}^N(I)$ are compact. Also note that any function in $\tilde{\mathfrak{F}}_N$ is a restriction of a function in $\mathcal{C}(\mathcal{P}(I))$. 
\begin{proposition}\label{Proposition-algebra separates points-Moran}
For any $N\in \N$, the algebra $\tilde{\mathfrak{F}}_N$ separates points, and hence is measure and convergence-determining on $\bar{\mathcal{C}}(\mathcal{P}^N(I))$. Also $\tilde{\mathfrak{F}}_N$ vanishes nowhere.
\end{proposition}
\begin{proof}
Let $m_1,m_2 \in \mathcal{P}^N(I)$ such that $m_1\neq m_2$. Let $\mathcal{R}:=supp(m_1)\cap supp(m_2)$. There exists $a\in \mathcal{R}$ such that $m_1(a)\neq m_2(a)$. Since $|\mathcal{R}| \leq 2N$, there exists $r>0$  such that the ball radius $r$ centred at $a$ (w.r.t. the metric of $I$), namely  $B(a,r)$, excludes all the points of $\mathcal{R}$ except $a$. It is clear that there exists a function $f\in \mathcal{C}(I)$ with $f(a)=1$ that vanishes outside of $B(a,r)$. Consider $\tilde{f}\in \mathcal{C}(I^N)$ which depends only on the first variable in $I^N$ and defined by $\tilde{f}(x_1,x_2,...,x_3)=f(x_1)$. Then
\begin{multline}
<m_1^{(N)},\tilde{f}>=<m_1,f>=m_1(a)\neq
 m_2(a)=<m_2,f>=<m_2^{(N)},\tilde{f}>.
\end{multline}
Also, $\tilde{\mathfrak{F}}_N$ vanishes nowhere, since the constant function $1 \in \mathcal{C}(I^N)$ and for any $m\in \mathcal{P}^N(I)$, $<m,1>=1\neq 0$. 
\end{proof}
\begin{remark}
Alternatively, the latter proposition can be proved by showing that $\tilde{\mathfrak{F}}_N$ strongly separates points. For the definition see \cite{EK86-book}, Section $3.4$.
\end{remark}
It is straightforward to see that the generator of the MRE with fitness process $e$ on $\tilde{\mathfrak{F}}_N$ is the operator process $\tilde{\mathcal{G}}^N: \Omega \times \R_+ \times \tilde{\mathfrak{F}}_N \rightarrow \tilde{\mathfrak{F}}_N$ consistent with the environment process $e$ given by
\begin{equation}
\begin{array}{l}
\tilde{\mathcal{G}}^N=\tilde{\mathcal{G}}^{N,e}:=\\
\tilde{\mathcal{G}}^{res,N,e}+\tilde{\mathcal{G}}^{mut,N,e}+\tilde{\mathcal{G}}^{sel,N,e}
\end{array}
\end{equation}
where $\tilde{\mathcal{G}}^{res,N,e}$ and $\tilde{\mathcal{G}}^{mut,N,e}$, i.e. the resampling and mutation generators, are linear operators from $\tilde{\mathfrak{F}}_N$ to $\tilde{\mathfrak{F}}_N$, and $\tilde{\mathcal{G}}^{sel,N,e}$, the selection generator, is an operator process consistent with $e$. We usually drop the superscript $e$, if there is no risk of ambiguity. To be more explicit, let 
\begin{equation}
\bar{I}_k=(\bigcup\limits_{n=k}^\infty I^n) \cup I^{\N}
\end{equation}
for $k\in \N$. For the resampling generator, we have
\begin{equation}\label{equation-resampling generator Moran}
\tilde{\mathcal{G}}^{res,N} \tilde\Phi_N^f(m)= \frac{\gamma}{2} \sum\limits_{i,j=1}^N <m^{(N)},f \circ \sigma_{ij} -f>
\end{equation}
where $\sigma_{ij}: \bar{I}_{i \vee j} \rightarrow \bar{I}_{i \vee j}$ is a map replacing the $j$-th component of $x \in \bar{I}_{i \vee j}$ with the $i$-th one ($i\vee j=max\{i,j\}$). In other words, defining another map $\sigma_j^y: \bar{I}_{j} \rightarrow \bar{I}_{j}$ for $j \in \N$ and $y\in I$ with $\sigma_j^{y}(x):=(x_1,...,x_{j-1},y,x_{j+1},..., x_n)$ if $j\leq n$ and $x=(x_1,..., x_n)$, and with  $\sigma_j^{y}(x):=(x_1,...,x_{j-1},y,x_{j+1},...)$ if $x=(x_1,x_2,...)$, we have $\sigma_{ij}(x)=\sigma_j^{x_i}$ (the reason to define these functions to be so general is to use them later for Fleming-Viot processes).

For mutation, we have
\begin{multline}
\tilde{\mathcal{G}}^{mut,N} \tilde\Phi_N^f(m)=\beta \sum\limits_{i=1}^N <m^{(N)},B_i f-f> \\ =\beta' \sum\limits_{i=1}^N <m^{(N)},B_i' f-f>+\beta'' \sum\limits_{i=1}^N <m^{(N)},B_i'' f-f>,
\end{multline}
where
\begin{equation}
B_i, B_i', B_i'': \mathcal{C}(I^{\N})\cup (\bigcup\limits_{n\geq i} \mathcal{C}(I^{n})) \rightarrow \mathcal{C}(I^{\N})\cup (\bigcup\limits_{n\geq i} \mathcal{C}(I^{n}))
\end{equation}
 are bounded linear operators defined by
\begin{equation}\label{mutation integral-1}
B_i f(x)=\int\limits_I f \circ \sigma_i^y (x)q(x_i,dy),
\end{equation}
\begin{equation}\label{mutation integral-2}
B_i' f(x)=B_i' f(x):=\int\limits_I f \circ \sigma_i^y (x)q'(dy),
\end{equation}
\begin{equation}\label{mutation integral-3}
B_i'' f(x)=B_i'' f(x):=\int\limits_I f \circ \sigma_i^y (x)q''(x_i,dy).
\end{equation}
Note that $B_i, B_i', B_i''$ leaves $\mathcal{C}(I^{\N})$ and, for any $n\in \N$, $\mathcal{C}(I^{n})$ invariant.

To have a generator process consistent with $e$ we first need to specify the time-dependent generator of $\tilde{\mu}_N^{\tilde{e}}$ for any given (quenched) environment $\tilde{e}$. For any $t\geq 0$, let 
\begin{equation}
\tilde{e}_i(t)(x):=\tilde{e}(t)(x_i)
\end{equation}
and, for $x\in I^N$ and $s\geq 0$, define
\begin{equation}
\overrightarrow{e}_i^{t,t+s}(x):=\int\limits_t^{t+s} \tilde{e}_i(r)(x)P_{t,t+s}^*(dr),
\end{equation}
where $P_{t,t+s}^*$ is the distribution of the position of a Poisson point in the interval $[t,t+s]$, conditioned to have exactly one Poisson point in the interval.
 Then the generator process for the selection is
\begin{equation}
\begin{array}{l}
\tilde{\mathcal{G}}_{\tilde{e}}^{sel,N}(t) \tilde\Phi_N^f(m)=\\
\lim\limits_{s\rightarrow 0} \frac{1}{s}(\frac{\alpha}{N}s\sum\limits_{i,j=1}^N<m^{(N)},\overrightarrow{e}_i^{t,t+s}f\circ\sigma_{ij}>+\frac{\alpha}{N}s\sum\limits_{i,j=1}^N<m^{(N)},(1-\overrightarrow{e}_i^{t,t+s})f>\\
+(1-\frac{\alpha}{N}s)\sum\limits_{i,j=1}^N<m^{(N)},f>-\sum\limits_{i,j=1}^N<m^{(N)},f>)\\
=\lim\limits_{s\rightarrow 0} \frac{1}{s}(\frac{\alpha}{N}s\sum\limits_{i,j=1}^N<m^{(N)},\overrightarrow{e}_i^{t,t+s}f\circ\sigma_{ij}-\overrightarrow{e}_i^{t,t+s}f>).
\end{array}
\end{equation}
But $\tilde{e}\in D_E[0,\infty)$ is right-continuous (has a right limit), and hence so is $\tilde{e}_i$ for any $i\leq N$. Therefore for any $x\in I^N$
\begin{equation}
\inf\limits_{t\leq r\leq t+s}\tilde{e}_i(r)\leq \overrightarrow{e}_i^{t,t+s}(x)\leq \sup\limits_{t\leq r\leq t+s}\tilde{e}_i(r), 
\end{equation}
and hence
\begin{equation}
\overrightarrow{e}_i^{t,t+s}(x) \rightarrow \tilde{e}_i(t)(x)
\end{equation} 
as $s\rightarrow 0$, and furthermore $\overrightarrow{e}_i^{t,t+s} \rightarrow \tilde{e}_i(t)$ in the sup-norm topology. This concludes that the generator process for the selection, $\tilde{\mathcal{G}}_{\tilde{e}}^{sel,N}: \R_+ \times \tilde{\mathfrak{F}}_N \rightarrow \tilde{\mathfrak{F}}_N$, is given by
\begin{equation}
\tilde{\mathcal{G}}_{\tilde{e}}^{sel,N}(t) \tilde\Phi_N^f(m)=\frac{\alpha}{N} \sum\limits_{i,j=1}^N <m^{(N)},\tilde{e}_i(t)(f \circ \sigma_{i,j}^N-f)>.
\end{equation}

For simplicity, similarly to the Definition \ref{Definition-operator process}, we denote
\begin{equation}
\tilde{\mathcal{G}}^{sel,N}(e^{-1}(\tilde{e}),t):=\tilde{\mathcal{G}}_{\tilde{e}}^{sel,N}(t).
\end{equation}
Therefore, the selection generator process is the operator process $\tilde{\mathcal{G}}^{sel,N}=(\tilde{\mathcal{G}}^{sel,N}_t)_{t\geq 0}$ or 
\begin{equation}
\tilde{\mathcal{G}}^{sel,N}:\Omega \times \R_+ \times \tilde{\mathfrak{F}}_N \rightarrow \tilde{\mathfrak{F}}_N
\end{equation}
that $\tilde{\mathcal{G}}^{sel,N}(\omega,t)$ is a linear operator from $\tilde{\mathfrak{F}}_N$ to $\tilde{\mathfrak{F}}_N$, defined by
\begin{equation}
\tilde{\mathcal{G}}^{sel,N}(\omega,t)=\tilde{\mathcal{G}}^{sel,N}(e^{-1}(\tilde{e}),t)
\end{equation}
for any $\omega\in e^{-1}(\tilde{e})$ and any $\tilde{e}$ in the range of $e$. Note that the value of $\tilde{\mathcal{G}}^{sel,N}(\omega,t)$ for $\omega \in e^{-1}(\tilde{e})$ with $\tilde{e}$ out of range of $e$ is not important and, actually, it can be any linear operator on $\tilde{\mathfrak{F}}_N$.
Equivalently,
\begin{equation}
\tilde{\mathcal{G}}^{sel,N}(\omega,t) \tilde\Phi_N^f(m)=\frac{\alpha}{N} \sum\limits_{i,j=1}^N <m^{(N)},e_i(\omega,t)(f \circ \sigma_{i,j}^N-f)>,
\end{equation}
where
\begin{equation}
e_i(\omega,t)(x):=e(\omega,t,x_i)=\tilde{\mathcal{G}}^{sel,N}(e^{-1}(\tilde{e}),t).
\end{equation}
Note that, in order to ensure $\tilde{\mathcal{G}}_{\tilde{e}}^{sel,N}(t): \tilde{\mathfrak{F}}_N \rightarrow \tilde{\mathfrak{F}}_N$, $\tilde{e}(t)$ must be in $\mathcal{C}(I,[0,1])$. In fact we have assumed more, i.e. $\tilde{e} \in D_E[0,\infty)$ (recall $E\subset \mathcal{C}(I,[0,1])$). The (quenched) linear generator of the Moran process with a deterministic fitness process $\tilde{e}\in D_E[0,\infty)$, namely $\tilde{\mathcal{G}}_{\tilde{e}}:\R_+ \times \tilde{\mathfrak{F}}_N \rightarrow \tilde{\mathfrak{F}}_N$, is given by
\begin{equation}
\tilde{\mathcal{G}}_{\tilde{e}}^N:=\tilde{\mathcal{G}}^{res,N}+\tilde{\mathcal{G}}^{mut,N}+\tilde{\mathcal{G}}_{\tilde{e}}^{sel,N}.
\end{equation}

\begin{proposition}\label{Proposition-well-posedness of Moran martingale problem}
Let $\tilde{\textbf{P}}_0^N: D_E[0,\infty) \rightarrow \mathcal{P}(\mathcal{P}^N(I))$ be measurable and $\tilde{\mathcal{G}}^N$ be as defined above. The $(\tilde{\mathcal{G}}^N,\tilde{\mathfrak{F}}_N,\tilde{\textbf{P}}_0^N,\mathbb{P}^{env})$-martingale problem is well-posed, and $\mu_N^e$ is identified as the solution of this martingale problem.
\end{proposition}
\begin{proof}
The existence has been shown by construction. Note that the constructed quenched solutions $\mu_N^{\tilde{e}}$, for every $\tilde{e}\in D_E[0,\infty)$, is also Markov. It suffices to prove that, for any $\tilde{e}\in D_E[0,\infty)$, the time-inhomogeneous martingale problem $(\tilde{\mathcal{G}}_{\tilde{e}}^N, \tilde{\mathfrak{F}}_N,\tilde{\textbf{P}}_0^N(\tilde{e}))$ is well-posed
. But this follows from Lemma \ref{constructive uniqueness of martingale problem} and the fact that the algebra $\tilde{\mathfrak{F}}_N$ separates points.
\end{proof}
\begin{remark}
The latter proposition can also be proved using the duality method in the same way that we show the uniqueness of the quenched martingale problem for Fleming-Viot process in r.e.. See sections \ref{Section-general duality} and \ref{Section-FV Duality}.
\end{remark}
\subsection{Fleming-Viot process in random envirnoments}
Identifying MRE as a solution of a well-posed quenched martingale problem, we prove that the FVRE process arises as the weak limit (in $D_{\mathcal{P}(I)}[0,\infty)$) of MRE processes with $N$ individuals as $N\rightarrow \infty$. In fact, we prove the following stronger weak convergence. For a sequence of fitness processes $e_N$ converging weakly to a fitness process $e$, in $D_E[0,\infty )$, FVRE process with the fitness process $e$, namely $\mu^e$, arises as the weak limit of MRE processes $\mu_N^{e_N}$ in $D_{\mathcal{P}(I)}[0,\infty)$. The first step to prove this kind of theorems is to introduce the FVRE martingale problem. Here we set up the quenched martingale problem for FVRE.

Let $\mathcal{B}_n(I^\N)$, $\mathcal{C}_n(I^\N)$, and $\mathcal{C}_n^k(I^\N)$ be the subsets of $\mathcal{B}(I^\N)$, $\mathcal{C}(I^\N)$, and $\mathcal{C}^k(I^\N)$, respectively, depending on the first $n$ variables of $I^\N$.
\begin{definition}
For $f \in \mathcal{B}_n(I^{\N})$, a polynomial is a function 
\begin{equation}
\tilde\Phi^f=\tilde\Phi: \mathcal{P}(I) \rightarrow \R
\end{equation}
defined by
\begin{equation} \label{polynomial}
\tilde\Phi^f(m):=<m^{\otimes \N}, f> \ \ for \ m \in \mathcal{P}(I),
\end{equation}
where $m^{\otimes \N}$ is the $\N$-fold product measure of $m$. The smallest number $n$ for which (\ref{polynomial}) holds is called the degree of $\tilde\Phi^f$. 
\end{definition}

Let
\begin{equation}
\tilde{\mathfrak{F}}^k=\{\tilde\Phi^f: f\in \bar{\mathcal{C}}_n^k(I^{\N}) \ for \ some \ n\in\N\}
\end{equation}
for $k\in \N\cup\{0,\infty\}$, and let $\tilde{\mathfrak{F}}=\tilde{\mathfrak{F}}^0$. 
\begin{proposition}\label{Proposition-algebra separates points-FV}
For $k\in \N\cup \{0,\infty\}$, $\tilde{\mathfrak{F}}^k$ is an algebra of functions that separates points and vanishes nowhere, therefore it is measure and convergence-determining.
\end{proposition}
\begin{proof}
To show that $\tilde{\mathfrak{F}}^k$ is an algebra of functions for every $k\in \N\cup \{0,\infty\}$, observe that for $\Phi^f, \Phi^g \in \tilde{\mathfrak{F}}^k$ with degree $n_1$ and $n_2$, respectively, we have
\begin{equation}
\Phi^f+ \Phi^g=\Phi^{f+g} \in \tilde{\mathfrak{F}}^k
\end{equation}
where $f+g\in C_n(I^{\N})$, for $n=max\{n_1,n_2\}$. Also
\begin{equation}\label{product of functions in the algebra}
\Phi^f\Phi^g=\Phi^{f.(g\circ \tau_{n_1})}\in \tilde{\mathfrak{F}}^k,
\end{equation}
where $\tau_k:I^{\N}\rightarrow I^{\N}$ is the translation operator on $I^{\N}$ defined by $\tau_r(a_1,a_2,...)=(a_{r+1}, a_{r+2},...)$ for $r\in \N$. Note that being convergence-determining and measure-determining are equivalent for $\tilde{\mathfrak{F}}^k \subset C(\mathcal{P}(I))$, by Lemma 3.4.3 \cite{EK86-book}. Thus, by Theorem 3.4.5 \cite{EK86-book}, it suffices to show that $\tilde{\mathfrak{F}}^k$ separates points. The latter follows from the fact that for any $m_1,m_2\in \mathcal{P}(I)$, there exists $f\in \mathcal{C}_1^k(I^{\N})$, for $k\in \N\cup \{0,\infty\}$, such that
\begin{equation}
\Phi^f(m_1)=<m_1,f>\neq <m_2,f>=\Phi^f(m_2).
\end{equation}
Also $1\in \mathcal{C}_0^k(I^{\N})$, and for any $m\in \mathcal{P}(I)$, we have $1=<m,1>\neq 0$. This proves the proposition.
\end{proof}
\begin{remark}
$\tilde{\mathfrak{F}}$ is dense in $\mathcal{C}(\mathcal{P}(I))$ in the topology of uniform convergence on compact sets.
\end{remark}
We are ready to define the generator of FVRE and state the quenched martingale problem in r.e. for it. For $n\in \N$ and $f\in \mathcal{C}_n(I^\N)$, let $\tilde{\Phi}^f$ be a polynomial. The generator of the FVRE with a fitness process $e$ is the operator process
\begin{equation}
\tilde{\mathcal{G}}^e:\Omega \times \R_+ \times \tilde{\mathfrak{F}} \rightarrow \tilde{\mathfrak{F}}
\end{equation}
also denoted by $\tilde{\mathcal{G}}=(\tilde{\mathcal{G}}_t)_{t\geq 0}$, defined as
\begin{equation}
\tilde{\mathcal{G}}^e=\tilde{\mathcal{G}}^{res}+\tilde{\mathcal{G}}^{mut}+\tilde{\mathcal{G}}^{sel,e},
\end{equation}
where the first and the second terms on the right hand side are the linear operators corresponding to resampling and mutation (generators) from $\tilde{\mathfrak{F}}$ to $\tilde{\mathfrak{F}}$, and the third one is an operator process serving as the selection generator. Usually, we drop the superscript $e$, when there is no risk of confusion. For $\tilde{\Phi}^f\in \tilde{\mathfrak{F}}$, $m\in \mathcal{P}(I)$ and $x=(x_1,x_2,...)$, the operator process is defined as follows.

The resampling generator is defined by 
\begin{equation}
\tilde{\mathcal{G}}^{res} \tilde{\Phi}^f(m)= \frac{\gamma}{2} \sum\limits_{i,j=1}^n <m^{\otimes \N},f \circ \sigma_{i,j}-f>.
\end{equation}

For mutation, put
\begin{multline}
\tilde{\mathcal{G}}^{mut} \tilde\Phi^f(m)=\beta \sum\limits_{i=1}^n <m^{\otimes \N},B_i f-f> \\ =\beta' \sum\limits_{i=1}^n <m^{\otimes \N},B_i^{'} f-f>+\beta'' \sum\limits_{i=1}^n <m^{\otimes \N},B_i^{''} f-f>
\end{multline}
Recall that
\begin{equation}
B_i f(x)=\int\limits_I f \circ \sigma_i^y (x)q(x_i,dy)
\end{equation}
\begin{equation}
B_i' f(x)=\int\limits_I f\circ \sigma_i^y (x)q'(dy)
\end{equation}
\begin{equation}
B_i'' f(x)=\int\limits_I f\circ \sigma_i^y (x)q''(x_i,dy).
\end{equation}

For the selection generator, define the following operator process
\begin{equation}
\tilde{\mathcal{G}}^{sel}:\Omega \times \R_+ \times \tilde{\mathfrak{F}} \rightarrow \tilde{\mathfrak{F}}
\end{equation}
consistent with $e$ such that $\tilde{\mathcal{G}}^{sel}(\omega,t)$ is defined to be a linear operator from $\tilde{\mathfrak{F}}$ to $\tilde{\mathfrak{F}}$ as
\begin{equation}
\tilde{\mathcal{G}}^{sel}(\omega,t) \tilde\Phi^f(m)= {\alpha} \sum\limits_{i=1}^n <m^{\otimes \N},e_i(\omega,t) f -e_{n+1}(\omega,t)f>.
\end{equation}
Recall that $e_i(\omega,t)(x)=e(\omega,t,x_i)$ and $\tilde{e}_i(t)(x)=\tilde{e}(t)(x_i)$ for a given 
trajectory $\tilde{e}\in D_E[0,\infty)$. Then as denoted in Definitions \ref{Definition-operator process} and \ref{Quenched martingale problem in r.e.}, we have
\begin{equation}
\tilde{\mathcal{G}}_{\tilde{e}}^{sel}(t) \tilde\Phi^f(m)=\tilde{\mathcal{G}}^{sel}(e^{-1}(\tilde{e}),t) \tilde\Phi^f(m)= \alpha \sum\limits_{i=1}^n <m^{\otimes \N},\tilde{e}_i(t) f - \tilde{e}_{n+1}(t)f>.
\end{equation}
Also, for a given trajectory $\tilde{e}\in D_E[0,\infty)$, let $\tilde{\mathcal{G}}_{\tilde{e}}=\tilde{\mathcal{G}}_{\tilde{e}}^e:\R_+\times\tilde{\mathfrak{F}}\rightarrow \tilde{\mathfrak{F}}$ be defined by
\begin{equation}
\tilde{\mathcal{G}}_{\tilde{e}}:=\tilde{\mathcal{G}}^{res}+\tilde{\mathcal{G}}^{mut}+\tilde{\mathcal{G}}_{\tilde{e}}^{sel}.
\end{equation}
The following theorems state the wellposedness of FVRE martingale problem and identify the limit of the measure-valued Moran processes in r.e. as the unique solution of a quenched martingale problem in r.e..
\begin{theorem}\label{Theorem-Wellposedness of FVRE martingale problem}
Let $e$ be a fitness process and let $\tilde{\textbf{P}}_0:D_E[0,\infty) \rightarrow \mathcal{P}(\mathcal{P}(I))$ be measurable, and $\mathbb{P}^{env,e}\in \mathcal{P}(D_E[0.\infty))$. The $(\tilde{\mathcal{G}}^e,\tilde{\mathfrak{F}},\tilde{\textbf{P}}_0,\mathbb{P}^{env,e})$-martingale problem is well-posed. Furthermore, the unique solution is a strong Markov process.
\end{theorem}
\begin{definition}
Let $e$ be an $E$-valued stochastic process with sample paths in $D_E[0,\infty)$ and law $\mathbb{P}^{env}$. The unique $\mathcal{P}(I)$-valued process which is the solution of the  martingale problem $(\tilde{\mathcal{G}}^e,\tilde{\mathfrak{F}},\tilde{\textbf{P}}_0,\mathbb{P}^{env,e})$, denoted by $\mu^e$, is called Fleming-Viot process in r.e. $e$ ($FVRE$). When there is no risk of ambiguity, we drop $e$ from the superscripts. For a given  trajectory of $e$, namely $\tilde{e}$ picked by $\mathbb{P}^{env}$, $\mu^{\tilde{e}}$ represents the quenched FV process with the deterministic (fixed) environment $\tilde{e}$.
\end{definition}

Recall that we frequently denote by $e, e_N$ stochastic fitness processes, and by $\tilde{e}$ a fixed time-dependent fitness function (an element of $D_E[0,\infty)$). Also note that measurable functions $\tilde{\mathbf{P}}_0:D_E[0,\infty)\rightarrow \mathcal{P}(\mathcal{P}(I))$ and $\tilde{\mathbf{P}}_0^N:D_E[0,\infty)\rightarrow \mathcal{P}(\mathcal{P}^N(I))$, for $N \in \N$, are initial distributions of FVRE and MRE, respectively. Also, in the following, we assume
\begin{equation}
\begin{array}{l}
\mathbb{P}^{env}=\mathbb{P}^{env,e}\in \mathcal{P}(D_E[0,\infty))\\
\mathbb{P}^{env,N}= \mathbb{P}^{env,e_N}\in \mathcal{P}(D_E[0,\infty))
\end{array}
\end{equation}
are the laws of fitness processes $e$ and $e_N$, for $N\in \N$, respectively. We usually use the environment $e$ for FVRE and $e_N$ for MRE with $N$ individuals and assume $e_N$ converges to $e$ in Skororkhod topology. In particular, let $\mu_N^{e_N}$ be the unique solution to the quenched martingale $(\tilde{\mathcal{G}}^{N,e_N},\tilde{\mathcal{F}}_N,\tilde{\mathbf{P}}_0^N, \mathbb{P}^{env,N})=(\tilde{\mathcal{G}}^N,\tilde{\mathcal{F}}_N,\tilde{\mathbf{P}}_0^N, \mathbb{P}^{env,N})$ and $\mu^e$ be the unique solution to $(\tilde{\mathcal{G}}^e,\tilde{\mathcal{F}},\tilde{\mathbf{P}}_0, \mathbb{P}^{env})=(\tilde{\mathcal{G}},\tilde{\mathcal{F}},\tilde{\mathbf{P}}_0, \mathbb{P}^{env})$.
\begin{theorem}\label{Theorem-continuity of sample paths of FVRE}
Let $e$ be a fitness process and $\textbf{P}$ be the law of $\mu^e$ with the family of quenched measures $\{\textbf{P}^{\tilde{e}}\}_{\tilde{e}\in D_E[0,\infty)}$ for $\{\mu^{\tilde{e}}\}_{\tilde{e}\in D_E[0,\infty)}$. Then for $\mathbb{P}^{env}$-a.e. $\tilde{e} \in D_E[0,\infty)$, the process $\mu^{\tilde{e}}$ has continuous sample paths (in $\mathcal{C}([0,\infty),\mathcal{P}(I))$) $\textbf{P}^{\tilde{e}}$-a.s., that is
\begin{equation}
\mathbb{P}^{env}(\tilde{e}\in D_E[0,\infty): \textbf{P}^{\tilde{e}}(t\mapsto\mu^{\tilde{e}}(t) \ is \ continuous)=1)=1.
\end{equation}
Therefore,
\begin{equation}
\textbf{P}(t\mapsto\mu^e(t) \ is \ continuous)=1.
\end{equation}
\end{theorem}
\begin{theorem}\label{Theorem-convergence of Moran to FV by quench processes}
Suppose $\tilde{\mathbf{P}}_0$ and $\tilde{\mathbf{P}}_0^N$ are continuous, for any $N\in \N$.
\begin{enumerate}[(i)]
\item Let $\tilde{e}_N,\tilde{e}\in D_E[0,\infty)$, for $N\in \N$, such that $\tilde{e}_N\rightarrow \tilde{e}$ in $D_E[0,\infty)$. Then
\begin{enumerate}[a)]
\item If $\tilde{\textbf{P}}_0^N(\tilde{e}_N) \rightarrow \tilde{\textbf{P}}_0(\tilde{e})$ in $\mathcal{P}(\mathcal{P}(I))$, as $N\rightarrow \infty$, then $\mu_N^{\tilde{e}_N}\Rightarrow \mu^{\tilde{e}}$ in $D_{\mathcal{P}(I)}[0,\infty)$, as $N\rightarrow \infty$.
\item For $M\in \N$,
\begin{equation}
\mu_M^{\tilde{e}_N}\Rightarrow \mu_M^{\tilde{e}} \ \ as \ \ N\rightarrow \infty
\end{equation}
in $D_{\mathcal{P}^M(I)}[0,\infty)$.
\item $\mu^{\tilde{e}_N}\Rightarrow \mu^{\tilde{e}}$ in $D_{\mathcal{P}(I)}[0,\infty)$, as $N\rightarrow \infty$.
\end{enumerate}
\item Let $e$ and $\{e_N\}_{N \in \N}$ be fitness processes (not necessarily Markov) such that $e_N \rightarrow e$ in $D_E[0,\infty)$ a.s., as $N\rightarrow \infty$. Then
\begin{enumerate}[a)]
\item If $\tilde{\textbf{P}}_0^N(e_N) \rightarrow \tilde{\textbf{P}}_0(e)$ 
in $\mathcal{P}(\mathcal{P}(I))$, as $N\rightarrow \infty$, a.s. 
 then
\begin{equation}
\mu_N^{e_N}\Rightarrow \mu^e
\end{equation}
in $D_{\mathcal{P}(I)}[0,\infty)$, as $N\rightarrow \infty$.
\item For $M\in \N$,
\begin{equation}
\mu_M^{e_N}\Rightarrow \mu_M^e
\end{equation}
in $D_{\mathcal{P}^M(I)}[0,\infty)$, as $N\rightarrow \infty$.
\item $\mu^{e_N}\Rightarrow \mu^e$ in $D_{\mathcal{P}(I)}[0,\infty)$, as $N\rightarrow \infty$.
\end{enumerate}
\end{enumerate}
\end{theorem}
\begin{remark}
We summarize the convergence theorem in the following diagram. If  $e_M\rightarrow e$, as $M\rightarrow \infty$, a.s. in $D_E[0,\infty)$, then
\begin{equation}
\begin{array}{ccc}
\mu_N^{e_M} & \Rightarrow & \mu^{e_M}\\
\Downarrow & \rotmedgeq & \Downarrow \\
\mu_N^e & \Rightarrow & \mu^e
\end{array}
\end{equation}
as $M\rightarrow \infty$ and $N \rightarrow \infty$, appropriately.

\end{remark}

\begin{definition}
We say an $S$-valued Markov process $Z$ is weakly ergodic if there exists $m\in \mathcal{P}(S)$ such that for every initial distribution of $Z$
\begin{equation}
\lim\limits_{t \rightarrow \infty} \int\limits_{\Omega} f(Z_t(\omega))d \mathbb{P}=<m,f>,\ \ \ \ f \in \bar{\mathcal{C}}(S).
\end{equation}
In other words, letting $\{T_t^Z\}$ be the semigroup of $Z$ on $\bar{\mathcal{C}}(S)$, there exists $m\in \mathcal{P}(S)$ such that
\begin{equation}
\lim\limits_{t\rightarrow \infty} T_t^Z f(x)=<m,f>
\end{equation}
for any $x\in S$ and $f\in \bar{\mathcal{C}}(S)$.
\end{definition}
\begin{theorem}\label{Theorem-main ergodic theorem}
Suppose there exists a parent-independent component in the mutation process, i.e. $\beta'>0$, and let $e$ either be a stationary fitness process (not necessarily Markov) or a weakly ergodic Markov fitness with semigroup $\{T_t^{env}\}$ such that $T_t^{env}:\mathcal{C}(E) \rightarrow \mathcal{C}(E)$ for any $t\geq 0$
. Then the following statement holds.
\begin{enumerate}[(i)]
\item There exists a $\mathcal{P}(I)$-valued random variable $\mu^e(\infty)$ such that
\begin{equation}
\mu^e(t)\Rightarrow \mu^e(\infty)
\end{equation}
as $t\rightarrow \infty$, in $\mathcal{P}(I)$.
\item By assumption on weak ergodicity of $e$, there exists an $E$-valued random variable $e(\infty)$ such that the annealed-environment process converges weakly, that is
\begin{equation}
(\mu^e(t),e(t))\Rightarrow (\mu^e(\infty), e(\infty))
\end{equation}
as $t\rightarrow \infty$, in $\mathcal{P}(I)\times E$, and the law of $(\mu^e(\infty), e(\infty))$ is the unique invariant distribution of $(\mu^e(t),e(t))_{t\geq 0}$.
\end{enumerate}
\end{theorem}

The strategy to prove these theorems for the annealed processes $\mu^e$ and $\mu_N^{e_N}$ is to prove them  for quenched processes (processes with fixed environments), first, and then integrating over the elements of $D_E[0,\infty)$ we get the result for the annealed process. As each quenched (fixed) environment is a deterministic process and thus Markov, one can characterize the quenched processes as a quenched martingale problem in r.e. regardless of having Markovian property for the environments. This is one important advantage of this method. The technique that we apply is the combination of martingale problem and duality method. As the fitness process and hence the quenched generators depend on time, the dual process also must do so. Therefore, we need to understand the behaviour of the time-dependent dual process. The next section prepares some generalities about dual processes for time-inhomogeneous Markov processes.
\section{Duality method for stochastic processes in random environments}\label{Section-general duality}
One application of the duality method for martingale problems is to transform the uniqueness problem into the existence problem. Furthermore, in many cases, studying the dual is relatively simpler than studying the main Markov process. This gives more information about the main process which is harder to study directly. Duality method has been developed for many time-homogeneous Martingale problems. In this section, we extend the method of duality for time-inhomogeneous and quenched martingale problems, and generalize the notion of time-dependent Feynman-Kac duals, namely we study general duals in which an exponential term appears. In fact, some important aspects of duality relation only appears for time-inhomogeneous martingale problems. Roughly speaking, the evolution of the dual is backward in time with respect to the main Markov process. In practice, we usually cannot avoid appearance of Feynman-Kac term in dual processes. However, in the case of Fleming-Viot process when the fitness process is bounded, one can give duals in which there is no exponential term. In fact, when the fitness function (process) is unbounded, the existence of Feynman-Kac term is unavoidable
.

In this section we assume that $S'$ and $S_1$ are Polish metric spaces, $S_2$ is a separable metric space, and $\mathbb{P}^{env}$ is the law of the environment $X:\Omega \rightarrow S'$ which is an $S'$-valued random variable. Let $\mathcal{D}_i \subset \mathcal{B}(S_i)$ for $i=1,2$. We assume $\mathcal{G}^{i,t}:\R_+ \times \mathcal{D}_i \rightarrow \mathcal{B}(S_i)$ to be time-dependent linear operators, for $i=1,2$, and let $m_0^i \in \mathcal{P}(S_i)$. Also, we assume, for $i=1,2$ and for any real number $t\geq 0$, $G^{i,t}:\Omega\times \R_+\times \mathcal{D}_i \rightarrow \mathcal{B}(S_i)$ are operator processes with domain $\mathcal{D}_i$, and $\textbf{P}_0^i:S'\rightarrow \mathcal{P}(S_i)$ are measurable. Let $f\in \mathcal{B}(S_1\times S_2)$ be such that $f(.,v) \in \mathcal{D}_1$ and $f(u,.) \in \mathcal{D}_2$ for any $v \in S_2$ and $u \in S_1$. Let $g:\R_+ \times S_2 \rightarrow \R$ be a Borel measurable function. We start with the definition of duality for two families of time-inhomogeneous problems.

Two families of time-dependent martingale problems $\mathscr G^{*1}=\{(\mathcal{G}^{1,t}, \mathcal{D}_1, m_0^1)\}_{t\in \R_+}$ and $\mathscr G^{*2}=\{(\mathcal{G}^{2,t}, \mathcal{D}_2, m_0^2)\}_{t\in \R_+}$ are said to be dual with respect to $(f,g)$, if for each family of solutions $\{\zeta^t\}_{t\in \R_+}$ to the martingale problem $\mathscr G^{*1}$, with respective laws $\{m^{1,t}\}_{t\in \R_+}$, and each family of solutions $\{\xi^t\}_{t\in \R_+}$ to the martingale problem $\mathscr G^{*2}$, with respective laws $\{m^{2,t}\}_{t\in \R_+}$, we have 
\begin{equation}\label{FK-finite-1}
\int\limits_{S_1}\mathbb{E}_{m^{2,t}}[|f(u,\xi^t(t))| \exp \{\int\limits_0^t g(s,\xi^t(s))ds\}] m_0^1(du)] <\infty
\end{equation}
for any $t\in \R_+$, and
\begin{equation}
\int\limits_{S_2}\mathbb{E}_{m^{1,t}}[f(\zeta^t(t),v)] m_0^2(dv) =
\int\limits_{S_1}\mathbb{E}_{m^{2,t}}[f(u,\xi^t(t)) \exp \{\int\limits_0^t g(s,\xi^t(s))ds\}] m_0^1(du).
\end{equation}

We extend this idea to two families of quenched martingale problems in random environment. Let $M^i(x)\subset \mathcal{P}(S_i)$ be a collection of measures on $S_i$, for any $x\in S'$ and $i=1,2$. Set 
\begin{equation}
M^i:=\{\textbf{P}_0^i \in \mathscr{B}(S',\mathcal{P}(S_i)): \forall x\in S', \ \textbf{P}_0^i(x)\in M^i(x)\},
\end{equation}
where $\mathscr{B}(S',\mathcal{P}(S_i))$ is the set of all Borel measurable functions from $S'$ to $\mathcal{P}(S_i)$ for $i=1,2$.\\

Two families of quenched martingale problems in r.e. $X$, namely
\[\mathscr{G}^1=\mathscr{G}^1(M^1):=\{(G^{1,t}, \mathcal{D}_1, \textbf{P}_0, \mathbb{P}^{env})\}_{(t,\textbf{P}_0)\in \R_+\times M^1}\]
 and 
\[\mathscr{G}^2=\mathscr{G}^2(M^2):= \{(G^{2,t}, \mathcal{D}_2, \textbf{Q}_0, \mathbb{P}^{env})\}_{(t,\textbf{Q}_0)\in \R_+\times M^2},\]
are said to be (strongly) dual with respect to $(f,g)$, if for each family of solutions $\{\zeta^{t,\textbf{P}_0,X}\}_{(t,\textbf{P}_0)\in \R_+\times M^1}$ to $\mathscr G^1$, where each solution $\zeta^{t,\textbf{P}_0,X}$ has the family of quenched laws $\{\textbf{P}^{t,\textbf{P}_0,x}\}_{x\in S'}$, and for each family of solutions $\{\xi^{t,\textbf{Q}_0,X}\}_{(t,\textbf{Q}_0)\in \R_+\times M^2}$ to $\mathscr G^2$, where each solution $\xi^{t,\textbf{Q}_0,X}$ has the family of quenched laws $\{\textbf{Q}^{t,\textbf{Q}_0,x}\}_{x\in S'}$, we have:
\begin{equation}\label{FK-finite}
\mathbb{E}_{\mathbb{P}^{env}}[\int\limits_{S_1}\mathbb{E}_{\textbf{Q}^{t,\textbf{Q}_0,X}}[|f(u,\xi^{t,\textbf{Q}_0,X}(t))| \exp \{\int\limits_0^t g(s,\xi^{t,\textbf{Q}_0,X}(s))ds\}] \textbf{P}_0(X)(du)] <\infty
\end{equation}
for any $t \in \R_+$, $\textbf{P}_0\in M^1$, $\textbf{Q}_0\in M^2$ (Recall $\textbf{P}^{t,\textbf{P}_0,X}:\omega\mapsto \textbf{P}^{t,\textbf{P}_0,X(\omega)}$ and $\textbf{Q}^{t,\textbf{Q}_0,X}:\omega\mapsto \textbf{Q}^{t,\textbf{Q}_0,X(\omega)}$), and for $\mathbb{P}^{env}$-a.e. $x$ 
\begin{multline}
\int\limits_{S_2}\mathbb{E}_{\textbf{P}^{t,\textbf{P}_0,x}}[f(\zeta^{t,\textbf{P}_0,x}(t),v)] \textbf{Q}_0(x)(dv) = \\
\int\limits_{S_1}\mathbb{E}_{\textbf{Q}^{t,\textbf{Q}_0,x}}[f(u,\xi^{t,\textbf{Q}_0,x}(t)) \exp \{\int\limits_0^t g(s,\xi^{t,\textbf{Q}_0,x}(s))ds\}] \textbf{P}_0(x)(du)
\end{multline}
for any $t\in \R_+$, $\textbf{P}_0\in M^1$ and $\textbf{Q}_0\in M^2$.\\

We say they are dual in average if for any $t\geq 0$, $\textbf{P}_0\in M^1$ and $\textbf{Q}_0\in M^2$ (\ref{FK-finite}) holds, and
\begin{multline}
\mathbb{E}_{\mathbb{P}^{env}}[\int\limits_{S_2}\mathbb{E}_{\textbf{P}^{t,\textbf{P}_0,X}}[f(\zeta^{t,\textbf{P}_0,X}(t),v)] \textbf{Q}_0(X)(dv)] = \\
\mathbb{E}_{\mathbb{P}^{env}}[\int\limits_{S_1}\mathbb{E}_{\textbf{Q}^{t,\textbf{Q}_0,X}}[f(u,\xi^{t,\textbf{Q}_0,X}(t)) \exp \{\int\limits_0^t g(s,\xi^{t,\textbf{Q}_0,X}(s))ds\}] \textbf{P}_0(X)(du)]
\end{multline}
for any $t\in \R_+$, $\textbf{P}_0\in M^1$ and $\textbf{Q}_0\in M^2$.
\begin{remark}
For $i=1,2$, $x\in S'$ and $t\in \R_+$, recall that $G_x^{i,t}:\R_+ \times \mathcal{D}_i \rightarrow \mathcal{B}(S_i)$ is defined by
\begin{equation}
G_x^{i,t}(s,h):=G^{i,t}(X^{-1}(x),s,h).
\end{equation}
For $i=1,2$, $x\in S'$ and $m_0^i \in \mathcal{P}(S_i)$, let $\mathscr G_x^{*i}(m_0^i)$ be the family of martingale problems $\{(G_x^{i,t}, \mathcal{D}_i, m_0^i)\}_{t\in \R_+}$. In fact $\mathscr G^1(M^1)$ and $\mathscr G^2(M^2)$ are dual if and only if for $\mathbb{P}^{env}$-a.e. $x \in S'$, $\mathscr G_x^{*1}(m_0^1)$ and $\mathscr G_x^{*2}(m_0^2)$ are dual for any $m_0^1\in M^1(x)$ and $m_0^2\in M^2(x)$.
\end{remark}
When there exist a time-dependent operator $\mathcal{G}^1:\R_+ \times \mathcal{D}_1 \rightarrow \mathcal{B}(S_1)$ and an operator process
 $G^1:\Omega\times \R_+\times \mathcal{D}_1 \rightarrow \mathcal{B}(S_1)$
 such that for any $t\geq 0$,
 $\mathcal{G}^{1,t}=\mathcal{G}^1$ and $G^{1,t}=G^1$,
 all the martingale problems in the families $\mathscr G^{*1}$ and $\mathscr G^1$ coincide with the ones in the families $\{(\mathcal{G}^1, \mathcal{D}_1, \textbf{P}_0)\}_{\textbf{P}_0\in M^1}$ and $\{(G^1,\mathcal{D}_1, \textbf{P}_0 ,\mathbb{P}^{env})\}_{\textbf{P}_0\in M^1}$, respectively. Because of the importance of these special cases, we give their definitions separately as follows.
\begin{definition}\label{Duality for inhomogeneuos martingale problem}
Suppose $\mathcal{G}^1:\R_+ \times \mathcal{D}_1 \rightarrow \mathcal{B}(S_1)$ is a time dependent linear operator, and let $m_0^i \in \mathcal{P}(S_i)$ for $i=1,2$. The martingale problem $\mathscr G^{*1}=(\mathcal{G}^1, \mathcal{D}_1, m_0^1)$ and the family of martingale problems $\mathscr G^{*2}=\{(\mathcal{G}^{2,t}, \mathcal{D}_2, m_0^2)\}_{t\in \R_+}$ are said to be dual with respect to $(f,g)$, if for each solution $\zeta$ to the martingale problem $\mathscr G^{*1}$, with law $m^1$, and each family of solutions $\{\xi^t\}_{t\in \R_+}$ to the martingale problem $\mathscr G^{*2}$, with respective laws $\{m^{2,t}\}_{t\in \R_+}$, (\ref{FK-finite-1}) holds for any $t\in \R_+$, and
\begin{equation}
\int\limits_{S_2}\mathbb{E}_{m^1}[f(\zeta(t),v)] m_0^2(dv) = 
\int\limits_{S_1}\mathbb{E}_{m^{2,t}}[f(u,\xi^t(t)) \exp \{\int\limits_0^t g(s,\xi^t(s))ds\}] m_0^1(du)
\end{equation}
for any $t\in \R_+$.
\end{definition}
\begin{remark}
If, in addition, we assume that $\mathcal{G}^1:\mathcal{D}_1\rightarrow \mathcal{B}(S_1)$ and there exists a linear operator $\mathcal G^2: \mathcal{D}_2 \rightarrow \mathcal{B}(S_2)$
 such that for any $t\geq 0$,
 $\mathcal{G}^{2,t}=\mathcal{G}^2$, then the duality in Definition \ref{Duality for inhomogeneuos martingale problem} reduces to the classic time-homogeneous duality. In this case, it is still possible to find a family of time-dependent duals (not necessarily one dual).
\end{remark}

For the quenched martingale problem in random environment we have:
\begin{definition}\label{Duality}
Let $f$ and $g$ be as defined above and $G^1:\Omega\times \R_+\times \mathcal{D}_1 \rightarrow \mathcal{B}(S_1)$ be an operator process. We say a family of quenched martingale problems \[\mathscr{G}^1=\mathscr{G}^1(M^1)=\{(G^1,\mathcal{D}_1, \textbf{P}_0 ,\mathbb{P}^{env})\}_{\textbf{P}_0 \in M^1}\] and a family of quenched martingale problems \[\mathscr{G}^2=\mathscr{G}^2(M^2)=\{(G^{2,t}, \mathcal{D}_2, \textbf{Q}_0, \mathbb{P}^{env})\}_{(t,\textbf{Q}_0)\in \R_+\times M^2}\] are (strongly) dual with respect to $(f,g)$ if for each family of solutions $\{\zeta^{\textbf{P}_0,X}\}_{\textbf{P}_0\in M^1}$ to $\mathscr{G}^1$, with the respective families of quenched laws $\{\{\textbf{P}^{\textbf{P}_0,x}\}_{x\in S'}\}_{\textbf{P}_0\in M^1}$, and for each family of solutions $\{\xi^{t,\textbf{Q}_0,X}\}_{(t,\textbf{Q}_0)\in \R_+\times M^2}$ to $\mathscr{G}^2$,
 with respective families of quenched laws $\{\{\textbf{Q}^{t,\textbf{Q}_0,x}\}_{x\in S'}\}_{(t,\textbf{Q}_0)\in \R_+\times M^2}$, we have: for every $t \geq 0$, $\textbf{P}_0\in M^1$ and $\textbf{Q}_0\in M^2$, (\ref{FK-finite}) holds and for  $\mathbb{P}^{env}$-a.e. $x\in S'$
\begin{multline}
\int\limits_{S_2}\mathbb{E}_{\textbf{P}^{\textbf{P}_0,x}}[f(\zeta^{\textbf{P}_0,x}(t),v)] \textbf{Q}_0(x) dv = \\
\int\limits_{S_1}\mathbb{E}_{\textbf{Q}^{t,\textbf{Q}_0,x}}[f(u,\xi^{t,\textbf{Q}_0,x}(t)) \exp \{\int\limits_0^t g(s,\xi^{t,\textbf{Q}_0,x}(s))ds\}] \textbf{P}_0(x) du
\end{multline}
for any $t\in \R_+$, $\textbf{P}_0\in M^1$ and $\textbf{Q}_0\in M^2$.\\

They are said to be dual in average if (\ref{FK-finite}) holds for any $t\geq 0$, $\textbf{P}_0\in M^1$, $\textbf{Q}_0\in M^2$ and
\begin{multline}
\mathbb{E}_{\mathbb{P}^{env}}[\int\limits_{S_2}\mathbb{E}_{\textbf{P}^{\textbf{P}_0,X}}[f(\zeta^{\textbf{P}_0,X}(t),v)] \textbf{Q}_0(X) dv] = \\
\mathbb{E}_{\mathbb{P}^{env}}[\int\limits_{S_1}\mathbb{E}_{\textbf{Q}^{t,\textbf{Q}_0,X}}[f(u,\xi^{t,\textbf{Q}_0,X}(t)) \exp \{\int\limits_0^t g(s,\xi^{t,\textbf{Q}_0,X}(s))ds\}] \textbf{P}_0(X) du]
\end{multline}
for any $t\geq 0$, $\textbf{P}_0\in M^1$ and $\textbf{Q}_0\in M^2$.\\
\end{definition}

\begin{remark}\label{Remark-duality relations-quenched and time-dependnet}
For $i=1,2$, $x\in S'$, and $t\in \R_+$, as we already defined, let $G_x^1:\R_+ \times \mathcal{D}_1 \rightarrow \mathcal{B}(S_1)$ and $G_x^{2,t}:\R_+ \times \mathcal{D}_2 \rightarrow \mathcal{B}(S_2)$ with
\begin{equation}
G_x^1(s,h):=G^1(X^{-1}(x),s,h)
\end{equation}
and
\begin{equation}
G_x^{2,t}(s,h):=G^{2,t}(X^{-1}(x),s,h).
\end{equation}
For $i=1,2$, $x\in S'$ and $m_0^i \in \mathcal{P}(S_i)$, let $\mathscr G_x^{*1}(m_0^1):=(G_x^1, \mathcal{D}_1, m_0^1)$ and $\mathscr G_x^{*2}(m_0^2):=\{(G_x^1, \mathcal{D}_1, m_0^2)\}_{t\in \R_+}$. We have that $\mathscr G^1(M^1)$ and $\mathscr G^2(M^2)$ are dual if and only if for $\mathbb{P}^{env}$-a.e. $x \in S'$, $\mathscr G_x^{*1}(m_0^1)$ and $\mathscr G_x^{*2}(m_0^2)$ are dual for any $m_0^1\in M^1(x)$ and $m_0^2\in M^2(x)$.
\end{remark}


When the family of functions $\{f(.,v): v \in S_2\}$ is sufficiently nice, in other words measure-determining, the duality relation ensures the coincidence of the one-dimensional distributions of any two solutions of the martingale problem, which itself implies the uniqueness of finite dimensional distributions of those which is equivalent to well-posedness of the martingale problem. The following proposition transforms the problem of uniqueness for a martingale problem to the problem of existence of a dual process, or in other words, to the problem of existence of a dual martingale problem. This is a generalization of Lemma 5.5.1\cite{D93}  and Proposition 4.4.7 \cite{EK86-book}.

Let $\mathcal{P}_c(S_1)=\{m\in \mathcal{P}(S_1): m \ has \ a \ compact \ support\}$, and recall that, for $y\in S_2$, $\delta_y$ is the delta measure with the support on $\{y\}$. 

\begin{proposition}\label{Proposition-existence of dual gives uniqueness}
Suppose that, for any $m_0\in \mathcal{P}_c(S_1)$ and $y\in S_2$, the time-dependent martingale problem \[\mathscr{G}^{*1}(m_0)=(\mathcal{G}^1,\mathcal{D}_1,m_0)\] and the families of time-dependent martingale problems \[\mathscr{G}^{*2}(\delta_y)=\{(\mathcal{G}^{2,t},\mathcal{D}_2,\delta_y)\}_{t\in \R_+}\] are dual with respect to $(f,g)$. Consider a collection of measures $\mathcal{M}\subset \mathcal{P}(S_1)$ containing $\mathbb{P}\zeta(s)^{-1}$ for every $s\geq 0$ and every solution $\zeta$ of $\mathscr{G}^{*1}(m_0)$ with $m_0 \in \mathcal{P}_c(S_1)$. Suppose that $\{f(.,y):y\in S_2\}$ is measure-determining on $\mathcal{M}$. If for every $y\in S_2$ and $t\geq 0$ the martingale problem $(\mathcal{G}^{2,t},\mathcal{D}_2,\delta_y)$ has a solution, then for any initial distribution $m_0 \in \mathcal{P}(S_1)$, the time-dependent martingale problem $(\mathcal{G}^1,\mathcal{D}_1,m_0)$ has at most one solution (a unique solution).
\end{proposition}
\begin{proof}
For $m_0\in \mathcal{P}_c(S_1)$, let $\zeta$ and $\zeta'$ be two solutions to $\mathscr{G}^{*1}(m_0)$, and denote by $\xi^{t,y}$ an arbitrary solution to the martingale problem $(\mathcal{G}^{2,t},\mathcal{D}_2,\delta_y)$ for $t\geq 0$ and $y\in S_2$. By the duality relation
\begin{equation}
\mathbb{E}[f(\zeta(t),y)] =
\int\limits_{S_1}\mathbb{E}[f(u,\xi^{t,y}(t)) \exp \{\int\limits_0^t g(s,\xi^{t,y}(s))ds\}] m_0^1(du)=
\mathbb{E}[f(\zeta'(t),y)]
\end{equation}
that, as $\{f(.,y)\}_{y\in S_2}$ is measure-determining on $\mathcal{M}$, implies the uniqueness of one-dimensional distributions, i.e. $\mathbb{P}\zeta(t)^{-1}$ and $\mathbb{P}\zeta'(t)^{-1}$ coincide for any $t\in \R_+$. Hence, $\mathbb{P}\zeta^{-1}=\mathbb{P}\zeta'^{-1}$ that means uniqueness.

For general $m_0\in \mathcal{P}(S_1)$, let $\zeta$ and $\zeta'$ be solutions to the martingale problem $\mathscr{G}^{*1}(m_0)$, and let $K\subset S_1$ be compact with $m_0(K)>0$. Denote by $\zeta_K$ and $\zeta_K'$, the processes $\zeta$ and $\zeta'$ conditioned on $\{\zeta(0)\in K\}$ and $\{\zeta'(0)\in K\}$, respectively. It is clear that $\zeta_K$ and $\zeta_K'$ are solutions to the martingale problem $\mathscr{G}^{*1}(m_0(.|K))$ with
\begin{equation}
m_0(.|K)=\frac{m_0(.\cap K)}{m_0(K)}\in \mathcal{P}_c(S_1).
\end{equation}
Thus, as proved above, $\mathbb{P}\zeta_K(t)^{-1}=\mathbb{P}\zeta_K'(t)^{-1}$ which means
\begin{equation}
\mathbb{P}(\zeta(t)\in A|\zeta(0)\in K)=\mathbb{P}(\zeta'(t)\in A|\zeta'(0)\in K)
\end{equation}
for any Borel measurable subset $A$ of $S_1$. Since $S_1$ is a Polish space, from regularity of $m_0$, there exist a sequence of compact sets $(K_n)_{n\in \N}$ such that $m_0(K_n)\rightarrow 1$ as $n\rightarrow \infty$. Therefore,
\begin{multline}
\mathbb{P}(\zeta(t)\in A)=\lim\limits_{n\rightarrow \infty} \mathbb{P}(\zeta(t)\in A|\zeta(0)\in K_n)=
\lim\limits_{n\rightarrow \infty} \mathbb{P}(\zeta'(t)\in A|\zeta'(0)\in K_n)=\mathbb{P}(\zeta'(t)\in A)
\end{multline}
which implies uniqueness, by Proposition \ref{Proposition-uniqueness-one dimensional distributions}. 
\end{proof}

We easily generalize the last Proposition to the case of quenched martingale problems. For every $x\in S'$, let $M_c(x)=\mathcal{P}_c(S_1)$, and let $M_\delta(x)=\{\delta_y \in \mathcal{P}(S_2): y \in S_2\}$. We define $M_c$ and $M_\delta$ from $M_c(x)$ and $M_\delta(x)$ as already defined.
\begin{proposition}
Suppose that the families of quenched martingale problems \[\mathscr{G}^1(M_c)=\{(G^1,\mathcal{D}_1,\textbf{P}_0, \mathbb{P}^{env})\}_{\textbf{P}_0\in M_c}\] and \[\mathscr{G}^2(M_\delta)=\{(G^{2,t},\mathcal{D}_2,\textbf{P}_0,\mathbb{P}^{env})\}_{(t,\textbf{P}_0)\in \R_+ \times M_\delta}\] are dual with respect to $(f,g)$. Consider a collection of measures $\mathcal{M}\subset \mathcal{P}(S_1)$ such that for $\mathbb{P}^{env}$-a.e. $x\in S'$ contains $\mathbb{P}\zeta^x(s)^{-1}$ for all $s\geq 0$ and all solutions $\zeta^x$ of $(G_x^1,\mathcal{D}_1,\mathbb{P}\zeta^x(0)^{-1})$ for which $\mathbb{P}\zeta^x(0)^{-1}\in \mathcal{P}_c(S_1)$. Suppose that $\{f(.,y):y\in S_2\}$ is measure-determining on $\mathcal{M}$. If for $\mathbb{P}^{env}$-a.e. $x\in S'$, for every $y\in S_2$ and $t\geq 0$ the martingale problem $(G_x^{2,t},\mathcal{D}_2,\delta_y)$ has a solution, then for any initial distribution function $\textbf{P}_0:S' \rightarrow \mathcal{P}(S_1)$ the quenched martingale problem $(G^1,\mathcal{D}_1,\textbf{P}_0, \mathbb{P}^{env})$ has at most one solution (a unique solution).
\end{proposition}
\begin{proof}
First note that, as mentioned in Remark \ref{Remark-duality relations-quenched and time-dependnet}, $\mathscr{G}^1(M_c)$ and $\mathscr{G}^2(M_\delta)$ are dual with respect to $(f,g)$ if and only if, for $\mathbb{P}^{env}$-a.e. $x\in S'$, $\mathscr{G}_x^{*1}=(G_x^1,\mathcal{D}_1,m_0)$ and $\mathscr{G}^{*2}(\delta_y)=\{(G_x^{2,t},\mathcal{D}_2,\delta_y)\}_{t\geq 0}$ are dual with respect to $(f,g)$ for every $m_0\in \mathcal{P}_c(S_1)$ and $y\in S_2$. For any initial distribution function $\mathbf{P}_0:S'\rightarrow \mathcal{P}(S_1)$, the quenched martingale problem $(G^1,\mathcal{D}_1,\mathbf{P}_0,\mathbb{P}^{env})$ has at most one solution if and only if $(G_x^1,\mathcal{D}_1,\mathbf{P}_0(x))$ has at most one solution for $\mathbb{P}^{env}$-a.e. $x\in S'$. But the latter follows from  Proposition \ref{Proposition-existence of dual gives uniqueness} and this finishes the proof.
\end{proof}


Now we try to find conditions that guarantee the duality relation between two families of martingale problems. The following proposition is a natural extension of a theorem by D. Dawson and T. Kurtz \cite{DK82} to the case of time-dependent duality relations.

\begin{proposition}\label{Proposition-criteria for duality-time dependent}
Let $S_1$ and $S_2$ be two metric spaces, and let $\mathcal{D}_1\subset \mathcal{B}(S_1)$ and $\mathcal{D}_2\subset \mathcal{B}(S_2)$. Let $\mathcal{G}^1:\R_+\times \mathcal{D}_1 \rightarrow \mathcal{B}(S_1)$ and $\mathcal{G}^{2,t}:\R_+\times \mathcal{D}_2 \rightarrow \mathcal{B}(S_2)$, for $t\geq 0$, be time-dependent linear operators. Consider functions $g\in \mathcal{B}(\R_+\times S_2)$ and $f\in \mathscr{B}(S_1\times S_2)$ such that, for any $u\in S_1$ and $v\in S_2$, $f(.,v)\in \mathcal{D}_1$ and $f(u,.)\in \mathcal{D}_2$, and for any $t\geq s\geq 0$
\begin{equation}
\mathcal{G}^1(s)f\in \mathscr{B}(S_1\times S_2)
\end{equation}
and
\begin{equation}
\mathcal{G}^{2,t}(s)f\in \mathscr{B}(S_1\times S_2),
\end{equation}
where for $(u,v)\in S_1\times S_2$
\begin{equation}
\mathcal{G}^1(s)f(u,v):=\mathcal{G}^1(s)f(.,v)(u)
\end{equation}
and
\begin{equation}
\mathcal{G}^{2,t}(s)f(u,v):=\mathcal{G}^{2,t}(s)f(u,.)(v).
\end{equation}
Let $m_0^i\in \mathcal{P}(S_i)$, for $i=1,2$. Let $\zeta^1$ and $\zeta^{2,t}$ be solutions to martingale problems $(\mathcal{G}^1,\mathcal{D}_1,m_0^1)$ and $(\mathcal{G}^{2,t},\mathcal{D}_2,m_0^2)$, for any $t\geq 0$, respectively. Assume that for any $t\geq 0$, there exists an integrable random variable $C_t$ such that
\begin{enumerate}[(i)]
\item 
\begin{equation}
\sup\limits_{s,r\leq t} |f(\zeta^1(s),\zeta^{2,t}(r))|\leq C_t
\end{equation}
\item 
\begin{equation}
\sup\limits_{s,r\leq t} |\mathcal{G}^1(s)f(\zeta^1(s),\zeta^{2,t}(r))|\leq C_t
\end{equation}
\item\begin{equation}
\sup\limits_{s,r\leq t} |\mathcal{G}^{2,t}(r)f(\zeta^1(s),\zeta^{2,t}(r))|\leq C_t.
\end{equation}
\end{enumerate}
If, for any $t\geq 0$, for a.e. $0\leq s\leq t$
\begin{equation}\label{criteria for duality-time dependent}
\mathcal{G}^1(s)f(.,v)(u)=\mathcal{G}^{2,t}(t-s)f(u,.)(v)+g(t-s,v)
\end{equation}
for every $u\in S_1$ and every $v\in S_2$,
then
\[
\mathscr{G}^{*1}(m_0^1)=(\mathcal{G}^1,\mathcal{D}_1,m_0^1)
\]
and
\[
\mathscr{G}^{*2}(m_0^2)=\{(\mathcal{G}^{2,t},\mathcal{D}_2,m_0^2)\}_{t\geq 0}
\]
are dual with respect to $(f,g)$.
\end{proposition}

\begin{proof}
We assume $\zeta^1$ and $\{\zeta^{2,t}\}_{t\geq 0}$ are independent. For $t\geq 0$ and $s,r\leq t$ define
\begin{equation}
F(s,r)=\mathbb{E}[f(\zeta^1(s),\zeta^{2,t}(r)).\exp(\int\limits_0^rg(u,\zeta^{2,t}(u))du)].
\end{equation}
Therefore, by martingale property
\begin{equation}
\begin{array}{c}
F(s,r)-F(0,r) =  \int\limits_0^s \mathbb{E}[\mathcal{G}^1(u)f(\zeta^1(u),\zeta^{2,t}(r))\exp\{\int\limits_0^r g(\tilde{u},\zeta^{2,t}(\tilde{u}))d\tilde{u}\}]du\\
\end{array}
\end{equation}
Let $F_1$ and $F_2$ be partial derivatives of $F$. Then
\begin{equation}
F_1(s,r)=\mathbb{E}[\mathcal{G}^1(s)f(\zeta^1(s),\zeta^{2,t}(r))\exp\{\int\limits_0^r g(\tilde{u},\zeta^{2,t}(\tilde{u}))d\tilde{u}\}].
\end{equation}
We must also compute $F(s,r)-F(s,0)$. In order to do so, applying lemma $3.1.2$ in \cite{DK82}, for $h\geq 0$ with $r+h\leq t$, we can write
\begin{equation}
\begin{array}{l}
F(s,r+h)-F(s,r)=\\
\mathbb{E}[f(\zeta^1(s),\zeta^{2,t}(r+h))(\int\limits_r^{r+h} g(u,\zeta^{2,t}(u))\exp \{\int\limits_0^u g(\tilde u,\zeta^{2,t}(\tilde u))d\tilde u\}du)]+\\
\mathbb{E}[\int\limits_r^{r+h} \mathcal{G}^{2,t}(u)(\zeta^1(s),\zeta^{2,t}(u))du. \exp\{\int\limits_0^r g(\tilde u,\zeta^{2,t}(\tilde u))d\tilde u\}]=\\
\mathbb{E}[\int\limits_r^{r+h} f(\zeta^1(s),\zeta^{2,t}(u)).g(u,\zeta^{2,t}(u)).\exp(\int\limits_0^u g(\tilde{u},\zeta^{2,t}(\tilde{u}))d\tilde{u})du]+\\
\mathbb{E}[\int\limits_r^{r+h} \int\limits_u^{r+h} \mathcal{G}^{2,t}(v)f(\zeta^1(s),\zeta^{2,t}(v))dv.g(u,\zeta^{2,t}(u).\exp(\int\limits_0^u g(\tilde{u},\zeta^{2,t}(\tilde{u}))d\tilde{u})du]+\\
\mathbb{E}[\int\limits_r^{r+h} \mathcal{G}^{2,t}(u)f(\zeta^1(s),\zeta^{2,t}(u))\exp(\int\limits_0^u g(\tilde{u},\zeta^{2,t}(\tilde{u}))d\tilde{u})du]+\\
\mathbb{E}[\int\limits_r^{r+h} \mathcal{G}^{2,t}(u)f(\zeta^1(s),\zeta^{2,t}(u)).\\
\{\exp(\int\limits_0^r g(\tilde{u},\zeta^{2,t}(\tilde{u}))d\tilde{u})-\exp(\int\limits_0^u g(\tilde{u},\zeta^{2,t}(\tilde{u}))d\tilde{u})\}du].
\end{array}
\end{equation}
Under the assumptions above integrals exist, and the second and the forth terms in the last equation are bounded by 
\begin{equation}
\frac{1}{2}h^2\mathbb{E}[C_t].
\end{equation}
Writing $F(s,r)-F(s,0)$ as
\begin{equation}
\sum\limits_{i=1}^l (F(s,r_i)-F(s,r_{i-1}))
\end{equation}
for $l\in \N$ and an increasing sequence of real numbers $0=r_0<r_1<...<r_l=r$, and letting $l\rightarrow 0$ and $max_i(r_i-r_{i-1}) \rightarrow 0$, we get
\begin{equation}
\begin{array}{llr}
F(s,r)-F(s,0)=\\
\int\limits_0^r \mathbb{E}[\{f(\zeta^1(s),\zeta^{2,t}(u))g(u,\zeta^{2,t}(u))+\mathcal{G}^{2,t}(u)f(\zeta^1(s),\zeta^{2,t}(u))\}.\\
\exp\{\int\limits_0^u g(\tilde{u},\zeta^{2,t}(\tilde{u}))d\tilde{u}\}]du.
\end{array}
\end{equation}
Thus the partial derivative $F_2$ exist for a.e. $r\leq t$ and
\begin{equation}
\begin{array}{l}
F_2(s,r)=\\
\mathbb{E}[\{f(\zeta^1(s),\zeta^{2,t}(r))g(r,\zeta^{2,t}(r))+\mathcal{G}^{2,t}(r)f(\zeta^1(s),\zeta^{2,t}(r))\}.\\
\exp\{\int\limits_0^r g(\tilde{u},\zeta^{2,t}(\tilde{u}))d\tilde{u}\}].
\end{array}
\end{equation}
By Lemma $3.1.1$ \cite{DK82}
\begin{equation}
\begin{array}{l}
F(t,0)-F(0,t)=\\
\int\limits_0^t F_1(s,t-s)-F_2(s,t-s))ds=\\
\int\limits_0^t \{\mathbb{E}[\mathcal{G}^1(s)f(\zeta^1(s),\zeta^{2,t}(t-s))\exp\{\int\limits_0^{t-s} g(\tilde{u},\zeta^{2,t}(\tilde{u}))d\tilde{u}\}]-\\
\mathbb{E}[\{f(\zeta^1(s),\zeta^{2,t}(t-s))g(t-s,\zeta^{2,t}(t-s))+\mathcal{G}^{2,t}(t-s)f(\zeta^1(s),\zeta^{2,t}(t-s))\}.\\
\exp\{\int\limits_0^{t-s} g(\tilde{u},\zeta^{2,t}(\tilde{u}))d\tilde{u}\}]\}dr.

\end{array}
\end{equation}
But this vanishes for a.e. $t\geq 0$ and a.e. $0 \leq s\leq t$ by (\ref{criteria for duality-time dependent}). The statement follows, since $F(t,0)$ and $F(0,t)$ are continuous for $t\in \R_+$.

\end{proof}
The following is an automatic extension of the last proposition to the case of quenched martingale problems.
\begin{proposition}\label{Proposition-criteria for duality-quenched}
Let $S_1$ and $S_2$ be two metric spaces, and let $\mathcal{D}_1\subset \mathcal{B}(S_1)$ and $\mathcal{D}_2\subset \mathcal{B}(S_2)$. Let $G^1:\Omega \times \R_+\times \mathcal{D}_1 \rightarrow \mathcal{B}(S_1)$ and $G^{2,t}:\Omega \times \R_+\times \mathcal{D}_2 \rightarrow \mathcal{B}(S_2)$, for $t\geq 0$, be operator processes. Consider functions $g\in \mathcal{B}(\R_+\times S_2)$ and $f\in \mathscr{B}(S_1\times S_2)$ such that, for any $u\in S_1$ and $v\in S_2$, $f(.,v)\in \mathcal{D}_1$ and $f(u,.)\in \mathcal{D}_2$. Let $\mathbb{P}^{env}\in \mathcal{P}(S')$, and $m_0^i\in \mathcal{P}(S_i)$, for $i=1,2$. Suppose $\{\zeta^{1,x}\}_{x\in S'}$ and $\{\zeta^{2,t,x}\}_{x\in S'}$ are solutions to quenched martingale problems $(G^1,\mathcal{D}_1,m_0^1,\mathbb{P}^{env})$ and $(G^{2,t},\mathcal{D}_2,m_0^2,\mathbb{P}^{env})$, for any $t\geq 0$, respectively.
Assume that for $\mathbb{P}^{env}$-a.e. $x\in S'$, for any $t\geq s\geq 0$
\begin{equation}
G_x^1(s)f\in \mathscr{B}(S_1\times S_2)
\end{equation}
and
\begin{equation}
G_x^{2,t}(s)f\in \mathscr{B}(S_1\times S_2),
\end{equation}
and $\mathbb{P}^{env}$-a.s. for any $t\geq 0$, there exists an integrable random variable $C_t^x$ such that
\begin{enumerate}[(i)]
\item 
\begin{equation}
\sup\limits_{s,r\leq t} |f(\zeta^{1,x}(s),\zeta^{2,t,x}(r))|\leq C_t^x
\end{equation}
\item 
\begin{equation}
\sup\limits_{s,r\leq t} |G_x^1(s)f(\zeta^{1,x}(s),\zeta^{2,t,x}(r))|\leq C_t^x
\end{equation}
\item\begin{equation}
\sup\limits_{s,r\leq t} |G_x^{2,t}(r)f(\zeta^{1,x}(s),\zeta^{2,t,x}(r))|\leq C_t^x
\end{equation}
\end{enumerate}
If for $\mathbb{P}^{env}$-a.e. $x\in S'$, for any $t\geq 0$ and for a.e. $0\leq s\leq t$
\begin{equation}\label{criteria for duality-quenched}
G_x^1(s)f(.,v)(u)=G_x^{2,t}(t-s)f(u,.)(v)+g(t-s,v)
\end{equation}
for every $u\in S_1$ and every $v\in S_2$,
then
\[
\mathscr{G}^{1}(m_0^1)=(G^1,\mathcal{D}_1,m_0^1,\mathbb{P}^{env})
\]
and
\[
\mathscr{G}^{2}(m_0^2)=\{(G^{2,t},\mathcal{D}_2,m_0^2,\mathbb{P}^{env})\}_{t\geq 0}
\]
are dual with respect to $(f,g)$.

\end{proposition}
\begin{proof}
The proof is an automatic application of Proposition \ref{Proposition-criteria for duality-time dependent} and Remark \ref{Remark-duality relations-quenched and time-dependnet}.
\end{proof}

\section{A function-valued dual for FVRE}\label{Section-FV Duality}
The goal of this section is to construct a dual process in r.e. $e$ which is a fitness process (not necessarily Markov).
 Recall that $\mathbb{P}^{env}$ is the law of $e$. For any $\tilde{e}\in D_E[0,\infty)$, we define the quenched dual family $\{\psi^{t,\tilde{e}}\}_{t\in \R_+}$ of Markov processes with the deterministic environment $\tilde{e} \in D_E[0,\infty)$, where $\psi^{t,\tilde{e}}=(\psi_s^{t,\tilde{e}})_{s\in \R_+}$. The process $\psi^{t,\tilde{e}}$ is a Markov jump process with the state space $\mathcal{C}^*=\bigcup\limits_{n\in \N}\mathcal{C}_n(I^\N)$ without any jumps after time $t\in \R_+$, i.e. $\psi^t(t+s)=\psi^t(t)$ for any $s\geq 0$ (The process stays forever in its location at time $t$). Also, as before, we assume that $e$ is independent of Poisson times of jumps, the mutation kernel, and the initial distribution of the process. In order to define the transition probabilities of $\psi^{t,\tilde{e}}$ at times of jumps, we need the following notations. For $\underline{a}=(a_1,a_2,...)\in I^\N$, define the insertion function $\varrho_i^{ins}:I^\N \rightarrow I^\N$ to be
\begin{equation}
\varrho_i^{ins}(\underline{a})=(a_{j-\textbf{1}_{\{j>i\}}})_{j\in \N},
\end{equation}
where the value of $\textbf{1}_{\{j>i\}}$ is $1$ if $j>i$, and it is $0$, otherwise.
Also, the deletion function $\varrho_i^{del}:I^\N \rightarrow I^\N$ is defined by
\begin{equation}
\varrho_i^{del}(\underline{a})=(a_{j+\textbf{1}_{\{j>i\}}})_{j\in \N}.
\end{equation}
For $n\in\N$, the process $\psi^{t,\tilde{e}}$ jumps from state $f\in \mathcal{C}_n(I^\N) \subset \mathcal{C}^*$ to
\begin{equation}
f\circ \sigma_{ij}\circ \varrho_j^{ins} \ \ \ \ at \ rate \ \frac{\gamma}{2} \ for \ i,j=1,...,n \ (resampling)
\end{equation}
to
\begin{equation}
B_i'f \ \ \ \ at \ rate \ \beta' \ for \ i=1,..,n \ (parent-independent \ mutation)
\end{equation}
to
\begin{equation}
B_i''f \ \ \ \ at \ rate \ \beta'' \ for \ i=1,..,n \ (parent-dependent \ mutation)
\end{equation}
(Recall $\beta=\beta'+\beta''$ and also recall the definition from (\ref{mutation integral-2}) and (\ref{mutation integral-3})),\\
to
\begin{equation}\label{equation-dual selection }
\tilde{e}_i(t-s) f+(1-\tilde{e}_i(t-s))(f \circ \varrho_i^{del}) \ \ \ \ at \ rate \ \alpha \ for \ i=1,...,n \ (selection)
\end{equation}
for a jump occurring at time $s\leq t$.

Having the fitness Markov $e$ with law $\mathbb{P}^{env}$ which can be considered as a $D_E[0,\infty)$-valued random variable, we can think of a family of stochastic processes in random environment $e$, namely $\{\psi^{t,e}\}_{t\in \R_+}$. In fact, for $t\geq 0$, $\psi^{t,e}$ is a stochastic process in random environment $e$ whose quenched processes, $\{\psi^{t,\tilde{e}}\}_{\tilde{e}\in D_E[0,\infty)}$, are defined as above.

We define the duality function $\tilde H:\mathcal{P}(I)\times \mathcal{C}^* \rightarrow \R$ by
\begin{equation}
\tilde H(m,f)=<m^{\otimes \N}, f>
\end{equation}
for $f \in \mathcal{C}^*$ and $m\in \mathcal{P}(I)$.

For $f\in \mathcal{C}_n(I^\N) \subset \mathcal{C}^*$, in fact $\tilde{H}(m,f)=\tilde\Phi^f(m)$, where by definition $\Phi^f\in \tilde{\mathfrak{F}}$. Note that $\tilde{H}$ is a continuous function and hence measureable function but not bounded. The following is automatic.
\begin{proposition}\label{Proposition-H is measure-determining}
The collection of functions $\{\tilde H(.,f):f\in \mathcal{C}^*\}$ is measure-determining on $\mathcal{P}(I)$.
\end{proposition}
\begin{proof}
The set in the statement of proposition is in fact $\tilde{\mathfrak{F}}$ and, it was already proved that $\tilde{\mathfrak{F}}$ is measure-determining.
\end{proof}
Before proving the duality relation, we find the generator of $\psi^{t,\tilde{e}}$ for $t\geq 0$. For $\tilde{m}\in \mathcal{P}(I^\N)$, define $\phi^{\tilde{m}}\in \mathcal{C}(\mathcal{C}^*)$ (remember $\mathcal{C}^*$ is equipped with sup-norm topology) as
\begin{equation}
\phi^{\tilde{m}}(f)=<\tilde{m},f> \ \ f\in \mathcal{C}^*.
\end{equation}
From construction, for $t\geq 0$ and $\tilde{m}\in \mathcal{P}(I^{\N})$, the time-dependent generator of $\psi^{t,\tilde{e}}$ on function $\phi^{\tilde{m}}$, namely $\mathcal{G}_{\tilde{e}}^{*t}$, is computed as follows. For $0\leq s\leq t$ and $f\in \mathcal{C}_n(I^\N)$ for $n\in \N$
\begin{equation}
\begin{array}{l}
\mathcal{G}_{\tilde{e}}^{*t}(s)\phi^{\tilde{m}}(f)=\frac{\gamma}{2}\sum\limits_{0\leq i,j\leq n}<\tilde{m},f\circ \sigma_{ij}\circ \varrho_j^{ins}-f> \\
+\beta'\sum\limits_{i=1}^n<\tilde{m},B_i'f-f>+\beta''\sum\limits_{i=1}^n<\tilde{m},B_i''f-f>\\
+\mathcal{G}_{\tilde{e}}^{*t,sel}(s)\phi^{\tilde{m}}(f)
\end{array}
\end{equation}
To continue, we must compute the last term that is the generator of the dual process $\psi^{t,\tilde{e}}$ corresponding to the selection jumps (\ref{equation-dual selection }). Recall that the probability measure $P_{s,s+r}^*$ on $[s,s+r]$, for any $t-s\geq r\geq 0$, is the law of choosing one Poisson point in the interval $[s,s+r]$ conditioned on having only one Poisson point in that interval. For $x\in I^\N$, let
\begin{equation}
\overleftarrow{e}_i^{s,s+r}(x)=\int\limits_s^{s+r} \tilde{e}_i(t-u)(x) P_{s,s+r}^*(du).
\end{equation}
As before, since $\tilde{e}\in D_E[0,\infty)$, the left limit of $\tilde{e}$ exists for any $i\leq n$ and any time $s\in \R_+$ and
\begin{equation}
\inf\limits_{s\leq u\leq s+r}\tilde{e}_i(t-u)(x)\leq \overleftarrow{e}_i^{t,t+s}(x)\leq \sup\limits_{s\leq u\leq s+r}\tilde{e}_i(t-u)(x). 
\end{equation}
Therefore
\begin{equation}
\lim\limits_{r\rightarrow 0} \int\limits_s^{s+r} \tilde{e}(t-u)(.) P_{s,s+r}^*(du)=\tilde{e}(s^-)(.)
\end{equation}
pointwisely, (and furthermore in sup-norm topology), where
\begin{equation}
\begin{array}{l}
\tilde{e}_i(s^-)(x)=\lim\limits_{u\rightarrow s^-} \tilde{e}_i(u)(x).
\end{array}
\end{equation}
Thus
\begin{equation}
\begin{array}{l}
\mathcal{G}_{\tilde{e}}^{*t,sel}(s)\phi^{\tilde{m}}(f)=\lim\limits_{r\rightarrow 0} \frac{1}{r}\{\alpha r \sum\limits_{i=1}^n \int\limits_s^{s+r}<\tilde{m},\tilde{e_i}(t-u)f+\\
(1-\tilde{e}_i(t-u))f\circ \varrho_i^{del}> P_{s,s+r}^*(du)+(1-\alpha r)\sum\limits_{i=1}^n<\tilde{m},f>-\sum\limits_{i=1}^n <\tilde{m},f>\}=\\
\lim\limits_{r\rightarrow 0} \alpha\sum\limits_{i=1}^n<\tilde{m},f\int\limits_s^{s+r}\tilde{e_i}(t-u)P_{s,s+r}^*(du)+f\circ \varrho_i^{del} (1-\int\limits_s^{s+r}\tilde{e_i}(t-u)P_{s,s+r}^*(du)>)\\
-\alpha \sum\limits_{i=1}^n <\tilde{m},f>=\\
\alpha \sum\limits_{i=1}^n <\tilde{m},f\tilde{e}_i(s^-)+f\circ \varrho_i^{del}(1-\tilde{e}_i(s^-))-f>=\\
\alpha \sum\limits_{i=1}^n <\tilde{m},f\tilde{e}_i(s^-)-f\circ \varrho_i^{del}\tilde{e}_i(s^-)>+\alpha\sum\limits_{i=1}^n (<\tilde{m},f\circ \varrho_i^{del}>-<\tilde{m},f>).
\end{array}
\end{equation}
If $\tilde{m}=m^{\otimes\N}$ for $m\in \mathcal{P}(I)$, then the right hand side of the equality will be
\begin{equation}
\alpha \sum\limits_{i=1}^n <m^{\otimes\N},f\tilde{e}_i(s^-)-f\tilde{e}_{n+1}(s^-)>
\end{equation}
\begin{remark}
As we already mentioned, assuming that $e$ has sample paths in $D_E[0,\infty)$ is essential in order to compute the operator process. Also note that under this assumption the operator processes $(\mathcal{G}_s^e)_{s\geq 0}$ and $(\mathcal{G}_s^{*t,e})_{s\geq 0}$ take sample paths in $D_E[0,\infty)$, a.s..
\end{remark}

Applying above computations on $\tilde{H}(.,.)$, we have

\begin{equation}
\begin{array}{l}
\mathcal{G}_{\tilde{e}}^{*t}(s)\tilde{H}(m^{\otimes\N},.)(f)=\frac{\gamma}{2}\sum\limits_{0\leq i,j\leq n}<m^{\otimes\N},f\circ \sigma_{ij}\circ \varrho_j^{ins}-f> \\
+\beta'\sum\limits_{i=1}^n<m^{\otimes\N},B_i'f-f>+\beta''\sum\limits_{i=1}^n<m^{\otimes\N},B_i''f-f>\\
+\alpha \sum\limits_{i=1}^n <m^{\otimes\N},f\tilde{e}_i(s^-)-f\tilde{e}_{n+1}(s^-)>.
\end{array}
\end{equation}
Because of exchangeability, we can rewrite the first term of the last equation as
\begin{equation}
\frac{\gamma}{2}\sum\limits_{0\leq i,j\leq n}<m^{\otimes\N},f\circ \sigma_{ij}\circ \varrho_j^{ins}-f>=\frac{\gamma}{2}\sum\limits_{0\leq i,j\leq n}<m^{\otimes\N},f\circ \sigma_{ij}-f>.
\end{equation}
On the other hand, for the function $\tilde{H}(.,f)=\Phi^f(.)$, where $f\in \mathcal{C}_n(I^\N)$, we already saw that
\begin{equation}
\begin{array}{l}
\tilde{\mathcal{G}}_{\tilde{e}} \tilde{H}(.,f)(m)=\\
\frac{\gamma}{2}\sum\limits_{0\leq i,j\leq n}<m^{\otimes\N},f\circ \sigma_{ij}-f>+\\
+\beta'\sum\limits_{i=1}^n<m^{\otimes\N},B_i'f-f>+\beta''\sum\limits_{i=1}^n<m^{\otimes\N},B_i''f-f>\\
+\alpha \sum\limits_{i=1}^n <m^{\otimes\N},f\tilde{e}_i(s)-f\tilde{e}_{n+1}(s)>.
\end{array}
\end{equation}
Since $\tilde{e}$  is in $D_E[0,\infty)$, it is right continuous with left limit. Also the number of discontinuity points of $\tilde{e}$ is at most countable. For any $t\geq 0$, this yields the equality
\begin{equation}\label{the equality for operators of FVRE and dual-in quenched env}
\tilde{\mathcal{G}}_{\tilde{e}}(s)=\mathcal{G}_{\tilde{e}}^{*t}(t-s)
\end{equation}
for every $s\leq t$ except possibly at most countable points of discontinuity of $\tilde{e}$. Furthermore, constructing the corresponding operator process of $\psi^{t,\tilde{e}}$, namely $\tilde{\mathcal{G}}^{*t,e}$, which is consistent with $e$, for any $t\geq 0$, there exist at most countable times $s\leq t$ for which
\begin{equation}\label{the equality for operators of FVRE and dual}
\tilde{\mathcal{G}}^e(s)=\mathcal{G}^{*t,e}(t-s)
\end{equation}
does not hold almost surely.

In fact we can deduce the duality relation between FV in environment $\tilde{e}$ and jump Markov processes $\{\psi^{t,\tilde{e}}\}_{t\geq 0}$. Before doing this, we need to know an easy property of the dual processes whose proof will be postponed, namely, for $t\geq 0$, starting at the state $\hat{\psi}_0\in \mathcal{C}^*$, 
$\|\psi_s^{t,\tilde{e}}\|_{\infty}$ ($\|.\|_\infty$-supnorm on $\mathcal{C}^*$) remains bounded by $\|\hat{\psi}_0\|_{\infty}$ for any $s\geq 0$, a.s.. The following proposition states that for any $\tilde{e} \in D_E[0,\infty)$ and $t\geq 0$, conditioning on $\mu_0^{\tilde{e}}=m_0\in \mathcal{P}(I)$ and $\psi_0^{t,\tilde{e}}=\hat{\psi}_0\in \mathcal{C}^*$,   the duality relation holds for $\mu^{\tilde{e}}$ and $\{\psi^{t,\tilde{e}}\}_{t\in \R_+}$. 
\begin{proposition}\label{Proposition-duality relation FVRE}
For every $\tilde{e}\in  D_E[0,\infty)$, $m_0\in \mathcal{P}(I)$, and $\hat{\psi}_0\in \mathcal{C}^*$, the time-dependent martingale problem
\begin{equation}
(\tilde{\mathcal{G}}_{\tilde{e}},\tilde{\mathfrak{F}},\delta_{m_0})
\end{equation}
and one-parameter family of time-dependent martingale problems
\begin{equation}
\{(\mathcal{G}_{\tilde{e}}^{*t},\{<m^{\otimes \N},.>\}_{m\in \mathcal{P}(I)}, \delta_{\hat{\psi}_0})\}_{t\geq 0}
\end{equation}
are dual with respect to $(\tilde{H},0)$, that is for 
every $t\geq 0$
\begin{equation}
\mathbb{E}^{m_0}[\tilde H(\mu_t^{\tilde{e}},\hat{\psi}_0)]=\mathbb{E}^{\hat{\psi}_0}[\tilde H(m_0,\psi_t^{t,\tilde{e}})]
\end{equation}
\end{proposition}
\begin{remark}
The statement of the theorem is stronger than Definition \ref{Duality} as it guarantees the duality relation for every fixed (quenched) environment. Also, from the proof, it will be clear that, for any integrable $\mathcal{C}^*$-valued random variable $\underline{\psi}$, the duality relation holds between $(\tilde{\mathcal{G}}_{\tilde{e}},\tilde{\mathfrak{F}},\delta_{m_0})$ and
\begin{equation}
\{(\mathcal{G}_{\tilde{e}}^{*t},\{<m^{\otimes \N},.>\}_{m\in \mathcal{P}(I)}, \mathbb{P}\underline{\psi}^{-1})\}_{t\geq 0}.
\end{equation} 
\end{remark}
\begin{proof}
Boundedness of $\psi_s^{t,\tilde{e}}$ a.s. for any $t,s\geq 0$ yields that for any $t\geq 0$ and for any $s,r\leq t$, there exists a constant $C>1$ such that
\begin{equation}
<\mu_s^{\tilde{e}},\psi_r^{t,\tilde{e}}>\leq \|\hat{\psi}_0\|_{\infty} \leq C\|\hat{\psi}_0\|_{\infty} \ \ a.s.
\end{equation}
\begin{equation}
\begin{array}{l}
\|\tilde{\mathcal{G}}_{\tilde{e}}(s)\tilde{H}(\mu_s^{\tilde{e}},\psi_r^{t,\tilde{e}})\|_\infty\leq C <\mu_s^{\tilde{e}},\psi_r^{t,\tilde{e}}>\leq C\|\hat{\psi}_0\|_{\infty} \ \ a.s. 
\end{array}
\end{equation}
and
\begin{equation}
\|\tilde{\mathcal{G}}^{*t}_{\tilde{e}}(s)\tilde{H}(\mu_s^{\tilde{e}},\psi_r^{t,\tilde{e}})\|_\infty\leq C <\mu_s^{\tilde{e}},\psi_r^{t,\tilde{e}}>\leq C\|\hat{\psi}_0\|_{\infty} \ \ a.s. 
\end{equation}
since for any $f\in \mathcal{C}^*$, $\|f\circ \sigma_{ij}\circ \varrho_j^{ins}\|_{\infty}$, $\|B_i'f\|_{\infty}$, $\|B_i''f\|_{\infty}$, and $\|f\tilde{e}_i(s)-f\tilde{e}_{n+1}(s)\|_{\infty}$ are bonded by $\|f\|_{\infty}$ (cf. Proposition \ref{Proposition-dual-non-decreasing}). Therefore the assumptions of Proposition \ref{Proposition-criteria for duality-time dependent} hold. Then the statement of theorem follows from (\ref{the equality for operators of FVRE and dual-in quenched env}) (for every $t$ and a.e. $s\leq t$), and Proposition \ref{Proposition-criteria for duality-time dependent}.
\end{proof}
\begin{remark}\label{Remark-duality relation FVRE for times s and t}
The same argument as the one in the proof of the last proposition shows that, in fact the following more general duality relation is true. For $r\geq 0$, let $\tilde{\mathcal{G}}^{+r}:\Omega \times \R_+\times \tilde{\mathfrak{F}} \rightarrow \tilde{\mathfrak{F}}$ be an operator process defined by
\begin{equation}
\tilde{\mathcal{G}}^{+r}(\omega,s):=\tilde{\mathcal{G}}(\omega,s+r) \ for \ s\geq 0 \ and \ \omega\in\Omega.
\end{equation}
Then, the simple observation that for $t\geq 0$ and for a.e. $s\leq t$
\begin{equation}
\mathcal{G}^{*t}(s)=\tilde{\mathcal{G}}(t-s)=\tilde{\mathcal{G}}^{+r}(t-r-s) \ \ \ \ \mathbb{P}^{env}-a.s.
\end{equation}
shows that for every $\tilde{e}\in  D_E[0,\infty)$, $m_0\in \mathcal{P}(I)$, $r\geq 0$, and $\hat{\psi}_0\in \mathcal{C}^*$, the time-dependent martingale problem
\begin{equation}
(\tilde{\mathcal{G}}_{\tilde{e}}^{+r},\tilde{\mathfrak{F}},\delta_{m_0})
\end{equation}
and one-parameter family of time-dependent martingale problems
\begin{equation}
\{(\mathcal{G}_{\tilde{e}}^{*t},\{<m^{\otimes \N},.>\}_{m\in \mathcal{P}(I)}, \delta_{\hat{\psi}_0})\}_{t\geq r}
\end{equation}
are dual with respect to $(\tilde{H},0)$, that is for 
every $t\geq r$
\begin{equation}
\mathbb{E}^{m_0}[\tilde H(\mu_{t-r}^{+r,\tilde{e}},\hat{\psi}_0)]=\mathbb{E}^{\hat{\psi}_0}[\tilde H(m_0,\psi_{t-r}^{t,\tilde{e}})],
\end{equation}
where $\mu_s^{+r,\tilde{e}}=\mu_{s+r}^{\tilde{e}}$ for $s\geq 0$.
\end{remark}
The family of annealed stochastic processes $\{\psi^{t,e}\}_{t\in \R_+}$ is called the dual in r.e. $e$. Also for any $\tilde{e}\in D_E[0,\infty)$, the family of time-inhomogeneous Markov processes $\{\psi^{t,\tilde{e}}\}_{t\in \R_+}$ is called the dual in quenched environment (or with quenched fitness process) $\tilde{e}$.
\begin{proposition}\label{Proposition-uniqueness of FVRE martingale problem}
For any measurable map $$\tilde{\mathbf{P}}_0:D_E[0,\infty)\rightarrow \mathcal{P}(\mathcal{P}(E)),$$ the $(\tilde{\mathcal{G}},\tilde{\mathfrak{F}},\tilde{\textbf{P}}_0,\mathbb{P}^{env})$-martingale problem has at most one solution.
\end{proposition}
\begin{proof}
Stronger than the statement of the theorem, we show that for every $\tilde{e}\in D_E[0,\infty)$, the time-dependent martingale problem $(\tilde{\mathcal{G}}_{\tilde{e}},\tilde{\mathfrak{F}},m_0)$ for any $m_0\in \mathcal{P}(\mathcal{P}(I))$ is well-posed. Since $\mathcal{P}(I)$ is compact, this is equivalent to well-posedness of $(\tilde{\mathcal{G}}_{\tilde{e}},\tilde{\mathfrak{F}},\delta_{\nu})$ for every $\nu\in \mathcal{P}(I)$. The latter is an immediate consequence of the duality relation (Proposition \ref{Proposition-duality relation FVRE}), Proposition \ref{Proposition-H is measure-determining}, and Proposition \ref{Proposition-existence of dual gives uniqueness}.
\end{proof}
In order to prove an ergodic theorem for FVRE we study the long-time behaviour of the dual family.
\begin{proposition}\label{Proposition-dual absorbing time}
Suppose there exists a parent-independent mutation component, that is $\beta' >0$. Then there exists an almost surely finite random time $\tau$ at which, for every $t\geq 0$ and $\tilde{e}\in D_E[0,\infty)$,  $\psi_\tau^{t,\tilde{e}}$ does not depend on variables of $I^\N$, i.e. $\psi_\tau^{t,\tilde{e}}$ is a random constant function (a $\mathcal{C}_0(I^N)$-valued random variable), and $\tau$ is independent of $t$ and $\tilde{e}$.
\end{proposition}
\begin{proof}
First note that if there exists such a random time, then it is independent of the choice of $t\geq 0$ and $\tilde{e}\in D_E[0,\infty)$. This is true since the random time is only a function of Poisson jump processes which by assumption are independent of $e$ and $t$. So it is enough to show the existence of $\tau$ for an arbitrary quenched process $\psi=(\psi_s)_{s\geq 0}:=\psi_s^{t,\tilde{e}}$ for $\tilde{e}\in D_E[0,\infty)$ and $t\geq 0$. Note that constant functions , i.e. the elements of $\mathcal{C}_0(I^\N)$, are absorbing states. For an arbitrary initial state $f\in \mathcal{C}_n(I^{\N})$, we prove that there exists a random almost surely finite time at which the process hits $\mathcal{C}_0(I^{\N})$. The degree of a function in $\mathcal{C}^*$ is the maximum number of variables (possibly $0$) on which the function depends. Consider the natural surjective mapping from $\mathcal{C}^*$ onto  $\N\cup\{0\}$ that corresponds to each function in $\mathcal{C}^*$, its degree in $\N\cup\{0\}$. This mapping induces a continuous time random walk on the state space $\N\cup\{0\}$, more precisely defined by $\varphi_t=n$ if the degree of $\psi_t$ is $n$.
 for $t\geq 0$. Note that $\psi$ hits a constant if and only if $\varphi=(\varphi_s)_{s\geq 0}$ hits $0$. In fact, we can see that $\varphi$ is a birth-death process with a quadratic rate of death and a linear rate of birth. In order to see this, we need to determine the degree of all the states to which $\psi$ can jump from an arbitrary state $f\in \mathcal{C}_n(I^\N)$. It is clear that, at any time $s\leq t$, $f$ can jump only to states
\begin{equation}
\begin{array}{l}
f\circ\sigma_{ij}\circ\varrho_j^{ins} \in \mathcal{C}_{n-1} \ with \ rate \ \frac{\gamma}{2} \ \ for \  i,j=1,..,n\\
B_i'f\in \mathcal{C}_{n-1} \ \ \ with \ rate \ \beta' \ for \ i=1,..,n\\
B_i''f\in \mathcal{C}_{n} \ \ \ with \ rate \ \beta'' \ for \ i=1,..,n\\
\tilde{e}(t-s)f+(1-\tilde{e}(t-s))f \in \mathcal{C}_{n+1} \ \ \ with \ rate \ \alpha \ for \ i=1,..,n\\.
\end{array}
\end{equation}
Again, it is clear that time $s$ and environment $\tilde{e}$ do not affect on the birth and death rates. Let $g$ be a polynomial with degree $n$. There exists an $\tilde{n}\geq n$ such that $g\in \mathcal{C}_{\tilde{n}}(I^\N)$. Therefore, at a time of jump, a birth occurs at state $n$ with probability
\begin{equation}
\frac{n\alpha}{{\tilde{n} \choose 2}\frac{\gamma}{2}+n\beta'+n\alpha},
\end{equation}
and a  death occur with probability
\begin{equation}
\frac{{\tilde{n} \choose 2}\frac{\gamma}{2}+n\beta'}{{\tilde{n} \choose 2}\frac{\gamma}{2}+n\beta'+n\alpha}.
\end{equation}
Thus, for any $n\geq 2$, $\varphi$ starting in $n$ will hit $1$ in an a.s. finite random time. Suppose $\varphi$ does not hit $0$. Then the last line of argument implies that $\varphi$ hits $1$ infinitely many times without jumping afterwards to $0$. But this is not possible due to the existence of the parent-independent mutation component which gives a positive probability, $\frac{\beta'}{\beta'+\alpha}$, of jumps from $1$ to $0$.
\end{proof}
Similarly to \cite{DGP12}, we show that the dual process is non-increasing a.s..
\begin{proposition}\label{Proposition-dual-non-decreasing}
For any $\tilde{e}\in D_E[0,\infty)$ and $t\geq 0$, the dual process $\psi^{t,\tilde{e}}$, starting in $\hat{\psi}_0$, is non-increasing and bounded by $\|\hat{\psi}_0\|_{\infty}$ a.s..
\end{proposition}
\begin{proof}
Let $n\in \N$ and $f\in \mathcal{C}_n(I^\N)\subset \mathcal{C}^*$ be arbitrary. For any $i,j\leq n$ and $a\in I$, $f\circ \sigma_{ij}\circ \varrho_i^{ins}$, $f\circ \sigma_i^a$, and $f\circ \varrho_i^{del}$ are defined by setting restrictions on the first $n$ variables of $f$, that is they are restrictions of $f$ on a subdomain and therefore
\begin{equation}
\begin{array}{l}
\|f\circ \sigma_i^a\|_{\infty}\leq \|f\|_{\infty}\\
\|f\circ \varrho_i^{del}\|_{\infty}\leq \|f\|_{\infty}\\
\|f\circ\sigma_{ij}\circ\varrho_i^{ins}\|_{\infty}\leq \|f\|_{\infty}\\.
\end{array}
\end{equation}
Also
\begin{equation}
\begin{array}{l}
\|B'f\|_{\infty}=\\
\sup\limits_{x\in I^\N}|\int\limits_I f\circ \sigma_i^u(x)q'(du)|\leq\\
\int\limits_I \|f\circ\sigma_i^u\|_{\infty}q'(du)\leq \|f\|_{\infty}
\end{array}
\end{equation}
Similarly,
\begin{equation}
\begin{array}{l}
\|B''f\|_{\infty}\leq \|f\|_{\infty}.
\end{array}
\end{equation}
For a selection jump at time $t\geq 0$, for $x\in I^\N$, if $f\circ \varrho_i^{del}(x)\leq f(x)$, then
\begin{equation}
\tilde{e}_i(t-s)(x)f(x)+(1-\tilde{e}_i(t-s)(x))f\circ \varrho_i^{del}(x)\leq f(x)\leq \|f\|_{\infty},
\end{equation}
and if, $f(x)\leq f\circ \varrho_i^{del}(x)$, then
\begin{equation}
\begin{array}{l}
\tilde{e}_i(t-s)(x)f(x)+(1-\tilde{e}_i(t-s)(x))f\circ \varrho_i^{del}(x)\\
\leq f\circ \varrho_i^{del}(x)\leq \|f\circ \varrho_i^{del}\|_{\infty} \leq \|f\|_{\infty}.
\end{array}
\end{equation}
Thus,
\begin{equation}
\|\tilde{e}_i(t-s)f+(1-\tilde{e}_i(t-s))f\circ \varrho_i^{del}\|_{\infty}\leq \|f\|_{\infty}.
\end{equation}
Therefore, all jumps lead us to a function with a smaller sup-norm. In other words, $t\mapsto\|\psi_s^{t,\tilde{e}}\|$ is a non-increasing function, a.s.. In particular, for any $s\geq 0$, $\|\psi_s^{t,\tilde{e}}\|_\infty\leq \|\hat{\psi}_0\|_\infty$ a.s..
\end{proof}
To understand the long time behaviour of FVRE, we can study the long time behaviour of the dual process. We need the following lemma for this purpose.

\begin{lemma}\label{Lemma-weakly ergodicity and convergence of joint distributions}
Let $Z^{\nu_0}$ be an $S$-valued Markov process with initial distribution $\nu_0\in \mathcal{P}(S)$ and with homogeneous transition probability function $p(t,x,dy)$ whose semigroup is  denoted by $(T_t^Z:\bar{\mathcal{C}}(S)\rightarrow \bar{\mathcal{C}}(S))_{t\geq 0}$, with
\begin{equation}
T_t^Zf(x)=\int\limits_Sf(y)p(t,x,dy) \ for \ f\in \bar{C}(S).
\end{equation}
Assume that $Z^{\nu_0}$ takes its sample paths in $D_S[0,\infty)$ a.s.. Suppose $Z$ is weakly ergodic, i.e. there exists $\nu\in \mathcal{P}(S)$ such that for any $x\in S$ and $f\in \bar{\mathcal{C}}(S)$ we have $T_t^Zf(x)\rightarrow <\nu,f>$ as $t\rightarrow \infty$. Let $m^{\nu_0}$ be the law of $Z^{\nu_0}$ for $\nu_0 \in \mathcal{P}(S)$, and let $\tilde{Z}=Z^{\nu}$ and denote its law by $\tilde{m}$ (i.e. $\tilde{Z}$ is a stationary Markov process with law $\tilde m=m^{\nu}$). Denote by $m_{t_1,t_2,...,t_k}^{\nu_0}$ and $\tilde m_{t_1,t_2,...,t_k}$ the $k$-dimensional distributions of $m^{\nu_0}$ and $\tilde m$, respectively, for $k\in \N$ and real numbers $0 \leq t_1<t_2<...<t_k$. Then
\begin{enumerate}[(i)]
\item For any $\nu_0\in \mathcal{P}(S)$, any $k\in \N$ and any sequence of real numbers $0\leq t_1<t_2<...<t_k$
\begin{equation}
m_{t_1+s,t_2+s,...,t_k+s}^{\nu_0} \Rightarrow \tilde m_{t_1,t_2,...,t_k}
\end{equation}
as $s\rightarrow\infty$, for any $k\in \N$ and $t_1<...<t_k\in \R_+$.
\item In addition to above assumptions, let $S$ be compact, and $\overline{\mathcal{D}(G^Z)}$, where $G^Z$ is the generator of $Z$, contains an algebra that separates points and vanishes nowhere. Let $\hat{Z}^{\nu_0,t}:=(\hat{Z}_s^{\nu_0,t})_{s\geq 0}$, defined by $\hat{Z}_s^{\nu_0,t}:=Z_{s+t}^{\nu_0}$(the initial distribution of $\hat{Z}^{\nu_0,t}$ is the law of $Z_{t}^{\nu_0}$). Then for any $\nu_0$ the process $\hat{Z}^{\nu_0,t}$ weakly converges to $\tilde{Z}$ in $D_S[0,\infty)$ as $t\rightarrow\infty$.
\end{enumerate}
\end{lemma}
\begin{remark}
Recall that another equivalent definition of weak ergodicity for $Z$ is that there exists a probability measure $\nu \in \mathcal{P}(S)$ such that for every initial distribution $\nu_0$
\begin{equation}
\lim\limits_{t \rightarrow \infty} \mathbb{E}_{m^{\nu_0}}[f(Z_t^{\nu_0})]=<\nu,f> \ \ \ f \in \bar{\mathcal{C}}(S)
\end{equation}
\end{remark}
\begin{proof}
\begin{enumerate}[(i)]
\item We must prove for arbitrary $\nu_0\in \mathcal{P}(S)$, $k\in \N$, $0\leq t_1<t_2<...<t_k$ and $f\in \bar{\mathcal{C}}(S^n)$
\begin{equation}
\int\limits_{S^n} f dm_{t_1+s,t_2+s,...,t_k+s}^{\nu_0} \rightarrow \int\limits_{S^n} fd\tilde m_{t_1,t_2,...,t_k}.
\end{equation}
In order to prove the convergence, it suffices to prove it for a convergence-determining set of functions. In particular, we prove that convergence holds for
\begin{multline}
\{f\in \bar{\mathcal{C}}(S): f(x_1,x_2,...,x_k)=f_1(x_1)f_2(x_2)...f_k(x_k),\\
 \ k\in \N, \  f_i\in \bar{\mathcal{C}}(S), i=1,...,k\}
\end{multline}
Set $s_i:=t_{i+1}-t_i$ for $i=1,...,k-1$. For $f_1,....f_k\in \tilde{\mathcal{C}}(S)$
\begin{equation}
\begin{array}{l}
\int\limits_{S^n} f_1(x_1)...f_k(x_k) dm_{t_1+s,t_2+s,...,t_k+s}^{\nu_0}(dx_1,...,dx_k)=\\
\int\limits_{S} f_1(x_1)\int\limits_{S}f_2(x_2)\int\limits_{S}...\int\limits_{S}f_{k-1}(x_{k-1})\times\\
\int\limits_{S}f_k(x_k)p(s_{k-1},x_{k-1},dx_k)...p(s_1,x_1,dx_2)p(t_1+s,x_0,dx_1)\nu_0(dx_0)=\\
\int\limits_S T_s^Z T_{t_1}^Z(f_1(x)T_{s_1}^Z(...(f_{k-2}(x)T_{s_{k-2}}^Z(f_{k-1}(x)T_{s_{k-1}}^Z(f_k(x))))...))d\nu_0.
\end{array}
\end{equation}
Under the assumptions, the function
\begin{equation}
g(x):=T_{t_1}^Z(f_1(x)T_{s_1}^Z(...(f_{k-2}(x)T_{s_{k-2}}^Z(f_{k-1}(x)T_{s_{k-1}}^Z(f_k(x))))...))
\end{equation}
is continuous, therefore $<\nu_0,T_s^Zg> \rightarrow <\nu,g>$ as $s\rightarrow \infty$, by weak ergodicity of $Z$.
\item It suffice to prove the tightness (cf. Theorem 3.7.8 \cite{EK86-book}). But this follows Remark 4.5.2 in \cite{EK86-book} and the fact that the generators of $\hat{Z}^{\nu_0,t}$, for any $t\geq 0$ is identical to $G^Z$.
\end{enumerate}
\end{proof}
Now, we are ready to state a main tool to study the long-time behaviour of the FVRE. We do this by the study of long time behaviour of the dual processes.
\begin{theorem}\label{Theorem-dual long time behaviour}
Suppose that there exists a parent-independent component in the mutation process, i.e. $\beta'>0$, and let $e$ either be a stationary fitness process (not necessarily Markov) or a weakly ergodic Markov fitness with semigroup $\{T_t^{env}\}$ such that $T_t^{env}:\mathcal{C}(E) \rightarrow \mathcal{C}(E)$ for any $t\geq 0$
. Then, conditioning on $\psi_0^{t,e}=\hat{\psi}_0\in \mathcal{C}^*$, the limit
\begin{equation}
\lim\limits_{t\rightarrow \infty} \mathbb{E}^{\hat{\psi}_0}[f(\psi_t^{t,e})]
\end{equation}
exists for any $f\in \bar{\mathcal{C}}(\mathcal{C}^*)$ and is bounded by $\|\hat{\psi}_0\|_{\infty}$. In particular
\begin{equation}\label{limdual}
\lim\limits_{t\rightarrow \infty} \mathbb{E}^{\hat{\psi}_0}[\psi_t^{t,e}]
\end{equation}
exists (remember that $\psi_t^{t,e}$ hits a constant function, an absorbing state, in finite time a.s. and therefore (\ref{limdual}) is meaningful).
\end{theorem}
\begin{proof}
Let $J_s^{t,e}$ be the set of all times of Poisson jumps for $\psi^{t,e}$ up to time $s$, including resampling, mutation , and selection times of jumps, that $J_s^{t,e}$ has all information of Poisson point processes, but not any information about $e$. Because of stationarity, the process $(J_s^{t,e})_{s\geq 0}$ is independent of $t$ and also $e$ (by assumption). Therefore, it is convenient to drop $(t,e)$ from the superscript. Similarly, we define the following stochastic processes which are independent of $e$ and $t$ (for the same reason) and therefore we drop $(t,e)$ from the superscript again. Let $\kappa_s$ be the stochastic jump process counting the number of selective events of $\psi^{t,e}$ up to time $s\geq 0$, and let 
\begin{equation}
T_s^{sel}:=\{\varsigma_1^{sel}\leq \varsigma_2^{sel}\leq...\leq \varsigma_{\kappa_s}^{sel}\}
\end{equation}
be the times of selective events occurring for $\psi^{t,e}$ up to time $s$. As before, let $\tau$ be the stopping time at which $\psi^{t,e}$ hits a random constant function. Recall that (cf. Proposition \ref{Proposition-dual absorbing time}) $\tau$ is independent of $e$ and $t$, and it is a.s. finite. 
For any $t$ and $e$, $\psi_{\tau}^{t,e}$ is a random constant time whose value is a function, $F$, of $\tau$, $J_{\tau}$, $\kappa_{\tau}$, $T_{\tau}^{sel}$, and $\{e_{t-\varsigma_i^{sel}}\}_{i=1}^{\kappa_\tau}$. Specially, fixing $\tau=s$, $J_{\tau}=\hat{J}_s$, $\kappa_{\tau}=k$, $T_{\tau}^{sel}=\{t_1,...,t_k\}$, the function $F$ is continuous with respect to $e_{t-t_1},...,e_{t-t_k}$, i.e it is continuous with respect to variables of $E^k$. Let $(e_s^*)_{s\geq 0}$ be the stationary process generated by the semigroup $T^{env}$ and invariant initial distribution $\nu$ (In the case that $e$ is stationary, let $e_s^*=e_s$, for $s\geq 0$, and continue the same proof). As $e$ is weakly ergodic, by Lemma \ref{Lemma-weakly ergodicity and convergence of joint distributions}, for any continuous function $g\in \mathcal{C}(E^k)$
\begin{equation}
\begin{array}{l}
\lim\limits_{t\rightarrow \infty} \mathbb{E}[g(e_{t-t_k},e_{t-t_{k-1}},...,e_{t-t_1})]=\\
\mathbb{E}[g(e^*_0,e^*_{t_k-t_{k-1}},...,e^*_{t_k-t_1})].
\end{array}
\end{equation}
Since $e$ is independent of $\tau$, $J_\tau$, $\kappa_\tau$, $T_\tau^{sel}$ (by assumption), the conditional process $e$ given values of $J_\tau$, $\tau$, $\kappa_\tau$, $T_\tau^{sel}$ is still weakly ergodic, and hence
\begin{equation}\label{165}
\begin{array}{l}
\lim\limits_{t\rightarrow\infty}\mathbb{E}^{\hat{\psi}_0}[\psi_\tau^{t,e}|\tau,J_\tau,\kappa,T_\tau^{sel}]=\\
\lim\limits_{t\rightarrow\infty}\mathbb{E}^{\hat{\psi}_0}[F(e_{t-{\varsigma_k}},...,e_{t-{\varsigma_1}},\tau,J_\tau,\kappa,T_\tau^{sel})| \tau,J_\tau,\kappa,T_\tau^{sel}]=\\
\mathbb{E}^{\hat{\psi}_0}[F(e^*_0, e^*_{\varsigma_k- \varsigma_{k-1}}...,e^*_{\varsigma_k- \varsigma_1},\tau,J_\tau,\kappa,T_\tau^{sel})| \tau,J_\tau,\kappa,T_\tau^{sel}].
\end{array}
\end{equation}
Similarly, for any $f\in \bar{\mathcal{C}}(\mathcal{C}^*)$,
\begin{equation}\label{166}
\begin{array}{l}
\lim\limits_{t\rightarrow\infty}\mathbb{E}^{\hat{\psi}_0}[f(\psi_\tau^{t,e})|\tau,J_\tau,\kappa,T_\tau^{sel}]=\\
\lim\limits_{t\rightarrow\infty}\mathbb{E}^{\hat{\psi}_0}[f\circ F(e_{t-{\varsigma_k}},...,e_{t-{\varsigma_1}},\tau,J_\tau,\kappa,T_\tau^{sel})| \tau,J_\tau,\kappa,T_\tau^{sel}]=\\
\mathbb{E}^{\hat{\psi}_0}[f\circ F(e^*_0, e^*_{\varsigma_k- \varsigma_{k-1}}...,e^*_{\varsigma_k- \varsigma_1},\tau,J_\tau,\kappa,T_\tau^{sel})| \tau,J_\tau,\kappa,T_\tau^{sel}]
\end{array}
\end{equation}
Getting another expectation, knowing that $\tau$ is finite and $\|\psi_{\tau}^{t,e}\|\leq \|\hat{\psi}_0\|$ a.s. yields that
\begin{equation}\label{167}
\begin{array}{l}
\lim\limits_{t\rightarrow\infty}\mathbb{E}^{\hat{\psi}_0}[\psi_\tau^{t,e}]=\\
\mathbb{E}^{\hat{\psi}_0}[F(e^*_0, e^*_{\varsigma_k- \varsigma_{k-1}}...,e^*_{\varsigma_k- \varsigma_1},\tau,J_\tau,\kappa,T_\tau^{sel})],
\end{array}
\end{equation}
and hence the limit exists and is bounded by $\|\hat{\psi}_0\|_\infty$. Similarly,
\begin{equation}
\begin{array}{l}
\lim\limits_{t\rightarrow\infty}\mathbb{E}^{\hat{\psi}_0}[f(\psi_\tau^{t,e})]=\\
\mathbb{E}^{\hat{\psi}_0}[f\circ F(e^*_0, e^*_{\varsigma_k- \varsigma_{k-1}}...,e^*_{\varsigma_k- \varsigma_1},\tau,J_\tau,\kappa,T_\tau^{sel})]
\end{array}
\end{equation}
for $f\in \bar{\mathcal{C}}(\mathcal{C}^*)$.
\end{proof}

\section{Convergence of generators}\label{Section-convergence of generators}
This section is devoted to the convergence of generators of MRE to FVRE. Before setting the convergence of generator processes, we need to extend the generators of the measure-valued Moran processes in a convenient sense. The Moran process takes values in $\mathcal{P}^N(I)$. Also all functions in $\tilde{\mathfrak{F}}_N$ have domain $\mathcal{P}^N(I)$. On the other hand, the FV process is a $\mathcal{P}(I)$-valued Markov process, and $\mathcal{P}(I)$ is the domain of polynomials in $\tilde{\mathfrak{F}}$. In order to measure the distance of the elements of $\tilde{\mathfrak{F}}$ and $\tilde{\mathfrak{F}}_N$, for $N\in \N$, we need to extend the functions in the second algebra to take all measures of $\mathcal{P}(I)$.

Let $S'\subset S$, and consider time-dependent linear operators $A:\R_+\times \mathcal{D}(A)\rightarrow \mathcal{B}(S)$ and $B:\R_+\times \mathcal{D}(B)\rightarrow \mathcal{B}(S')$ with $\mathcal{D}(A)\subset \mathcal{B}(S)$ and $\mathcal{D}(B)\subset \mathcal{B}(S')$. For $f\in \mathcal{D}(A)$ and $g\in \mathcal{D}(B)$, let $g^f\in \mathcal{B}(S)$ be
\begin{equation}
g^f(x)=
\left\{
	\begin{array}{ll}
	g(x) & \mbox{if } x\in S'\\
     f(x)  &  \mbox{if } x\in S\setminus S'.
	\end{array}
\right.
\end{equation}
Set
\begin{equation}
D:=\mathcal{D}(B,A)=\{g^f\in \mathcal{B}(S): f\in \mathcal{D}(A),g\in\mathcal{D}(B)\},
\end{equation}
and define, for $t\geq 0$, the time-dependent linear operator $B^A:\R_+\times D\rightarrow \mathcal{B}(S)$ by
\begin{equation}
B^A(t)g^f(x)=
\left\{
	\begin{array}{ll}
	B(t)g(x) & \mbox{if } x\in S'\\
     A(t)f(x)  &  \mbox{if } x\in S\setminus S'.
	\end{array}
\right.
\end{equation}
It is clear that $(f,g) \mapsto g^f$ is bilinear with respect to the function addition. Moreover, if $\mathcal{D}(A)$ and $\mathcal{D}(B)$ are algebras, then so is $D$. The time-dependent linear operator $(B^A,D)$ is called the extension of $(B,\mathcal{D}(B))$ with respect to $(A,\mathcal{D}(A))$.

For a moment, denote by $\|.\|_{\infty}=\|.\|_{\infty}^S$ the sup-norm on $\mathcal{B}(S)$. We extend the notion of supnorm to restrictions of functions to a subdomain of $S$. More precisely, for 
\begin{equation}
h_1,h_2\in \bigcup\limits_{\emptyset\neq S'\subset \tilde{S}\subset S} \mathcal{B}(\tilde{S}),
\end{equation}
define
\begin{equation}
\|h_1-h_2\|_{\infty}^{S'}=\sup\limits_{x\in S'}|h_1(x)-h_2(x)|.
\end{equation}
Then, the following properties are trivial.
\begin{equation}
\begin{array}{l}
\|g^f-f\|_{\infty}=\|g-f\|_{\infty}^{S'}\\
\|B^Ag^f-Af\|_{\infty}=\|Bg-Af\|_{\infty}^{S'}
\end{array}
\end{equation}
Let $\iota_{S',S}$ be the natural embedding from $S'$ into $S$, and let $m_0\in \mathcal{P}(S')$, and, as before, denote by $\iota_{S',S}*m_0$ its push-forward measure under $\iota_{S',S}$. If an $S'$-valued measureable stochastic process $\zeta=(\zeta_t)_{t\geq 0}$ is a solution to the martingale problem $(B,\mathcal{D}(B),m_0)$, then its image under the natural embedding, $(\iota_{S',S}(\zeta_t))_{t\geq 0}$, is a solution to the martingale problem $(B^A,\mathcal{D}(B,A),\iota_{S',S}*m_0)$.

The following proposition is a generalization of Lemma $4.5.1$ \cite{EK86-book}.
\begin{proposition}\label{Proposition-convergence of generators-general}
Let $S$ be a separable metric space, $S_n\subset S$, $\mathcal{D} \subset \bar{\mathcal{C}}(S)$, $\mathcal{D}_n\subset \mathcal{B}(S_n)$, for $n\in \N$. Consider time-dependent linear operators
\begin{equation}
A:\R_+\times \mathcal{D} \rightarrow \bar{\mathcal{C}}(S)
\end{equation} 
and
\begin{equation}
A_n:\R_+ \times \mathcal{D}_n \rightarrow \mathcal{B}(S_n).
\end{equation}
Denote the sup-norm on $S$ by $\|.\|_{\infty}$, and let
\begin{equation}
\|.\|_{\infty,n}:=\|.\|_\infty^{S_n} \ \ for \ n\in\N,
\end{equation}
where the right side is defined as before. Let $m_0^n\in \mathcal{P}(S_n)$, for $n\in \N$, and $m_0 \in \mathcal{P}(S)$. Let $Z_n$ be a solution of the martingale problem $(A_n, \mathcal{D}_n, m_0^n)$ (with sample paths in $D_{S_n}[0,\infty)\subset D_{S}[0,\infty)$) for every $n\in \N$. Assume that, for any $f\in \mathcal{D}$, there exists a sequence $(f_n)_{n\in \N}$, with $f_n \in \mathcal{D}_n$ for every $n\in \N$, such that
\begin{enumerate}[(i)]
\item $\lim\limits_{n\rightarrow \infty}\|f_n-f\|_{\infty,n}=0$
\item $\lim\limits_{n\rightarrow \infty}\|A_n(s)f-A(s)f_n\|_{\infty,n}=0$ for a.e. $s\geq 0$.
\item For any $t\geq 0$
\begin{equation}
\sup\limits_{n\in\N, s\leq t}\|A_n(s)f_n\|_{\infty,n}< \infty.
\end{equation}
\end{enumerate}
If $Z$ is an $S$-valued stochastic process with the initial distribution $m_0$ such that $Z_n \Rightarrow Z$ in $D_S[0,\infty)$, as $n\rightarrow\infty$, then $Z$ is a solution of the martingale problem $(A,\mathcal{D},m_0)$.
\end{proposition}
\begin{proof}
Let $\tilde{Z}_n(t)=\iota_{S_n,S}(Z_n(t))$, for $t\geq 0$, where $\iota_{S_n,S}$ is the natural embedding from $S_n$ into $S$. As explained above, $\tilde{Z}_n$ is a solution to the martingale problem $(A_n^A,\mathcal{D}(A_n,A),\tilde{m}_0^n)$, for $n\in \N$, where $A_n^A$ is the extension of $A_n$ with respect to $A$ with the domain $\mathcal{D}(A_n,A)$ (defined as before), and $\tilde{m}_0^n=\iota_{S_n,S}*m_0^n$ is the image (push-forward measure) of $m_0^n$ under $\iota_{S_n,S}$. In order for $Z$ to be a solution to $(A,\mathcal{D},m_0)$, it is necessary and sufficient that for any $k\in \N$, $0\leq t_1\leq ...\leq t_k\leq t\leq s$, $h_i\in \bar{\mathcal{C}}(S)$, for $i=1,...,k$, and $f\in \mathcal{D}$
\begin{equation}
\begin{array}{l}
\mathbb{E}[\{f(Z(s))-f(Z(t))-\int\limits_t^s A(r)f(Z(r))dr\}\prod\limits_{i=1}^k h_i(Z(t_i))]=0.
\end{array}
\end{equation}
Under the assumptions, for any $f\in \mathcal{D}$, there exists a sequence of $f_n\in \mathcal{D}_n$ such that
\begin{equation}
\lim\limits_{n\rightarrow \infty}\|f_n^f-f\|_{\infty}=\lim\limits_{n\rightarrow \infty}\|f_n-f\|_{\infty,n}=0
\end{equation}
and
\begin{equation}
\lim\limits_{n\rightarrow \infty}\|A_n^A(s)f_n^f-A(s)f\|_{\infty}=\lim\limits_{n\rightarrow \infty}\|A_n(s)f_n-A(s)f\|_{\infty,n}=0
\end{equation}
for a.e. $s\geq 0$. Let $k\in \N$, $h_i\in \bar{\mathcal{C}}(S)$, for $i=1,...,k$. Let all $0\leq t_1\leq ...\leq t_k\leq t\leq s$ be in the times of continuity for $Z$, i.e. they belong to $\{r\geq 0: Z(r)=Z(r^-) \ a.s.\}$ which contains all positive real numbers except possibly at most countable ones. Since $\tilde{Z}_n\Rightarrow Z$ in $D_S[0,\infty)$, as $n\rightarrow \infty$, and $f$ is bounded continuous, we have
\begin{equation}
\lim\limits_{t\rightarrow \infty}\mathbb{E}[A_n^A(r)f_n^f(\tilde{Z}_n(r))]=\mathbb{E}[A(r)f(Z(r))]
\end{equation}
for a.e. $r\geq 0$ (for all $r\geq 0$ except possibly at most a countable number of points). Thus, by $(iii)$,
\begin{equation}
\lim\limits_{t\rightarrow \infty}\int\limits_t^s\mathbb{E}[A_n^A(r)f_n^f(\tilde{Z}_n(r))]dr=\int\limits_t^s\mathbb{E}[A(r)f(Z(r))]dr.
\end{equation}
Therefore, as $\tilde{Z}_n$ is a solution to $(A_n^A,\mathcal{D}(A_n,A),\tilde{m}_0^n)$ for any $n\in \N$,
\begin{equation}
\begin{array}{l}
0=\lim\limits_{n\rightarrow \infty}\mathbb{E}[\{f_n^f(\tilde{Z}_n(s))-f_n^f(\tilde{Z}_n(t))-\int\limits_t^s A_n^A(r)f_n^f(\tilde{Z}_n(r))dr\}\prod\limits_{i=1}^k h_i(Z(t_i))]=\\
\mathbb{E}[\{f(Z(s))-f(Z(t))-\int\limits_t^s A(r)f(Z(r))dr\}\prod\limits_{i=1}^k h_i(Z(t_i))]=0.
\end{array}
\end{equation}
Since $\tilde{Z}$ is right continuous a.s., by $(iii)$ and continuity of $f$, the last equality holds for any choice of $k\in \N$ and $t_1,...,t_k,t,s\geq 0$.
\end{proof}
\begin{remark}
One can replace $(ii)$ and $(iii)$ in the assumptions of the last proposition by
\begin{enumerate}
\item[(ii)'] There exists a measure-zero subset of $\R_+$, namely $U$, such that for any $t\geq 0$, $\lim\limits_{n\rightarrow \infty}\|A_n(s)f_n-A(s)f\|_{\infty,n}=0$ uniformly on $s\in [0,t]\cap U^c$.
\item[(iii)'] For any $t\geq 0$
\begin{equation}
\sup\limits_{s\leq t}\|A(s)f\|_{\infty}< \infty.
\end{equation}
\end{enumerate}
In fact $(ii)'$ and $(iii)'$ conclude $(ii)$ and $(iii)$.
\end{remark}
To apply the last proposition in our problem, we must verify the validity of the assumptions $(i)$,$(ii)$, and $(iii)$, for the generators of MRE and FVRE. We can see that the generators are uniformly bounded in a  very strong sense.

Recall that for $\tilde{e}\in D_E[0,\infty)$, $\tilde{\mathcal{G}}_{\tilde{e}}=\tilde{\mathcal{G}}^{res}+\tilde{\mathcal{G}}^{mut}+\tilde{\mathcal{G}}_{\tilde{e}}^{sel}$ and $\tilde{\mathcal{G}}_{\tilde{e}}^N=\tilde{\mathcal{G}}^{res,N}+\tilde{\mathcal{G}}^{mut,N}+\tilde{\mathcal{G}}_{\tilde{e}}^{sel,N}$ for $N\in \N$. Also, from now on in the rest of the paper, we denote by $\|.\|$ the supnorm on $\mathcal{B}(\mathcal{P}(I))$, and by $\|.\|_N:=\|.\|^{\mathcal{P}^N(I)}$ the supnorm on restrictions on $\mathcal{B}(\mathcal{P}^N(I))$. Also, similarly to the last section, we denote by $\|.\|_\infty$ the supnorm on $\mathcal{B}(I^\N)$, specially we use this notation for the functions on $\mathcal{C}^*$, i.e. the state space of the dual process, and we denote by $\|.\|_{\infty,N}=\|.\|_\infty^{I^N}$ the supnorm on restrictions on $\mathcal{B}(I^N)$.

\begin{proposition}\label{Proposition-Uniform-boundedness for Generator of FV}
For any $\tilde{\Phi}:=\tilde{\Phi}^f\in \tilde{\mathfrak{F}}$,
\begin{equation}
\sup\limits_{s\geq 0, \tilde{e}\in D_E[0,\infty)}\|\tilde{\mathcal{G}}_{\tilde{e}}(s)\tilde{\Phi}\|< \infty
\end{equation} 
\end{proposition}
\begin{proof}
Assume that $f\in \mathcal{C}_n(I^\N)$. As we saw in the proof of Proposition \ref{Proposition-dual-non-decreasing}, for any $i,j\leq n$ and $y\in I$, $f\circ \sigma_{ij}$ and $f\circ \sigma_i^y$ are defined by setting restrictions on the $n$ first variables of $f$, that is they are restrictions of $f$ on a subdomain, and therefore,
\begin{equation}
\begin{array}{l}
\|f\circ \sigma_i^y\|_{\infty}\leq \|f\|_{\infty}\\
\|f\circ\sigma_{ij}\|_{\infty}\leq \|f\|_{\infty}\\.
\end{array}
\end{equation}
Thus, for any $m\in \mathcal{P}(I)$, 
\begin{equation}
\begin{array}{l}
|\tilde{\mathcal{G}}^{res} \tilde{\Phi}^f(m)|=\\
\frac{\gamma}{2} |\sum\limits_{i,j=1}^n <m^{\otimes \N},f \circ \sigma_{ij}-f>|\leq\\
\frac{\gamma}{2}{n \choose 2}(\|f\circ \sigma_{ij}\|_\infty+\|f\|_\infty)\leq\\
\gamma{n \choose 2}\|f\|_\infty.
\end{array}
\end{equation}
Also,
\begin{equation}
\begin{array}{l}
|\tilde{\mathcal{G}}^{mut} \tilde\Phi^f(m)|= \\ =|\beta' \sum\limits_{i=1}^n <m^{\otimes \N},B_i^{'} f-f>+\beta'' \sum\limits_{i=1}^n <m^{\otimes \N},B_i^{''} f-f>|\leq\\
\beta' (\sum\limits_{i=1}^n (\|B_i^{'} f\|_\infty +\|f\|_\infty)+\beta'' (\sum\limits_{i=1}^n (\|B_i^{''} f\|_\infty +\|f\|_\infty)\leq\\
4\beta\|f\|_\infty,
\end{array}
\end{equation}
and for any $\tilde{e}\in D_E[0,\infty)$ and $s\geq 0$
\begin{equation}
\tilde{\mathcal{G}}_{\tilde{e}}^{sel}(s) \tilde\Phi^f(m)=\alpha \sum\limits_{i=1}^n <m^{\otimes \N},\tilde{e}_i(s) f - \tilde{e}_{n+1}(s)f>\leq 2n\alpha \|f\|_\infty.
\end{equation}
Therefore, there exists a constant $C>0$ such that for any $s\geq 0$ and $\tilde{e}\in D_E[0,\infty)$
\begin{equation}
|\tilde{\mathcal{G}}(s) \tilde\Phi^f(m)|<C\|f\|_\infty <\infty.
\end{equation}
\end{proof}

\begin{proposition}\label{Proposition-convergence of generators Moran to FV-uniform boundedness}
Let $\tilde{e}_N,\tilde{e}\in D_E[0,\infty)$, for $N\in \N$, such that $\tilde{e}_N\rightarrow \tilde{e}$ in $D_E[0,\infty)$, as $n\rightarrow \infty$. For any $\tilde{\Phi}:=\tilde{\Phi}^f\in \tilde{\mathfrak{F}}$, there exists a sequence of functions $\tilde{\Phi}_N:=\tilde{\Phi}_N^{f_N}\in \tilde{\mathfrak{F}}_N$, for $N\in \N$, such that 
\begin{enumerate}[(i)]
\item $\lim\limits_{N \rightarrow\infty}\|\tilde{\Phi}_N-\tilde{\Phi}\|_N=0$.
\item For a.e. $t\geq 0$
\begin{equation}
\lim\limits_{N\rightarrow \infty} \|\tilde{\mathcal{G}}_{\tilde{e}_N}^N(t)\tilde\Phi_N-\tilde{\mathcal{G}}_{\tilde{e}}(t)\tilde\Phi\|_N=0.
\end{equation}
\item 
\begin{equation}
\sup\limits_{N\in \N, s\geq 0}\|\tilde{\mathcal{G}}_{\tilde{e}_N}^N(s)\tilde{\Phi}_N\|_N< \infty     
\end{equation}
\end{enumerate}
\end{proposition}
\begin{proof}
To simplify the notation, when it comes to applying it in the proof, we denote
\begin{equation}
\tilde{e}_N^i(s)(x):=\tilde{e}_N(s)(x_i),
\end{equation}
for $s\geq 0$ and $x=(x_1,x_2,...)$ (or $x=(x_1,...,x_N)$). We assume $f\in \mathcal{C}_n(I^\N)$, for fixed $n\in \N$. Consider an arbitrary sequence of injective maps $\eta^N:I^N\rightarrow I^\N$, for $N\in \N$, such that each $\eta^N$ is identical on the first $N$ coordinates (e.g. $\eta^N:(x_1,x_2,...,x_N)\mapsto (y_1,y_2,...)$ where $x_i=y_i$ for $i\leq N$ and, for $i>N$, $y_i=c$ for a fixed $c\in I$). Note that, as we deal with the limit and the supremum, for $N\leq n$, neither the value of $\tilde{\Phi}_N$, nor the value of $\tilde{\mathcal{G}}_{\tilde{e}_N}^N \tilde{\Phi}_N$ are important. So we assume these functions are $0$ functions, for $N\leq n$. For $m\in \mathcal{P}^N(I)$, recall the definition of $m^{(N)}$ from (\ref{Convolution power-without replacement}) and set
\begin{equation}
\tilde{m}^{(N)}:=\eta^N*m^{(N)} \ \ N\in \N,
\end{equation}
that is the push-forward measure of $m^{(N)}$ under $\eta^N$. It is clear that for any function $g\in C(I^\N)$
\begin{equation}
\int\limits_{I^\N}gd\tilde{m}^{(N)}=\int\limits_{I^N}g\circ \eta^Ndm^{(N)}.
\end{equation}
As $f$ depends only on the first $n$ variables, we can define $\bar{f}\in \mathcal{C}(I^n)$ by
\begin{equation}
\bar{f}(x_1,x_2,...,x_n)=f(x_1,x_2,...,x_n,x_{n+1},...),
\end{equation}
for an arbitrary choice of $x_i$ for $i>n$. Let $f_N:=f\circ \eta^N$.
\begin{equation}
\begin{array}{l}
\|\Phi_N^{f_N}-\Phi^f\|_N=\\
\sup\limits_{m\in \mathcal{P}^N(I)} |<m^{(N)},f\circ \eta^N>-<m^{\otimes \N},f>|=\\
\sup\limits_{m\in \mathcal{P}^N(I)}|<\tilde{m}^{(N)},f>-<m^{\otimes \N},f>|=\\
\sup\limits_{m\in \mathcal{P}^N(I)}|\int\limits_{I^n}\bar{f}(x)m(dx_1)\prod\limits_{i=2}^n\frac{Nm(dx_i)-\sum\limits_{j=1}^{i-1}\delta_{x_j}}{N-i+1}-\int\limits_{I^n}\bar{f}(x)\prod\limits_{i=1}^nm(dx_i)|\leq\\
\|f\|_\infty \sup\limits_{m\in \mathcal{P}^N(I)} |\sum\limits_{(x_1,...,x_n)\in supp(m)^n}(m(x_1)\prod\limits_{i=2}^n\frac{Nm(x_i)-\sum\limits_{j=1}^{i-1}\delta_{x_j}}{N-i+1}-\prod\limits_{i=1}^nm(x_i))|\leq\\
\|f\|_\infty(\frac{C}{N}+o(\frac{1}{N}))
\end{array}
\end{equation}
for a constant $C$. This yields $(i)$.\\

To prove $(ii)$, first we observe that for $i,j\leq N$
\begin{equation}
f_N\circ\sigma_{ij}=f\circ\eta^N\circ\sigma_{ij}=f\circ\sigma_{ij}\circ\eta^N.
\end{equation}
Let $tr_{ij}$ be the transposition operator on $i$ and $j$. Since $f$ depends only on the first $n$ variables, for $N\geq n$

\begin{equation}\label{last term of resampling vanishes}
\begin{array}{l}
\tilde{\mathcal{G}}^{res,N}\tilde{\Phi}_N^{f_N}(m)=\\
\frac{\gamma}{2}\sum\limits_{i,j=1}^N<m^{(N)},f\circ\sigma_{ij}\circ\eta^N-f\circ\eta^N>=\\
\frac{\gamma}{2}\sum\limits_{i,j=1}^N<\tilde{m}^{(N)},f\circ\sigma_{ij}-f>=\\
\frac{\gamma}{2}\sum\limits_{i,j=1}^n<\tilde{m}^{(N)},f\circ\sigma_{ij}-f>+\\
\frac{\gamma}{2}\sum\limits_{i=n+1}^N\sum\limits_{j=1}^n<\tilde{m}^{(N)},f\circ\sigma_{ij}-f>=\\
\frac{\gamma}{2}\sum\limits_{i,j=1}^n<\tilde{m}^{(N)},f\circ\sigma_{ij}-f>+\\
\frac{\gamma}{2}\sum\limits_{i=n+1}^N\sum\limits_{j=1}^n<\tilde{m}^{(N)},f\circ tr_{ij}-f>.\\
\end{array}
\end{equation}
But the last term vanishes because of exchangeability of $m^{(N)}$. Therefore, similarly to the proof of $(i)$,
\begin{equation}
\begin{array}{l}
\|\tilde{\mathcal{G}}^{res,N}\tilde{\Phi}_N^{f_N}-\tilde{\mathcal{G}}^{res}\tilde{\Phi}^f\|_N\leq\\
\frac{\gamma}{2}\sup\limits_{m\in \mathcal{P}^N(I)}\sum\limits_{i,j=1}^n|<\tilde{m}^{(N)},f\circ\sigma_{ij}-f>-<m^{\otimes \N},f\circ\sigma_{ij}-f>|<\\
\frac{C_1}{N}+o(\frac{1}{N})
\end{array}
\end{equation}
for a constant $C_1>0$. Similarly, for mutation
\begin{equation}
\begin{array}{l}
\|\tilde{\mathcal{G}}^{mut,N}\tilde{\Phi}_N^{f_N}-\tilde{\mathcal{G}}^{mut}\tilde{\Phi}^f\|_N\leq\\
\sup\limits_{m\in \mathcal{P}^N(I)}\sum\limits_{i=1}^n|<\tilde{m}^{(N)},\beta'(B_i'f-f)+\beta''(B_i''f-f)>\\
-<m^{\otimes \N},\beta'(B_i'f-f)+\beta''(B_i''f-f)>|<\\
\frac{C_2}{N}+o(\frac{1}{N}).
\end{array}
\end{equation}
To verify this for the selection, first note that as $\tilde{e}_N\rightarrow\tilde{e}$ in $D_E[0,\infty)$, as $N\rightarrow \infty$, for every positive real number $t\geq 0$, except possibly a countable number of them, we have $\tilde{e}_N(t)\rightarrow\tilde{e}(t)$. The selection operator is very similar to the resampling one, except, here, the constant rate $\frac{\gamma}{2}$ is replaced by a time-dependent c\`adl\`ag fitness, and hence, the terms corresponding to $i=n+1,...N, j=1,...,n$ do not necessarily vanish.

More explicitly, for a continuity time $s\geq 0$ of  $\tilde{e}$
\begin{equation}
\begin{array}{l}
\frac{\alpha}{N}\sum\limits_{i=n+1}^N\sum\limits_{j=1}^n<\tilde{m}^{(N)},\tilde{e}_N^i(s)f\circ\sigma_{ij}-\tilde{e}_N^i(s)f>=\\
\frac{\alpha}{N}\sum\limits_{i=n+1}^N\sum\limits_{j=1}^n<\tilde{m}^{(N)},(\tilde{e}_N^j(s)f)\circ\sigma_{ij}-\tilde{e}_N^i(s)f>=\\
\frac{\alpha}{N}\sum\limits_{i=n+1}^N\sum\limits_{j=1}^n<\tilde{m}^{(N)},(\tilde{e}_N^j(s)f)\circ tr_{ij}-\tilde{e}_N^i(s)f>=\\
\frac{\alpha}{N}\sum\limits_{i=n+1}^N\sum\limits_{j=1}^n<\tilde{m}^{(N)},\tilde{e}_N^j(s)f-\tilde{e}_N^i(s)f>=\\
\frac{\alpha(N-n)}{N}\sum\limits_{j=1}^n<\tilde{m}^{(N)},\tilde{e}_N^j(s)f-\tilde{e}_N^{n+1}(s)f>.\\
\end{array}
\end{equation}
Therefore, for constants $C_3$ and $C_4$,
\begin{equation}
\begin{array}{l}
\|\tilde{\mathcal{G}}_{\tilde{e}_N}^{sel,N}(s)\tilde{\Phi}_N^{f_N}-\tilde{\mathcal{G}}_{\tilde{e}}^{sel}(s)\tilde{\Phi}^{f}\|_N=\\
\sup\limits_{m\in \mathcal{P}^N(I)}|\frac{\alpha}{N}\sum\limits_{i,j=1}^N<\tilde{m}^{(N)},\tilde{e}_N^i(s)(f\circ\sigma_{ij}-f)>\\
-\alpha\sum\limits_{i=1}^n<m^{\otimes N},\tilde{e}_i(s)f-\tilde{e}_{n+1}(s)f>|\leq\\
\sup\limits_{m\in \mathcal{P}^N(I)}|\frac{\alpha}{N}\sum\limits_{i,j=1}^n<\tilde{m}^{(N)},\tilde{e}_N^i(s)(f\circ\sigma_{ij}-f)>|+\\
\sup\limits_{m\in \mathcal{P}^N(I)}|\frac{\alpha}{N}\sum\limits_{i=n+1}^N\sum\limits_{j=1}^n<\tilde{m}^{(N)},\tilde{e}_N^i(s)f\circ\sigma_{ij}-\tilde{e}_N^i(s)f>\\
-\alpha\sum\limits_{i=1}^n<m^{\otimes N},\tilde{e}_i(s)f-\tilde{e}_{n+1}(s)f>|\leq\\
\frac{C_3}{N}+o(\frac{1}{N})+\\
\sup\limits_{m\in \mathcal{P}^N(I)}|\alpha\frac{N-n}{N}\sum\limits_{i=1}^n<\tilde{m}^{(N)},\tilde{e}_N^i(s)f-\tilde{e}_N^{n+1}(s)f>\\
-\alpha\sum\limits_{i=1}^n<m^{\otimes N},\tilde{e}_i(s)f-\tilde{e}_{n+1}(s)f>|,\\
\end{array}
\end{equation}
where the right hand side is bounded by
\begin{equation}
\begin{array}{l}
\frac{C_3}{N}+o(\frac{1}{N})+\\
\alpha\sum\limits_{i=1}^n\|f\|_\infty\{\|\tilde{e}_N^i(s)- \tilde{e}_i(s)\|_\infty+\|\tilde{e}_N^{n+1}(s)-\tilde{e}_{n+1}(s)\|_\infty\}\\
\end{array}
\end{equation}
Hence, there exists $C_5$ such that
\begin{equation}
\begin{array}{l}
\|\tilde{\mathcal{G}}_{\tilde{e}_N}^{sel,N}(s)\tilde{\Phi}_N^{f_N}-\tilde{\mathcal{G}}_{\tilde{e}}^{sel}(s)\tilde{\Phi}^{f}\|_N<\frac{C_5}{N}+o(\frac{1}{N})
\end{array}
\end{equation}
This finishes the proof of $(ii)$.

For part $(iii)$, similarly to the proof of Proposition \ref{Proposition-Uniform-boundedness for Generator of FV}, there exists a constant $C>0$ such that for any $m\in \mathcal{P}^N(I)$, $s\geq 0$
\begin{equation}
|\tilde{\mathcal{G}}_{\tilde{e}_N}^N(s)\Phi_N^{f_N}(m)|\leq C\|f_N\|_{\infty,N}=C\|f\circ \eta^N\|_{\infty,N}=C\|f\|_\infty
\end{equation}
\end{proof}
\begin{proposition}\label{Proposition-Convergenec of generators FV to FV different environments}
Let $\tilde{e}_N, \tilde{e} \in D_E[0,\infty)$, and suppose that $\tilde{e}_N \rightarrow \tilde{e}$ in $D_E[0,\infty)$. For any $\tilde{\Phi}=\tilde{\Phi}^f \in \tilde{\mathfrak{F}}$
\begin{equation}
\lim\limits_{N\rightarrow \infty} \|\tilde{\mathcal{G}}_{\tilde{e}_N}(s)\tilde\Phi-\tilde{\mathcal{G}}_{\tilde{e}}(s)\tilde\Phi\|=0
\end{equation}
for every $s\geq 0$ except possibly a countable number of real numbers. Moreover,
\begin{equation}
\sup\limits_{N\in\N, s\geq 0}\|\tilde{\mathcal{G}}_{\tilde{e}_N}(s)\tilde{\Phi}^f\|< \infty
\end{equation}
\end{proposition}
\begin{proof}
As $\tilde{e}_N\rightarrow \tilde{e}$ in $D_E[0,\infty)$, for every positive real numbers $s\geq 0$, except possibly countable ones $\tilde{e}_N(s)\rightarrow \tilde{e}(s)$ in $E$, as $N\rightarrow \infty$. The resampling and mutation rates, for both generators, are identical. Thus we need to verify that the limit is $0$ for the selection terms. To this end, for any $m\in \mathcal{P}(I)$, and any continuity point $s\geq 0$ of $\tilde{e}$,
\begin{equation}
\begin{array}{l}
\|\tilde{\mathcal{G}}_{\tilde{e}_N}(s)\tilde{\Phi}-\tilde{\mathcal{G}}_{\tilde{e}}(s)\tilde{\Phi}\|\leq\\
\sum\limits_{i=1}^n\sup\limits_{m\in\mathcal{P}(I)}|<m,\tilde{e}_N^i(s)f-\tilde{e}_N^{n+1}(s)f>-<m,\tilde{e}_i(S)f-\tilde{e}_{n+1}(s)f>|\\
\leq \sum\limits_{i=1}^n\ \|(\tilde{e}_N^i(s)f-\tilde{e}_i(s)f)-(\tilde{e}_N^{n+1}(s)f-\tilde{e}_{n+1}(s)f)\|_\infty\leq\\
\sum\limits_{i=1}^n\|f\|_\infty(\|\tilde{e}_N^i(s)-\tilde{e}_i(s)\|_\infty+\|\tilde{e}_N^{n+1}(s)-\tilde{e}_{n+1}(s)\|_\infty).
\end{array}
\end{equation}
The last term converges to $0$ and this yields the result.

For the second part, write
\begin{equation}
\begin{array}{l}
\sup\limits_{N\in\N, s\geq 0}\|\tilde{\mathcal{G}}_{\tilde{e}_N}(s)\tilde{\Phi}^f\|\leq \sup\limits_{\tilde{e}\in D_E[0,\infty), s\geq 0}\|\tilde{\mathcal{G}}_{\tilde{e}}(s)\tilde{\Phi}^f\| < \infty,
\end{array}
\end{equation}
where the last inequality follows Proposition \ref{Proposition-Uniform-boundedness for Generator of FV}.
\end{proof}
\begin{proposition}\label{Proposition-convergence of generators-uniform boundedness Moran process}
Let $M\in \N$, and let $\tilde{e}_N, \tilde{e} \in D_E[0,\infty)$, and suppose that $\tilde{e}_N \rightarrow \tilde{e}$ in $D_E[0,\infty)$. Then for any $\tilde{\Phi}_M=\tilde{\Phi}_M^f \in \tilde{\mathfrak{F}}_M$
\begin{equation}
\lim\limits_{N\rightarrow \infty} \|\tilde{\mathcal{G}}_{\tilde{e}_N}^M(s)\tilde\Phi_M-\tilde{\mathcal{G}}_{\tilde{e}}^M(s)\tilde\Phi_M\|_M=0
\end{equation}
for every $s\geq 0$ except possibly a countable number of real numbers. Moreover,
\begin{equation}
\sup\limits_{N\in \N, s\geq 0}\|\tilde{\mathcal{G}}_{\tilde{e}_N}^M(s)\tilde{\Phi}_M\|_M<\infty
\end{equation}
\end{proposition}
\begin{proof}
The proof is similar to the proof of Proposition \ref{Proposition-Convergenec of generators FV to FV different environments}. Again, resampling and mutation terms of both generators are the same, and for a continuity point $s\geq 0$ of the function $\tilde{e}$
\begin{equation}
\begin{array}{l}
\|\tilde{\mathcal{G}}_{\tilde{e}_N}^M(s)\tilde\Phi_M-\tilde{\mathcal{G}}_{\tilde{e}}^M(s)\tilde\Phi_M\|_M\leq\\
\sum\limits_{i,j=1}^N\|\tilde{e}_N^i(s)-\tilde{e}_i(s)\|_\infty\|f\circ\sigma_{ij}-f\|_{\infty}\leq\\
2(N-1)\sum\limits_{i=1}^N\|\tilde{e}_N^i(s)-\tilde{e}_i(s)\|_\infty\|f\|_{\infty},
\end{array}
\end{equation}
where the last term is converging to $0$ as $N\rightarrow \infty$.
As before, there exists a constant $C>0$ such that for any $m\in \mathcal{P}^M(I)$ and $s\geq 0$
\begin{equation}
|\tilde{\mathcal{G}}_{\tilde{e}_N}^M(s)\Phi_M(m)|\leq C\|f\|_{\infty,M}<\infty.
\end{equation}
\end{proof}
\section{Convergence of MRE to FVRE}\label{Section-proof of convergence of MRE to FVRE and wellposedness}
In this section we prove the wellposedness of the FVRE martingale problem (Theorem \ref{Theorem-Wellposedness of FVRE martingale problem}), and also prove convergence of MRE to FVRE (Theorem \ref{Theorem-convergence of Moran to FV by quench processes}). In the previous sections, we prepared all necessary tools to construct FVRE from MRE. We proved the uniqueness of the FVRE martingale problem, convergence of generators, and other required properties. What remained is the proof of tightness that is relatively simple, due to compactness of the state spaces $I$ and $E$ and uniform boundedness of the fitness process. This section essentially is devoted to the problem of tightness, and proves it for $\{\mu_N^{\tilde{e}_N}\}_{N\in \N}$, $\{\mu^{\tilde{e}_N}\}_{N\in \N}$, and $\{\mu_M^{\tilde{e}_N}\}_{N\in \N}$, where $\tilde{e}$ and $\tilde{e}_N$, for $N\in \N$, are c\`adl\`ag functions in $D_E[0,\infty)$ and $M\in \N$. We apply a modification of Remark $4.5.2$ \cite{EK86-book} which best fits our problem.
\begin{lemma}\label{Lemma-criteria for tightness-general}
Let $S$ be a Polish space, and $S_n\subset S$. Consider $\mathcal{D}\subset \bar{\mathcal{C}}(S)$ that contains an algebra that separates points and vanishes nowhere. Let $\mathcal{D}_n\subset\mathcal{B}(S_n)$ and consider time-dependent linear operators
\begin{equation}
A_n:\R_+\times \mathcal{D}_n\rightarrow \mathcal{B}(S_n)
\end{equation}
for $n\in \N$. For any $N\in \N$, suppose there exists an $S_n$-valued solution $Z_n=(Z_n(s))_{s\geq 0}$ (with sample paths in $D_{S_n}[0,\infty)\subset D_S[0,\infty)$) to the martingale problem $(A_n, \mathcal{D}_n,m_0^n)$ where $m_0^n\in \mathcal{P}(S_n)$. Assume:
\begin{enumerate}[(i)]
\item For any $f\in \mathcal{D}$, there exists a sequence $(f_n)_{n\in \N}$, $f_n\in \mathcal{D}_n$, such that
\begin{equation}
\|f_n-f\|_{\infty,n}\rightarrow 0,
\end{equation}
as $n\rightarrow \infty$ (here $\|.\|_{\infty,n}$ is the general norm $\|.\|_\infty^{S_n}$ defined in the beginning of Section \ref{Section-convergence of generators}).
\item For any $t\geq 0$, there exists $r>1$, such that
\begin{equation}
\limsup\limits_n\mathbb{E}[(\int\limits_0^t|A_n(s)f_n(Z_n(s))|^rds)^{\frac{1}{r}}]<\infty.
\end{equation}
\item (Compact containment condition)\\
For any $\varepsilon, t>0$, there exists a compact set $K_{\varepsilon,t}\subset S$ such that
\begin{equation}
\inf\limits_n \mathbb{P}[Z_n(s)\in K_{\varepsilon,t} \ \ for \ 0\leq s\leq t\}\geq 1-\varepsilon
\end{equation}
Then $\{Z_n\}_{n\in \N}$ is relatively compact (equivalently tight) in $D_S[0,\infty)$.
\end{enumerate}
\end{lemma}
\begin{proof}
The set $\bar{\mathcal{D}}$ contains an algebra that separates points and vanishes nowhere, and hence it is dense in $\bar{\mathcal{C}}(S)$ in the topology of uniform convergence on compact sets. As the compact containment condition holds, and $Z_n$ takes sample paths in $D_S[0,\infty)$ for any $n$, applying Theorem $3.9.1$ \cite{EK86-book}, it suffices to show for any $f\in \mathcal{D}$, $f\circ Z_n=(f\circ Z_n(s))_{s\geq 0}$ is tight in $D_\R[0,\infty)$. Theorem $3.9.4$ \cite{EK86-book} gives certain criteria under which $f\circ Z_n$ is tight, namely for any $t\geq 0$
\begin{equation}
\limsup\limits_n \mathbb{E}[\sup\limits_{s\in [0,t]\cap \mathbb{Q}}|f_n(Z_N(s))-f(Z_n(s))]<\|f_n-f\|_{\infty,n}.
\end{equation}
But the last term converges to $0$ as $n\rightarrow \infty$. Thus $(ii)$ and the fact that $Z_n$ are solutions to martingale problems $(A_n,\mathcal{D}_n,m_0^n)$ yields the result.
\end{proof}
\begin{lemma}\label{Lemma-tightness of Moran and FV}
For any $M\in \N$ and $(\tilde{e}_N)_{N\in \N}\subset D_E[0,\infty)$, the sequences $(\mu_N^{\tilde{e}_N})_{N\in \N}$ and $(\mu^{\tilde{e}_N})_{N\in \N}$ are tight in $D_{\mathcal{P}(I)}[0,\infty)$, and $(\mu_M^{\tilde{e}_N})_{N\in \N}$ is tight in $D_{\mathcal{P}^M(I)}[0,\infty)$.
\end{lemma}
\begin{proof}
First note that the compact containment condition always holds due to the compactness of the state space $\mathcal{P}(I)$. Propositions \ref{Proposition-convergence of generators Moran to FV-uniform boundedness}, \ref{Proposition-Convergenec of generators FV to FV different environments}, and \ref{Proposition-convergence of generators-uniform boundedness Moran process} guarantee the conditions $(i)$ and $(ii)$ of Lemma \ref{Lemma-criteria for tightness-general}. Of course, condition $(i)$ holds for $\{\mu^{\tilde{e}_N}\}_{N\in \N}$ by setting $\mathcal{D}_N=\tilde{\mathfrak{F}}$ and $\tilde{\Phi}^{f_N}= \tilde{\Phi}$ for any $N\in \N$. Similarly, in the case of $\{\mu_M^{\tilde{e}_N}\}_{N\in \N}$, we set  $\mathcal{D}_N=\tilde{\mathfrak{F}}_M$, for any $N\in \N$, and the sequence of functions, $\Phi_M^{f_N}$, all identical to $\tilde{\Phi}_M$. Further, $\tilde{\mathfrak{F}}$ and $\tilde{\mathfrak{F}}_N$ separate points and vanishes nowhere (Propositions \ref{Proposition-algebra separates points-FV}, \ref{Proposition-algebra separates points-Moran}) on $\mathcal{P}(I)$ and $\mathcal{P}^N(I)$, respectively. This finishes the proof.
\end{proof}
Now we are ready to prove Theorems \ref{Theorem-Wellposedness of FVRE martingale problem} and \ref{Theorem-convergence of Moran to FV by quench processes}.

\begin{flushleft}
\textbf{Proof of Theorems \ref{Theorem-Wellposedness of FVRE martingale problem} and \ref{Theorem-convergence of Moran to FV by quench processes}}
\end{flushleft}
For Theorem \ref{Theorem-Wellposedness of FVRE martingale problem}, it suffices to  prove existence of a solution to the martingale problem $(\tilde{\mathcal{G}}_{\tilde{e}},\tilde{\mathfrak{F}},\tilde{P}_0)$ for every $\tilde{e}$ (uniqueness of such a martingale problem has been proved (cf. Proposition \ref{Proposition-uniqueness-one dimensional distributions}). Then wellposedness of the quenched martingale problem follows immediately. Also, for Theorem \ref{Theorem-convergence of Moran to FV by quench processes}, integrating over $\Omega$, part $(ii)$ follows part$(i)$, automatically. Therefore, we concentrate on the proof of existence of $(\tilde{\mathcal{G}}_{\tilde{e}},\tilde{\mathfrak{F}},\tilde{P}_0)$, for any $\tilde{e}\in D_E[0,\infty)$, and proof of part $(i)$ in Theorem \ref{Theorem-convergence of Moran to FV by quench processes}.

Let $M$ be an arbitrary natural number. By assumption and continuity of $\tilde{P}_0$ and $\tilde{P}^M_0$, we have the convergence of initial distributions, i.e.
\begin{equation}
\begin{array}{l}
\tilde{P}_0^N(\tilde{e}_N)\rightarrow \tilde{P}_0(\tilde{e}) \ \ \ \ in \ \mathcal{P}(\mathcal{P}(I))\\
\tilde{P}_0(\tilde{e}_N)\rightarrow \tilde{P}_0(\tilde{e}) \ \ in \ \ \ \mathcal{P}(\mathcal{P}(I))\\
\tilde{P}_0^M(\tilde{e}_N)\rightarrow \tilde{P}_0^M(\tilde{e}) \ \ in \ \ \ \mathcal{P}(\mathcal{P}^M(I)).\\
\end{array}
\end{equation}

Propositions \ref{Proposition-convergence of generators-general}, \ref{Proposition-convergence of generators Moran to FV-uniform boundedness}, \ref{Proposition-Convergenec of generators FV to FV different environments}, and \ref{Proposition-convergence of generators-uniform boundedness Moran process} combined with the uniqueness of martingale problems $(\tilde{\mathcal{G}}_{\tilde{e}}, \tilde{\mathfrak{F}},\tilde{P}_0(\tilde{e}))$, $(\tilde{\mathcal{G}}_{\tilde{e}_N}^M, \tilde{\mathfrak{F}}_M,\tilde{P}_0^M(\tilde{e}))$ (Propositions \ref{Proposition-uniqueness of FVRE martingale problem} and \ref{Proposition-well-posedness of Moran martingale problem}) ensure that any convergent subsequence of $(\mu_N^{\tilde{e}_N})_{N\in \N}$ ($(\mu^{\tilde{e}_N})_{N\in \N}$ and $(\mu_M^{\tilde{e}_N})_{N\in \N}$, respectively) converges weakly in the corresponding Skorokhod topology to the unique solution of time-dependent martingale problem $(\tilde{\mathcal{G}}_{\tilde{e}}, \tilde{\mathfrak{F}},\tilde{P}_0(\tilde{e}))$ ($(\tilde{\mathcal{G}}_{\tilde{e}}, \tilde{\mathfrak{F}},\tilde{P}_0(\tilde{e}))$, $(\tilde{\mathcal{G}}_{\tilde{e}_N}^M, \tilde{\mathfrak{F}}_M,\tilde{P}_0^M(\tilde{e}))$, respectively). By tightness of $(\mu_N^{\tilde{e}_N})_{N\in \N}$ ($(\mu^{\tilde{e}_N})_{N\in \N}$ and $(\mu_M^{\tilde{e}_N})_{N\in \N}$, respectively) , Lemma \ref{Lemma-tightness of Moran and FV}, the weak limits exist. This yields part $(i)$ of Theorem \ref{Theorem-convergence of Moran to FV by quench processes}. Note that, for any $\tilde{e}\in D_E[0,\infty)$, letting $\tilde{e}_N=\tilde{e}$ (or $\tilde{e}_N=\tilde{e}_M$, respectively) for all $N\in \N$, we have also proved that
\begin{equation}
\begin{array}{l}
\mu_N^{\tilde{e}}\Rightarrow \mu^{\tilde{e}}\\
\mu_N^{\tilde{e}_M}\Rightarrow \mu^{\tilde{e}_M}
\end{array}
\end{equation}
in $D_{\mathcal{P}(I)}[0,\infty)$ as $N\rightarrow \infty$. Also, in particular for any $\tilde{e}\in D_E[0,\infty)$, setting $\tilde{e}_N=\tilde{e}$ for all $N\in \N$, implies existence of a solution to $(\tilde{\mathcal{G}}_{\tilde{e}},\tilde{\mathfrak{F}},\tilde{P}_0)$, and this, together with the uniqueness, Proposition \ref{Proposition-uniqueness of FVRE martingale problem}, deduces wellposedness.

Note that we do not need continuity of $\tilde{\mathcal{P}}_0$, except to prove convergence results $\mu^{\tilde{e}_N}\Rightarrow \mu$ and $\mu_M^{\tilde{e}_N}\Rightarrow \mu_M$ in the corresponding Skorokhod topology, as $N\rightarrow \infty$.
\section{Continuity of sample paths of FVRE}\label{Section-Continuity}
The purpose of this section is to prove the continuity of sample paths for the FVRE process.
We make use of the criteria developed recently by Depperschmidt et al. in \cite{DGP12} (see also \cite{Bakry-Emery}). To formulate the sufficient conditions under which FVRE takes continuous paths a.s., we shall introduce the concept of first and second order operators. We follow the definitions and the proof of Section 4 in \cite{DGP12}.

\begin{definition}\label{Definition-first and second order operator}
Let $L$ be a Banach space and suppose $\mathcal{D}\subset L$ contains an algebra $\mathcal{A}$. A linear operator $G$ on $L$ with the domain $\mathcal{D}$ is said to be a first order operator with respect to $\mathcal{A}$ if for any $f\in \mathcal{A}$
\begin{equation}
Gf^2-2fGf=0.
\end{equation}
It is said to be a second order operator, if it is not a first order operator, and for every $f\in \mathcal{A}$
\begin{equation}
Gf^3+3f^2Gf-3fGf^2=0.
\end{equation} 
\end{definition}

The following lemma is an extension of Proposition 4.5 in \cite{DGP12} to the case of time-dependent martingale problems.
\begin{lemma}\label{Lemma-Criteria for continuity of sample paths}
Let $S$ be a Polish space, and consider $\mathcal{D} \subset \bar{\mathcal{C}}(S)$ containing a countable algebra $\mathcal{A}$ that separates points of $S$ and contains constant functions. Let $G:\R_+\times \mathcal{D}\rightarrow \mathcal{B}(S)$ be a time-dependent linear operator such that, for any $t\geq 0$,
\begin{equation}
G(t)=G^1(t)+G^2(t),
\end{equation} 
where $G^1(t)$ and $G^2(t)$ are first and second order linear operators for every $t\geq 0$, respectively. Assume that for any $t\geq 0$ and $f\in \mathcal{A}$, $G(.)f$ is uniformly bounded on $[0,t]$, i.e.
\begin{equation}
\sup\limits_{s\leq t} |G(s)f|<\infty.
\end{equation}
Let $m_0 \in \mathcal{P}(S)$. Then, any general solution $Z=(Z_t)_{t\geq 0}$ of the martingale problem $(G,\mathcal{D},m_0)$, has sample paths in $\mathcal{C}([0,\infty),S)$ a.s..
\end{lemma}
\begin{remark}
Note that under the assumption, any general solution of the martingale problem $(G,\mathcal{D},m_0)$ is a solution, i.e. takes its sample paths in $D_S[0,\infty)$ a.s. (cf. Theorem 4.3.6 \cite{EK86-book}).
\end{remark}
\begin{proof}
We follow the proof of Proposition 4.5 in \cite{DGP12}. Let $Z=(Z_t)_{t\geq 0}$ be a solution to the martingale problem $(G,\mathcal{D},m_0)$. For any $f\in \mathcal{A}$, first we prove $(f(Z_t))_{t\geq 0}$ has continuous paths a.s.. For $f\in \mathcal{A}$ and $x,y\in S$, let $F_{f,y}(x):=f(x)-f(y)$. Since $\mathcal{A}$ is an algebra, for any $y\in S$ and $f\in \mathcal{D}$, $F_{f,y}\in \mathcal{A}$, and hence $F_{f,y}^2, F_{f,y}^4\in \mathcal{A}$. Thus
\begin{equation}
F_{f,y}^2(Z_t)-\int\limits_0^t G(u)F_{f,y}^2(Z_u)du
\end{equation}
is a martingale with respect to the canonical filtration. In particular, for $t\geq s$
\begin{equation}
\mathbb{E}[F_{f,Z_s}^2(Z_t)]=\int\limits_s^t \mathbb{E}[G(u)F_{f,Z_s}^2(Z_u)]du\leq C_1(t-s)
\end{equation}
for a constant $C_1>0$. Also,
\begin{equation}
\begin{array}{l}
\mathbb{E}[(f(Z_t)-f(Z_s))^4]=\\
\mathbb{E}[F_{f,Z_s}^4(Z_t)]=  \ \ (by \ Lemma \ 4.4. \ \cite{DGP12})\\
\int\limits_s^t\mathbb{E}[F_{f,Z_s}^2(Z_u)(6GF_{f,Z_s}^2(Z_u)-8F_{f,Z_s}(Z_u)GF_{f,Z_s}(Z_u))]du\leq\\
C_2\int\limits_s^t \mathbb{E}[F_{f,Z_s}^2(Z_u)]du\leq\\
C_1C_2\int\limits_s^t(u-s)du=\frac{C_1C_2(t-s)^2}{2}.
\end{array}
\end{equation}
Continuity of $(f(Z_t))_{t\geq 0}$ follows from Proposition $3.10.3$ \cite{EK86-book}. The remainder of the proof is identical to that of Lemma $4.4$ in (Depperschmidt et al. \cite{DGP12}).
\end{proof}
\begin{lemma}\label{Lemma-order of FVRE generator}
Let $\tilde{e}\in D_E[0,\infty)$. Then 
\begin{enumerate}[(a)]
\item $\tilde{\mathcal{G}}^{mut}$ is first order.
\item For any $t\geq 0$, $\tilde{\mathcal{G}}_{\tilde{e}}^{sel}(t)$ is first order.
\item $\tilde{\mathcal{G}}^{res}$ is second order.
\end{enumerate}
\end{lemma}
\begin{proof}
See Proposition $4.10$ in \cite{DGP12}.
\end{proof}
\begin{flushleft}
\textbf{Proof of Theorem \ref{Theorem-continuity of sample paths of FVRE}}\\
\end{flushleft}
Continuity of the sample paths of $\mu^e$ a.s. follows the continuity of sample paths of $\mu^{\tilde{e}}$ for every $\tilde{e}\in D_E[0,\infty)$. The latter is a consequence of Propositions \ref{Proposition-algebra separates points-FV}, \ref{Proposition-Uniform-boundedness for Generator of FV}, Theorem \ref{Theorem-Wellposedness of FVRE martingale problem}, and Lemmas \ref{Lemma-Criteria for continuity of sample paths} and \ref{Lemma-order of FVRE generator}.

\section{An ergodic theorem for FVRE}\label{Section-proof of ergodic theorem}
This section proves the main ergodic theorem, Theorem \ref{Theorem-main ergodic theorem}, for the FV annealed-environment process. Before giving a complete proof, we show that the semigroup of FV with any deterministic fitness process $\tilde{e}\in D_E[0,\infty)$ has Feller property, i.e. it is from $\mathcal{C}(\mathcal{P}(I))$ to $\mathcal{C}(\mathcal{P}(I))$. For any $0\leq s\leq t$ and $\tilde{e}\in D_E[0,\infty)$, let $p^{\tilde{e}}(s,x;t,dy)$ and $(T_{s,t}^{\tilde{e}})_{0\leq s\leq t}$ be, respectively, the transition probability and the semigroup of $\mu^{\tilde{e}}$, i.e. for $f\in \mathcal{C}(\mathcal{P}(I))$ and $m\in \mathcal{P}(I)$
\begin{equation}
T_{s,t}^{\tilde{e}}f(m)=\int\limits_{\mathcal{P}(I)}f(m')p^{\tilde{e}}(s,m;t,dm').
\end{equation}
\begin{proposition}
Let $\tilde e\in D_E[0,\infty)$ be a deterministic fitness process. Then, $(T_{s,t}^{\tilde e})_{0\leq s\leq t}$ is a Feller semigroup, i.e. for any $0\leq s\leq t$ and for any $f\in \mathcal{C}(\mathcal{P}(I))$, $T_{s,t}^{\tilde e}f\in \mathcal{C}(\mathcal{P}(I))$. In other words, for any $0\leq s\leq t$
\begin{equation}
T_{s,t}^{\tilde e}:\mathcal{C}(\mathcal{P}(I))\rightarrow \mathcal{C}(\mathcal{P}(I)).
\end{equation}
\end{proposition}
\begin{proof}
Recall from Remark \ref{Remark-duality relation FVRE for times s and t} the definition of $(\mu_r^{+s,\tilde e})_{r\geq 0}$:
\begin{equation}
\mu_r^{+s,\tilde e}=\mu_{r+s}^{\tilde e}
\end{equation} 
Let $m_0^n \rightarrow m_0$ in $\mathcal{P}(I)$, as $n\rightarrow \infty$, and $\hat{\psi}_0\in \mathcal{C}^*$. For any $0\leq s\leq t$, the duality relation follows Proposition \ref{Proposition-duality relation FVRE} and Remark \ref{Remark-duality relation FVRE for times s and t}, and since $\hat{\psi}_0$ is bounded continuous depending on a finite number of variables, we have
\begin{equation}
\begin{array}{l}
\lim\limits_{n\rightarrow \infty}\mathbb{E}^{m_0^n}[<\mu_{t-s}^{+s,\tilde e},\hat{\psi}_0>]=\lim\limits_{n\rightarrow \infty}\mathbb{E}^{\hat{\psi}_0}[<m_0^n,\psi_{t-s}^{t,\tilde e}>]=\\
\mathbb{E}^{\hat{\psi}_0}[<m_0,\psi_{t-s}^{t,\tilde e}>]=\mathbb{E}^{m_0}[<\mu_{t-s}^{+s,\tilde e},\hat{\psi}_0>].
\end{array}
\end{equation}
As $\{<.,\hat{\psi}_0>\}_{\hat{\psi}_0\in \mathcal{C}^*}$ is measure-determining, for any $f\in \mathcal{C}(\mathcal{P}(I))$
\begin{equation}
\begin{array}{l}
\lim\limits_{n\rightarrow \infty} T_{s,t}^{\tilde e}f(m_0^n)=
\lim\limits_{n\rightarrow \infty}\mathbb{E}^{m_0^n}[f(\mu_{t-s}^{+s,\tilde e})]=
\mathbb{E}^{m_0}[f(\mu_{t-s}^{+s,\tilde e})]=
T_{s,t}^{\tilde e}f(m_0).
\end{array}
\end{equation}
\end{proof}
\begin{flushleft}
\textbf{Proof of Theorem \ref{Theorem-main ergodic theorem}}
\end{flushleft}
By duality relation (in average), Proposition \ref{Proposition-duality relation FVRE}, and Theorem \ref{Theorem-dual long time behaviour}, for any $m_0\in \mathcal{P}(I)$ and $\hat{\psi}_0\in \mathcal{C}^*$
\begin{equation}
\lim\limits_{t\rightarrow \infty} \mathbb{E}^{m_0}[<\mu_t^e,\hat{\psi}_0>]=\lim\limits_{t\rightarrow \infty}\mathbb{E}^{\hat{\psi}_0}[<m_0,\psi_t^{t,e}>]=\lim\limits_{t\rightarrow \infty}\mathbb{E}^{\hat{\psi}_0}[\psi_t^{t,e}],
\end{equation}
where the limit on the right hand side of the last equality exists and does not depend on $m_0\in \mathcal{P}(I)$. Since $\mathcal{P}(I)$ is compact, $\{\mu_t^e\}_{t\geq 0}$ is tight and therefore there exist some convergent subsequences. Let $(t_i)_{i\in \N}$ and $(s_i)_{i\in \N}$ be two strictly increasing sequences of positive real numbers, and let  $\{\mu_{t_i}^e\}_{i\in \N}$ and $\{\mu_{s_i}^e\}_{i\in \N}$ (with $\mu_0^e=m_0$) be such that
\begin{equation}
\mu_{t_i}^e\Rightarrow\mu_1^{e,m_0}(\infty)
\end{equation}
\begin{equation}
\mu_{s_i}^e\Rightarrow\mu_2^{e,m_0}(\infty)
\end{equation}
in $\mathcal{P}(I)$ as $N\rightarrow \infty$, where $\mu_i^{e,m_0}(\infty),$ for $i=1,2$, are random measures in $\mathcal{P}(I)$. For any $\hat{\psi}_0\in \mathcal{C}^*$
\begin{equation}
\begin{array}{l}
\mathbb{E}[<\mu_1^{e,m_0}(\infty),\hat{\psi}_0>]=\lim\limits_{i\rightarrow \infty}\mathbb{E}^{m_0}[<\mu_{t_i}^e,\hat{\psi}_0>]=\lim\limits_{t_i\rightarrow \infty}\mathbb{E}^{\hat{\psi}_0}[\psi_{t_i}^{t_i,e}]=\\
\lim\limits_{s_i\rightarrow \infty}\mathbb{E}^{\hat{\psi}_0}[\psi_{s_i}^{s_i,e}]=\lim\limits_{i\rightarrow \infty}\mathbb{E}^{m_0}[<\mu_{s_i}^e,\hat{\psi}_0>]=\mathbb{E}[<\mu_2^{e,m_0}(\infty),\hat{\psi}_0>].
\end{array}
\end{equation}
As $\lim\limits_{t_i\rightarrow \infty}\mathbb{E}^{\hat{\psi}_0}[\psi_{t_i}^{t_i,e}]=\lim\limits_{s_i\rightarrow \infty}\mathbb{E}^{\hat{\psi}_0}[\psi_{s_i}^{s_i,e}]$ does not depend on $m_0$, so do $\mu_i^{e,m_0}(\infty)$, for $i=1,2$. Hence, there exists a random probability measure $\mu^e(\infty)\in \mathcal{P}(I)$ such that for any $m_0\in \mathcal{P}(I)$, conditioning on $\mu_0^e=m_0$,
\begin{equation}
\mu_t^e\Rightarrow \mu^e(\infty).
\end{equation}
For part $(ii)$, it is sufficient to prove that conditioning on any initial distribution of $(\mu_t^e,e_t)$, namely $\tilde{m}_0\in \mathcal{P}(\mathcal{P}(I)\times E)$,
\begin{equation}\label{limit-ergodc-little}
\lim\limits_{t\rightarrow \infty}\mathbb{E}[f_1(\mu_t)f_2(e_t)] \ \ \ f_1\in \mathcal{C}(\mathcal{P}(I)), \ f_2\in \mathcal{C}(E)
\end{equation}
exists and does not depend on $\tilde{m}_0$. In that case, since $\mathcal{P}(I)\times E$ is compact, any convergent subsequence of $(\mu_t^e,e_t)$ converges weakly to a unique limit, and from part $(i)$ and assumption, then
\begin{equation}
(\mu_t^e,e_t)\Rightarrow (\mu^e(\infty),e(\infty)).
\end{equation}
To prove the existence of the limit for any arbitrary $\hat{\psi}_0\in \mathcal{C}^*$ and $m_0\in \mathcal{P}(I)$, write
\begin{equation}
\begin{array}{l}
\mathbb{E}[<\mu_t^e,\hat{\psi}_0> f_2(e_t)]=\\
\mathbb{E}[\mathbb{E}[<\mu_t^e,\hat{\psi}_0>|e]f_2(e_t)]=\\
\mathbb{E}[\mathbb{E}[<m_0,\psi_t^{t,e}>|e]f_2(e_t)]
\end{array}
\end{equation}
But, conditioning on $\tau,J_\tau,\kappa,T_\tau^{sel}$, knowing the fact that $\tau$ is finite a.s. and does not depend on $e$, and replacing the continuous function $f$ in (\ref{166} )by $\psi_t^{t,e}e(t)$, we can see the limit of the last term in the last equality exists and does not depend on the choice of $m_0$ and $\hat{\psi}_0$. (Recall that $\{<.,f>\}_{f\in \mathcal{C}^*}$ is measure and convergence-determining.)
Now let $\nu^*$ be the distribution of $(\mu^e(\infty),e(\infty))$. Let $T_t^*$ be the semigroup of the joint annealed-environment process, $(\mu_t^e,e_t)$. For $F\in \mathcal{C}(\mathcal{P}(I)\times E)$, $\tilde{m}_0\in \mathcal{P}(\mathcal{P}(I)\times E)$, and $s\geq 0$,
\begin{equation}
\begin{array}{l}
\int\limits_{\mathcal{P}(I)\times E} T_s^*F d\nu^*=\lim\limits_{t\rightarrow \infty}\int\limits_{\mathcal{P}(I)\times E} T^*_{t+s}Fd\tilde{m}_0=\int\limits_{\mathcal{P}(I)\times E} Fd\nu^*.
\end{array}
\end{equation}
The last equation holds for any $\tilde{m}_0$, including all invariant measures. Hence the uniqueness holds.
\section*{Acknowledgement}
The author wishes to thank Prof. Donald Dawson for introducing the problem, and also for his great support and many helpful discussions during accomplishment of this work.


\begin{thebibliography}{99}
\bibitem{Bakry-Emery} Bakry, D., and \'Emery, M.: Diffusions hypercontractives. In Séminaire de Probabilités, XIX, 1983/84. Lecture Notes in Mathematics, 1123 177–206. Springer, Berlin, 1985.







\bibitem{D93} Dawson, D.A.: Measure-valued Markov processes. In Hennequin, P.L. editor, '{E}cole d'\'{e}t\'{e} de Probabilit\'{e}s de Saint-Flour XXI - 1991, volume 1541 of Lecture Notes in Mathematics, pages 1-260, Springer Berlin Heidelberg, 1993.



\bibitem{DG14} Dawson, D.A., and Greven, A.: Spatial Fleming-Viot models with selection and mutation. Lecture Notes in Mathematics, 2092. Springer, Cham, 2014.

\bibitem{DH82} Dawson, D.A., and Hochberg, K.J.: Wandering random measures in the Fleming-Viot model. The Annals of Probability, p. 554-580, 1982.

\bibitem{IFVRE-DA16} Dawson, D.A., and Jamshidpey, A.: Interacting Fleming-Viot systems in random environments, \textit{in preparation}, 2016.

\bibitem{DK82} Dawson, D.A., and Kurtz, T.G.: Applications of duality to measure-valued diffusion processes. In: Advances in Filtering and Optimal Stochastic Control. Springer Berlin Heidelberg, p. 91-105, 1982.

\bibitem{DGP12} Depperschmidt, A., Greven, A., and Pfaffelhuber, P.: Tree-valued Fleming-Viot dynamics with mutation and selection. The Annals of Applied Probability, v. 22, n. 6, p. 2560-2615, 2012.









\bibitem{etheridge-book} Etheridge, A.M.: An introduction to superprocesses. Providence, RI: American Mathematical Society, 2000.

\bibitem{EK86-book} Ethier, S.N., and Kurtz, T.G.:  Markov Processes: Characterization and Convergence. Wiley, New York, 1986.

\bibitem{EK93} Ethier, S.N., and Kurtz, T.G.: Fleming-Viot processes in population genetics. SIAM Journal on Control and Optimization. 31(2):345-86, 1993.



\bibitem{FV79} Fleming, W., and Viot, M.: Some Measure-Valued Markov Processes in Population Genetics Theory, Indiana Univ. Math. J. 28, n. 5, p. 817–843, 1979.

\bibitem{GPW} Greven, A., Pfaffelhuber, P., and Winter, A.: Tree-valued resampling dynamics Martingale problems and applications. Probability Theory and Related Fields, 155(3-4), pp.789-838, 2013.






\bibitem{Thesis} Jamshidpey, A.: Ph.D. thesis on Population dynamics in random environments, random walks on symmetric groups, and phylogeny reconstruction, University of Ottawa online PhD theses, 2016.












\bibitem{K98} Kurtz, T.: Martingale problems for conditional distributions of Markov processes. Electronic Journal of Probability,  v.3, n.9, p.1-29, 1998.

\bibitem{liggett-book} Liggett, T.: Interacting particle systems. Springer-Verlag, 2005.

\bibitem{Mustonen-Lassig09} Mustonen, V., and L\"assig, M.: From fitness landscapes to seascapes: non-equilibrium dynamics of selection and adaptation. Trends in Genetics, v. 25, n. 3, p. 111-119, 2009.

\bibitem{Mustonen-Lassig10} Mustonen, V., and L\"assig, M.: Fitness flux and ubiquity of adaptive evolution. Proceedings of the National Academy of Sciences, v.107, n.9, p.4248-4253, 2010.




















\end{thebibliography}
\end{document}